\documentclass[a4paper]{amsart}
\usepackage{amssymb}
\usepackage{amsmath}
\usepackage{amsthm}
\usepackage[dvipsnames]{xcolor}
\usepackage{braket,comment,soul}
\usepackage{hyperref}
\usepackage{pdfpages,faktor}
\usepackage{graphicx,mathabx,bbm}
\usepackage{caption}
\usepackage{subcaption}
\usepackage{mathrsfs} 
\usepackage{tikz-cd} 
\usepackage{bm}
\usepackage{enumitem}
\usepackage{mathtools}
\usepackage{pgfplots}
\usepackage{import}
\usepackage{xifthen}
\usepackage{transparent}

\pgfplotsset{compat=1.18} 
\numberwithin{equation}{section}
\theoremstyle{plain}
\newtheorem{theorem}{Theorem}[section]

\newtheorem{prop}[theorem]{Proposition}
\newtheorem{lem}[theorem]{Lemma}
\newtheorem{cor}[theorem]{Corollary}

\newtheorem*{question*}{Question}

\theoremstyle{definition}
\newtheorem{defn}[theorem]{Definition}

\newtheorem{remark}{Remark}

\theoremstyle{definition}
\newtheorem{thmx}{Theorem}

\newcommand{\R}{\mathbb{R}}

\newcommand{\C}{\mathbb{C}}

\newcommand{\N}{\mathbb{N}}
\newcommand{\D}{\mathbb{D}}

\newcommand{\cF}{\mathcal{F}}
\newcommand{\cN}{\mathcal{N}}
\newcommand{\cR}{\mathcal{R}}
\newcommand{\cS}{\mathcal{S}}

\newcommand{\cM}{\mathcal{M}}
\newcommand{\cP}{\mathcal{P}}
\newcommand{\cD}{\mathcal{D}}
\newcommand{\cH}{\mathcal{H}}
\newcommand{\cT}{\mathcal{T}}
\newcommand{\cU}{\mathcal{U}}
\newcommand{\cV}{\mathcal{V}}
\newcommand{\cW}{\mathcal{W}}

\newcommand{\crit}{\mathrm{crit}}
\renewcommand{\epsilon}{\varepsilon}
\renewcommand{\phi}{\varphi}

\DeclareMathOperator{\Int}{Int}
\DeclareMathOperator{\Ext}{Ext}

\DeclareMathOperator{\val}{val}

\usepackage[colorinlistoftodos,prependcaption,textsize=tiny]{todonotes}

\makeatletter
\DeclareFontFamily{U}{tipa}{}
\DeclareFontShape{U}{tipa}{m}{n}{<->tipa10}{}
\newcommand{\arc@char}{{\usefont{U}{tipa}{m}{n}\symbol{62}}}

\newcommand{\arc}[1]{\mathpalette\arc@arc{#1}}

\newcommand{\arc@arc}[2]{
  \sbox0{$\m@th#1#2$}
  \vbox{
    \hbox{\resizebox{\wd0}{\height}{\arc@char}}
    \nointerlineskip
    \box0
  }
}
\makeatother

\makeatletter
\def\leftrightarrowsfill@{\arrowfill@\leftrarrows\Rrelbar\lrightarrows}
\newcommand{\xleftrightarrows}[2][]{\ext@arrow 3399\leftrightarrowsfill@{#1}{#2}}
\makeatother

\date{\today}

\begin{document}

\title[Topology and singularities of quadrature domains]{On topology and singularities of\\ quadrature domains}

\begin{author}[Rashmita]{Rashmita}
\address{School of Mathematics, Tata Institute of Fundamental Research, 1 Homi Bhabha Road, Mumbai 400005, India}
\email{rashmita@math.tifr.res.in}
\end{author}

\begin{author}[S.~Mukherjee]{Sabyasachi Mukherjee}
\address{School of Mathematics, Tata Institute of Fundamental Research, 1 Homi Bhabha Road, Mumbai 400005, India}
\email{sabya@math.tifr.res.in}
\end{author}

\thanks{Both authors were  supported by  the Department of Atomic Energy, Government of India, under project no.12-R\&D-TFR-5.01-0500. SM was supported in part by SERB research project grant MTR/2022/000248 and by an endowment of the Infosys Foundation.}

\begin{abstract}
We prove a linear upper bound for the number of singular points on the boundary of a quadrature domain, improving a previously known quadratic bound due to Gustafsson \cite{Gus88}. This linear upper bound on the number of boundary double points also strengthens the bound on the connectivity (i.e., the number of complementary components) of a quadrature domain given by Lee and Makarov \cite{LM16}.
Our proofs use conformal dynamics and hyperbolic geometry arguments. Finally, we introduce a new dynamical method to construct multiply connected quadrature domains.
\end{abstract}

\maketitle

\setcounter{tocdepth}{1}
\tableofcontents

\section{Introduction}\label{intro_sec}

Quadrature domains lie at the confluence of complex analysis, statistical physics, fluid dynamics, and dynamical systems. These are domains in the Riemann sphere that admit finite-node quadrature identities for integrable analytic functions.
It turns out that they are characterized by having a semi-global Schwarz reflection map; i.e., by the property of having a real-analytic boundary such that the local Schwarz reflection map with respect to the boundary extends anti-meromorphically to the interior.

\subsection{Quadrature domains in analysis, mathematical physics, and conformal dynamics} Quadrature domains were first considered by Davis in the context of Schwarz functions \cite{Dav74}, and by Aharonov and Shapiro from the point of view of quadrature identities \cite{AS76}.
Numerous applications of quadrature domains were discovered in the following decades. They played an important role in various complex-analytic problems such as quadrature identities \cite{Dav74, AS76, Sak82, Gus83}, extremal problems for conformal mapping \cite{ASS99,LMM21,LMMN20}, Hele-Shaw flows \cite{Ric72, EV92, GV06}, moment problems \cite{Sak78,GHMP00}, dessin d'enfants \cite{ILRS23}, etc. (see \cite{EGKP05}, \cite[\S 1.1]{LLMM2} and the references therein for more connections). Quadrature domains also arise naturally in the study of equilibrium measures in many statistical physics problems and in random matrix theory \cite{ABWZ02,Wie02,TBAZW05,EF05,HM13,LM16,NW24,BY25}. More recently, the dynamics of Schwarz reflection maps associated with quadrature domains was explored in details, and it led to many deep connections between rational dynamics, actions of Kleinian reflection groups, and dynamics of algebraic correspondences \cite{LLMM3,LLMM1,LLMM2,LMMN20,LMM24,LM23}.

\subsection{Complexity of quadrature identity}\label{quad_identity_complexity_subsec}
A quadrature domain $\Omega$ admits a \emph{quadrature function} $R_\Omega$, which is a rational map with all of its poles in $\Omega$, such that the \emph{quadrature identity}
$$
\displaystyle \int_{\Omega} f dxdy=\frac{1}{2i} \oint_{\partial\Omega} f(z) R_{\Omega}(z) dz
$$ 
is satisfied for all analytic functions $f$ on $\Omega$ extending continuously to $\overline{\Omega}$ (for $\Omega$ unbounded, the functions $f$ are also required to vanish at $\infty$).

The degree $d_\Omega$ of $R_\Omega$ is called the \emph{order} of the quadrature domain. Further, the number of poles of $R_\Omega$ (also called the \emph{nodes} of $\Omega)$ is denoted as $n_\Omega$. These numbers can be regarded as a measure of complexity of the quadrature identity satisfied by~$\Omega$. In other words, the analytic complexity of $\Omega$ is encoded in $d_\Omega, n_\Omega$.

\subsection{Real-algebraicity of Schwarz reflections and boundaries of quadrature domains}\label{algebraicity_subsec}
As mentioned above, a quadrature domain $\Omega$ admits an anti-meromorphic map $\sigma:\overline{\Omega}\to\widehat{\C}$ that is identity on $\partial\Omega$. Using a Schottky double construction, Gustafsson proved that the Schwarz reflection $\sigma$ is a real-algebraic function \cite{Gus83}. Specifically, there exists a symmetric compact Riemann surface $\Sigma$ (the `double' of $\Omega$), an anti-conformal involution $\eta$ on $\Sigma$, and a meromorphic map $f:\Sigma\to\widehat{\C}$ such that $\sigma$ admits a simple description in terms of $f$ and $\eta$ .
Denoting the degree of $f$ by $d_f$, one easily obtains a linear relation between the numbers $d_f$ and $d_\Omega$ (see Section~\ref{schwarz_alg_prop} for details). As an upshot of this construction, it was also demonstrated in \cite{Gus83} that $\partial\Omega$ is a real-algebraic curve and hence has a finite number of cusp and double point singularities.

\subsection{Topological complexity of quadrature domains}\label{main_thm_1_subsec}
In many of the analytic, physical, and dynamical applications mentioned above, the topology and geometry of quadrature domains play an important role. In a groundbreaking work, Lee and Makarov showed that the topological complexity of a quadrature domain $\Omega$ can be bounded by the complexity of its quadrature identity \cite{LM16}.
More precisely, they proved that the integers $d_\Omega, n_\Omega$ provide a sharp linear upper bound on the \emph{connectivity} of $\Omega$. Here, the connectivity of $\Omega$, denoted by $\mathrm{conn}(\Omega)$, is the number of components of the \emph{droplet} $\Omega^\complement=\widehat{\C}\setminus\Omega$.

The main result of this paper improves the upper bounds furnished in \cite{LM16}. We show that the same numbers give an upper bound for the sum of the connectivity of $\Omega$, the number of double points on $\partial\Omega$, and certain multiplicities associated with the singular points on $\partial\Omega$.
We make a standing non-degeneracy assumption that $\partial\Omega$ neither contains an isolated point nor contains any critical value of $\sigma$.

\begin{thmx}\label{main_thm_1}
Let $\Omega$ be a quadrature domain with $d_f\geq 3$. We denote the set of all double points (respectively, cusps) on $\partial\Omega$ by $D$ (respectively, $C$). Then,
$$
\mathrm{conn}(\Omega) + \# D + \sum_{p\in D} \delta_p + \sum_{p\in C} \delta_p \leq \min \{d_f+n_\Omega-2, 2d_f-4\}.
$$
Moreover, if $\Omega$ has a node at $\infty$, then
$$
\mathrm{conn}(\Omega) + \# D + \sum_{p\in D} \delta_p + \sum_{p\in C} \delta_p \leq d_f +n_\Omega - 3.
$$
Here, $\delta_p=\lfloor n/4\rfloor$ if $p$ is a cusp of type $(n,2)$, and $\delta_p=\lfloor n/2\rfloor$ if $p$ is a double point with order of contact $n$.
\end{thmx}

Theorem~\ref{main_thm_1} bounds the topological and geometric complexity of a quadrature domain $\Omega$ in terms of its analytic complexity. For quadrature domains with non-singular boundary, Theorem~\ref{main_thm_1} recovers the main results of \cite{LM16}.

\subsection{Linear upper bound on the number of singular points}\label{sing_linear_bound_subsec}

In \cite{Gus88}, algebro-geometric considerations involving the Schottky double of a quadrature domain were used to give an upper bound on the total number of (weighted) singular points on the boundary. The upper bound given in \cite{Gus88} is \emph{quadratic} in the order $d_\Omega$ of the quadrature domain. Theorem~\ref{main_thm_1}, combined with the Riemann-Hurwitz formula, allows us to improve the bound to a linear one.

\begin{thmx}\label{main_thm_2}
   Let $\Omega$ be a quadrature domain with $d_f\geq 3$. We denote the set of all double points (respectively, cusps) on $\partial\Omega$ by $D$ (respectively, $C$). Then,
   $$
   \# C + 2 \# D + 3\left(\sum_{p\in D} \delta_p + \sum_{p\in C} \delta_p\right) \leq \min \{3d_f+3n_\Omega-6, 6d_f-12\}.
   $$
   Moreover, if $\Omega$ has a node at $\infty$, then
$$
\# C + 2 \# D + 3\left(\sum_{p\in D} \delta_p + \sum_{p\in C} \delta_p\right) \leq  \min \{ 3d_f+3n_\Omega-8, 4d_f+2n_\Omega-10\}.
$$
\end{thmx}

\subsection{Comments on the proof of Theorem~\ref{main_thm_1}}\label{proof_comments_subsec}

While the main idea in the proof of the upper bounds given in \cite{LM16} is a combination of a quasiconformal surgery argument and Hele-Shaw flow of algebraic droplets, our proof relies only on classical holomorphic dynamics and planar hyperbolic geometry methods. However, both proofs depend crucially on iteration of Schwarz reflection maps.
We now highlight the main differences in the proofs.

\subsection*{The non-singular case}
In \cite{LM16}, the upper bounds on connectivity are first established for quadrature domains $\Omega$ with non-singular boundary. The non-singularity assumption allows one to replace the action of the Schwarz reflection $\sigma$ over the droplet with appropriate attracting dynamics modeled on rational maps. The construction of this modified holomorphic dynamical system uses quasiconformal surgery tools, which in general do not apply in the presence of singularities. The desired upper bound is then deduced using a classical result of Fatou that guarantees the existence of a critical point in each attracting basin.

Our proof, in the non-singular case, exploits the dynamics of $\sigma$ over the droplet. Instead of modifying the map $\sigma$, we study its \emph{escaping dynamics}, and apply a modulus argument to locate sufficiently many critical points of $\sigma$ escaping to the droplet. In particular, our proof does not use quasiconformal methods.

\subsection*{The singular case}
To handle quadrature domains with singular boundary, Lee and Makarov resorted to the theory of Hele-Shaw flows of droplets \cite{LM16}. They showed that subjecting a singular droplet to backward Hele-Shaw flow (so that the droplet `shrinks' in the process) resolves the singularities without reducing the connectivity of the quadrature domain. Thus, such a flow brings one back to the setting of non-singular quadrature domains whence the upper bound established in the non-singular case applies.

On the other hand, we employ the idea of studying the escaping dynamics of $\sigma$ to deal with singular quadrature domains as well. The singular situation requires more care, and instead of reducing to the non-singular case, we carry out a detailed analysis of the effect of cusps and double points on the dynamics of $\sigma$. Using  techniques from dynamics on hyperbolic Riemann surfaces, we find critical points of $\sigma$ escaping to singular droplets and critical values of the meromorphic function $f$ (see Section~\ref{algebraicity_subsec}) on boundaries of droplets. A good understanding of the connection between double points and the global dynamics of $\sigma$ allows us to bound the number of double points in Theorem~\ref{main_thm_1}.

\subsection*{Orders of singular points}
The local dynamics of $\sigma$ near a cusp or a double point on $\partial\Omega$ reflect the algebro-geometric quantities called \emph{orders} of cusps and \emph{orders of contact} of double points. We carry out a local analysis of the dynamics of $\sigma$ near singular points, and show that a (possibly empty) collection of critical points of $\sigma$ converge to the singular points under iteration. This gives rise to the additional terms on the left hand side of the inequalities appearing in Theorem~\ref{main_thm_1}.

\subsection*{Consequences of a non-perturbative proof}
Since we work directly with the Schwarz reflection dynamics of the quadrature domain $\Omega$ (without performing quasiconformal modifications or Hele-Shaw perturbations), our proof gives precise information on the relation between the critical points of the meromorphic function $f$ and the components of the droplet $\Omega^\complement$. This also sheds light on the topological/dynamical structure of the \emph{escaping set} of $\sigma$ corresponding to individual components of the droplet.

It is worth pointing out that the Hele-Shaw flow arguments of \cite{LM16} can plausibly be applied to get a linear upper bound on the sum of the connectivity of $\Omega$ and the number of double points on $\partial\Omega$. But such an argument, being perturbative in nature, will not directly yield an exact association between the critical points of $f$ and the components of the droplet.

\subsection{Droplets and Coulomb gas ensembles}\label{stat_phys_subsec}
The complement of a quadrature domain $\Omega$ arises as the support of the equilibrium measure of a `Coulomb gas ensemble'.
According to \cite[\S 2,3]{LM16} (cf. \cite{HM13}), the complement of a quadrature domain is the support of an absolutely continuous measure that minimizes the \emph{combined energy}
$$
I_Q[\mu]:= \int_{\C\times\C} \ln\vert z-w\vert^{-1} d\mu(z) d\mu(w) + \int_{\C} Q d\mu
$$
over all compactly supported probability measures $\mu$. Here, the first term in the summand is the two-dimensional Coulomb energy due to the charge distribution $\mu$ and the second term is the energy of interaction with the external field $Q$, which is of the form
$$
Q(z)=\vert z\vert^2-H(z),
$$
where $H$ is harmonic with $h:=\partial H$ a rational function. Further, $h$ is precisely the quadrature function $R_\Omega$ of $\Omega$. Thus, the order of the quadrature domain $\Omega$ also measures the complexity of the external field $Q$. In view of the above discussion, Theorems~\ref{main_thm_1} and ~\ref{main_thm_2} have the following physical interpretation: 
\emph{the topological and geometric complexities of the quadrature domain $\Omega$ are bounded by linear functions of the complexity of the associated external field $Q$.}

\subsection{Topological configurations of extremal droplets}\label{droplet_topology_subsec}
According to \cite{LM14}, the upper bound of Theorem~\ref{main_thm_1} is sharp. In a sequel to this paper, we will prove a stronger version of this sharpness statement; namely, given any $m\geq 1$ and $n\geq 0$ satisfying $m+n= 2d_f-4$, there exists a quadrature domain $\Omega$ of order $d_\Omega=d_f$, connectivity $m$, and having exactly $n$ boundary double points. There, we will also investigate all possible topological configurations of such \emph{extremal} quadrature domains/droplets. We refer the reader to \cite{LMM21}, \cite[\S 12]{LMMN20} for a complete description of possible topological configurations of a special class of extremal singular droplets arising from \emph{Suffridge polynomials}, and to \cite[\S 2.4]{LM16} for a classification of topological configurations of non-singular droplets.

In this paper, we will explicate the construction of all possible topological configurations of extremal multiply connected quadrature domains for $d_f=3,4$. We will also outline an alternative construction of non-singular quadrature domains of maximal connectivity $2d_f-4$. Our construction of multiply connected quadrature domains is fundamentally different from the existing methods in the literature (cf. \cite{Bel04,CM04,LM14}), and is based on some recently developed surgery techniques in conformal dynamics (cf. \cite{LMMN20}).

\subsection{Organization of the paper}\label{organize_subsec}
The paper is organized as follows. In Section~\ref{prelim_sec}, we recall several background results on quadrature domains and associated Schwarz reflection maps. We recall the algebraic description of a Schwarz reflection map given in terms of a meromorphic map (called the \emph{uniformizing meromorphic map} of a quadrature domain) and an anti-conformal involution defined on a symmetric compact Riemann surface. We then describe the singular points on boundaries of quadrature domains and critical points of the Schwarz reflection maps in terms of the uniformizing meromorphic maps. We also recall the fundamental invariant partition of the dynamical plane of a Schwarz reflection map into \emph{escaping} and \emph{non-escaping sets}. Section~\ref{cusp_dp_dyn_sec} contains a basic account of the local dynamics of a Schwarz reflection map near the singularities on the quadrature domain boundary. This allows us to relate the singular points to certain critical points of the Schwarz reflection map (equivalently, to certain critical points of the 
uniformizing meromorphic map).

Section~\ref{special_case_sec} is the technical heart of the paper. Here we prove a weaker version of Theorem~\ref{main_thm_1} by associating critical points of the Schwarz reflection map (equivalently, of the uniformizing meromorphic map) to each complementary component of a quadrature domain. The proof uses conformal dynamics and hyperbolic geometry techniques, and requires separate analysis of the cases where the quadrature domain has (i) non-singular boundary, (ii) singular boundary without double points, and (iii) singular boundary with double points.

In Section~\ref{pf_main_thm_sec}, we  complete the proofs of our main theorems by combining the results of Sections~\ref{cusp_dp_dyn_sec} and~\ref{special_case_sec} with the information of critical points of the uniformizing meromorphic map associated with the nodes of the quadrature domain (prepared in Section~\ref{prelim_sec}).

The final Section~\ref{sharpness_sec} shows, by means of worked out examples, that the upper bound of Theorem~\ref{main_thm_1} is sharp in an effective manner. This paves the way of future investigation of topological shapes of extremal quadrature domains of arbitrary order.

\noindent\textbf{Notation.}
\begin{itemize}
    \item For a set $X\subset\widehat{\C}$, we denote the interior (respectively, exterior) of $X$ by $\Int{X}$ (respectively, $\Ext{X})$.
    \item An open (respectively, closed) ball with center at $x$ and radius $r>0$ will be denoted by $B(x,r)$ (respectively, $\overline{B}(x,r)$).
\end{itemize}

\section{Preliminaries}\label{prelim_sec}

\subsection{Quadrature domains and Schwarz reflection maps}\label{qd_schwarz_subsec}
Throughout this section, we let $\Omega\subsetneq\widehat{\C}$ be a domain such that $\infty\notin\partial\Omega$ and $\Int{\overline{\Omega}}=\Omega$. 
We will denote the complex conjugation map by $\iota$.

\begin{defn}\label{quad_domain_def}
A domain $\Omega$ is called a \emph{quadrature domain} if there exists a continuous function $\sigma:\overline{\Omega}\to\widehat{\C}$ satisfying the following two properties:
\begin{enumerate}\upshape
\item $\sigma=\mathrm{id}$ on $\partial \Omega$.

\item $\sigma$ is anti-meromorphic on $\Omega$.
\end{enumerate}

\noindent The map $\sigma$ is called the \emph{Schwarz reflection map} of $\Omega$. Its complex conjugate $\iota\circ\sigma$ is called the \emph{Schwarz function} of $\Omega$.
\end{defn}

By \cite{Sak91}, except for a finite number of singular points, which are necessarily cusps and double points, the boundary of a quadrature domain consists of finitely many disjoint non-singular real-analytic curves. 
Thus, for a quadrature domain $\Omega$, the map $\sigma$ is the anti-meromorphic extension of the Schwarz reflection map with respect to $\partial \Omega$ (the reflection map fixes $\partial\Omega$ pointwise).

\begin{defn}[Quadrature functions]\label{quad_func_def}
Let $\Omega\subsetneq\widehat{\C}$ be a domain with $\infty\notin\partial\Omega$ and $\Int{\overline{\Omega}}=\Omega$. 
Functions in $H(\Omega)\cap C(\overline{\Omega})$ are called \emph{test functions} for $\Omega$ (if $\Omega$ is unbounded, we further require test functions to vanish at $\infty$). A rational map $R_\Omega$ is called a \emph{quadrature function} of $\Omega$ if all poles of $R_\Omega$ are inside $\Omega$ (with $R_\Omega(\infty)=0$ if $\Omega$ is bounded), and the identity 
\begin{equation}
\displaystyle \int_{\Omega} f dA=\frac{1}{2i} \oint_{\partial\Omega} f(z) R_{\Omega}(z) dz
\label{quad_identity_1_eqn}
\end{equation}
holds for every test function $f$ for $\Omega$.
\end{defn}

For a bounded domain $\Omega$, Relation~\eqref{quad_identity_1_eqn} can be rewritten as
\begin{equation}
\displaystyle \int_{\Omega} f dxdy= \sum_{r=1}^n \sum_{s=0}^{m_r-1} c_{r_s} f^{(s)}(z_r),
\label{quad_identity_2_eqn}
\end{equation}
where the points $z_1,\cdots,z_n\in\Omega$ are the poles of $R_\Omega$; $m_1,\cdots,m_n$ are the orders of these poles; and $c_{r_s}$, $r\in\{1,\cdots,n\}, s\in\{0,\cdots,m_r-1\}$, are complex numbers. Relations of the form~\eqref{quad_identity_2_eqn} are called \emph{quadrature identities}.

The following theorem is classical, and a proof can be found in \cite[Lemma~2.3]{AS76}, \cite[Lemma~3.1]{LM16}.

\begin{theorem}[Characterization of quadrature domains]\label{characterization}
The following are equivalent.
\begin{enumerate}
\item $\Omega$ is a quadrature domain; i.e. it admits a Schwarz reflection map.

\item $\Omega$ admits a quadrature function $R_\Omega$.

\item The Cauchy transform $\widehat{\chi}_{\Omega}$ of the characteristic function $\chi_{\Omega}$ (of $\Omega$) is rational outside $\Omega$. 
\end{enumerate}
\end{theorem}

We call $d_\Omega:=\deg (R_\Omega)$ the \emph{order} of the quadrature domain $\Omega$. The poles of $R_\Omega$ are called \emph{nodes} of $\Omega$. The number of distinct nodes of $\Omega$ is denoted by $n_\Omega$.

\subsection{Algebraicity and mapping degrees of Schwarz reflections}\label{schwarz_algebraic_subsec}
Let $\Omega$ be a quadrature domain of connectivity $k$; i.e.,
$$
\mathrm{conn}(\Omega):= \#\ \left(\mathrm{connected\ components\ of}\ \Omega^\complement\right)  = k,
$$
where $\Omega^\complement:=\widehat{\C}\setminus\Omega$.
Denote the Schwarz reflection map of $\Omega$ by $\sigma$.

The following result, which is implicitly proved in \cite{Gus83} (cf. \cite[Lemma~4.1]{LM16}), gives an algebraic description of Schwarz reflections.
We supply a proof for completeness.
\begin{prop}\label{schwarz_alg_prop}
There exist 
\begin{enumerate}
\item a (symmetric) compact Riemann surface $\Sigma$ of genus $k-1$, 
\item an antiholomorphic involution $\eta:\Sigma\to\Sigma$ whose fixed-point set $P$ is a disjoint union of $k$ simple, closed, non-singular real-analytic curves, and 
\item a meromorphic function $f:\Sigma\to\widehat{\C}$,
\end{enumerate} 
such that
\begin{enumerate}
\item[(i)] $\Sigma\setminus P$ has two connected components $W^\pm$,
\item[(ii)] $f:W^+\to\Omega$ is a conformal isomorphism, and
\item[(iii)] $\sigma\equiv f\circ \eta\circ (f\vert_{\overline{W^+}})^{-1}$ on $\overline{\Omega}$.
\end{enumerate}
\end{prop}
\begin{proof}
Given a quadrature domain $\Omega$ of connectivity $k$, by the Koebe uniformization theorem, there exist a circle domain $\cD$ of the same connectivity (i.e., a domain $\cD\subset\widehat{\C}$ whose boundary consists of $k$ disjoint round circles) and a conformal isomorphism $\varphi:\cD\to\Omega$. Since $\partial\Omega$ is real-analytic, the map $\varphi$ extends to a continuous surjection $\varphi:\overline{\cD}\to\overline{\Omega}$. Further, $\varphi:\partial\cD\to\partial\Omega$ semi-conjugates $\mathrm{Id}\vert_{\partial\cD}$ to $\sigma\vert_{\partial\Omega}$.
\begin{figure}[ht!]
\captionsetup{width=0.98\linewidth}
\includegraphics[width=0.96\linewidth]{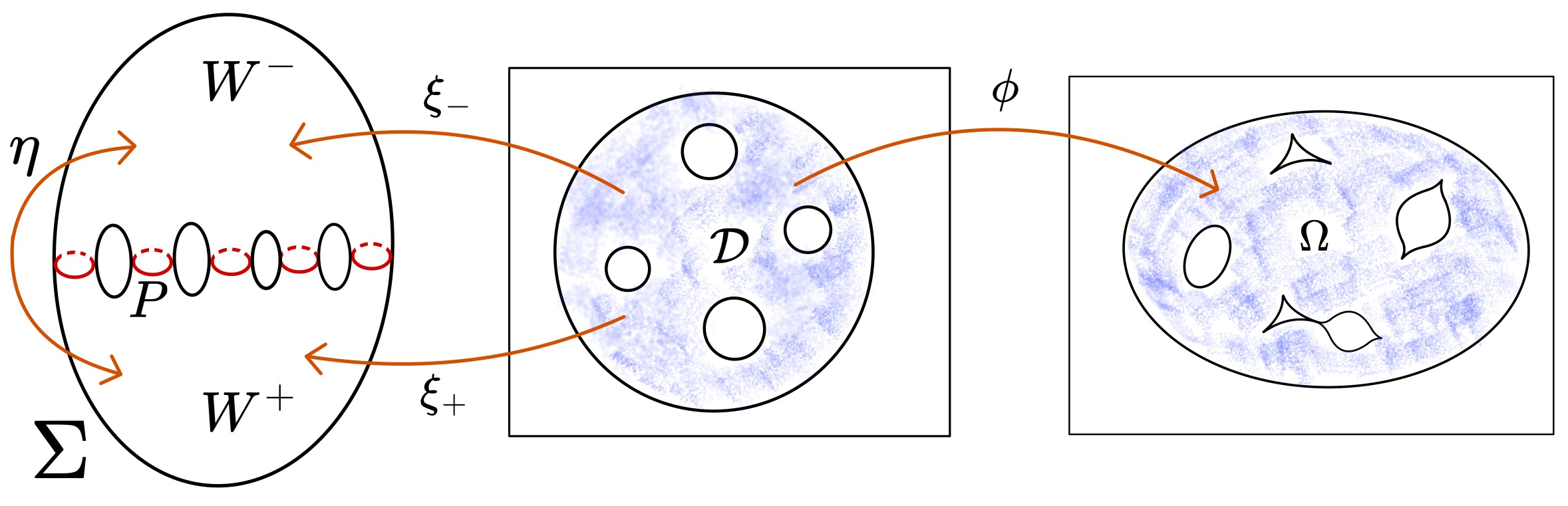}
\caption{Illustrated is the construction of the Riemann surface double $\Sigma$ of the quadrature domain $\Omega$ and the meromorphic map $f$ (cf. Proposition~\ref{schwarz_alg_prop}).}
\label{schottky_double_fig}
\end{figure}

We now construct a compact Riemann surface $\Sigma$ of genus $k-1$ by `doubling' $\cD$; i.e., $\Sigma$ is obtained by welding a copy of $\cD$ carrying the usual complex structure with another copy of $\cD$ carrying the opposite complex structure along the boundary $\partial\cD$, where the identification is given by the identity map $\mathrm{Id}:\partial\cD\to\partial\cD$ (see \cite[Chapter~II, \S 3]{AS60} for a detailed account of the Riemann surface double construction). 
The welding construction gives topological embeddings $\xi_\pm:\overline{\cD}\hookrightarrow\Sigma$ such that
\begin{enumerate}
\item $\xi_+$ is conformal on $\cD$, while $\xi_-$ is anti-conformal on $\cD$;
\item $\xi_+(\cD)\cap\xi_-(\cD)=\emptyset$;
\item $\xi_+(\partial\cD)=\xi_-(\partial\cD)$ is a disjoint union of $k$ non-singular, real-analytic, closed curves;
\item $\xi_+(w)=\xi_-(w)$, for $w\in\partial\cD$; and
\item $\Sigma=\xi_+(\cD\cup\partial\cD)\cup\xi_-(\cD)$.
\end{enumerate}
We set $W^\pm:=\xi_\pm(\cD)$ and $P:=\xi_\pm(\partial\cD)=\partial W^\pm$. Clearly, the map 
$$
\eta:\Sigma\to\Sigma,\quad 
\eta:= \begin{cases}
    \xi_-\circ\xi_+^{-1},\quad \mathrm{on}\ \overline{W^+},\\
    \xi_+\circ\xi_-^{-1},\quad \mathrm{on}\ \overline{W^-},    
\end{cases}
$$
is an anti-conformal involution whose fixed point set is $P$.

We now define a meromorphic map
\begin{align*}
& \hspace{4cm} f:\Sigma\to\widehat{\C}\\ 
& f:= \begin{cases}
    \varphi\circ\xi_+^{-1},\quad \mathrm{on}\ \overline{W^+},\\
    \sigma\circ\varphi\circ\xi_-^{-1}=\sigma\circ\varphi\circ(\xi_+^{-1}\circ\eta)=\sigma\circ\left(f\vert_{\overline{W^+}}\right)\circ\eta,\quad \mathrm{on}\ \overline{W^-}.    
\end{cases}
\end{align*}
Evidently, $f:W^+\to\Omega$ is a conformal isomorphism, and $\sigma\equiv f\circ \eta\circ (f\vert_{\overline{W^+}})^{-1}$ on~$\overline{\Omega}$.
\end{proof}

We will refer to the meromorphic function $f$ associated with a quadrature domain $\Omega$ (constructed in Proposition~\ref{schwarz_alg_prop}) the \emph{uniformizing meromorphic map} of $\Omega$. The degree $d_f$ of the meromorphic function $f:\Sigma\to\widehat{\C}$ is related to the mapping degrees of the Schwarz reflection map $\sigma$. 

\begin{prop}\label{schwarz_deg_prop}
\begin{align*}
d_f:=\deg\left(f:\Sigma\to\widehat{\C}\right) = \deg\left(\sigma:\sigma^{-1}(\Int{\Omega^\complement})\to\Int{\Omega^\complement}\right)\\
 = 1+\deg\left(\sigma:\sigma^{-1}(\Omega)\to\Omega\right).
\end{align*}
Moreover,
$$
d_f=\begin{cases}
d_\Omega\hspace{1.6cm} \mathrm{if}\ \Omega\ \mathrm{is\ bounded},\\
1+d_\Omega\hspace{0.95cm} \mathrm{if}\ \Omega\ \mathrm{is\ unbounded}.
\end{cases}
$$
\end{prop}
\begin{proof}
The first part follows from the relation $\sigma\vert_{\Omega}\equiv f\circ \eta\circ (f\vert_{W^+})^{-1}$ and the fact that $f$ maps $W^+$ injectively onto $\Omega$ (see Proposition~\ref{schwarz_alg_prop}).

By \cite[Lemma~3.1]{LM16}, the Schwarz reflection $\sigma$ and the quadrature function $R_\Omega$ have the same poles. Thus, $\infty$ has $d_\Omega$ preimages under $\sigma$ counted with multiplicity. By our standing assumption, $\infty\notin\partial\Omega$. Hence, $\Omega$ is bounded (respectively, unbounded) if $\infty\in\Omega^\complement$ (respectively, $\infty\in\Omega$). It now follows from the above facts and the first part of the proposition that 
$$
d_f=\deg\left(\sigma:\sigma^{-1}(\Int{\Omega^\complement})\to\Int{\Omega^\complement}\right)=d_\Omega\quad \mathrm{if\ \Omega\ is\ bounded}
$$
and
$$
d_f-1=\deg\left(\sigma:\sigma^{-1}(\Omega)\to\Omega\right)=d_\Omega\quad \mathrm{if\ \Omega\ is\ unbounded}.
$$
\end{proof}

\subsection{Relation between singularities of quadrature domains and uniformizing meromorphic maps}\label{singularity_subsec}

A point $p\in\partial\Omega$ is called \emph{regular} if there is a disc $B(p,\epsilon)$ such that $\Omega\cap B(p,\epsilon)$ is a Jordan domain and $\partial\Omega\cap B(p,\epsilon)$ is a simple non-singular real-analytic arc. A point $p\in\partial\Omega$ is called \emph{singular} if it is not a regular point (cf. \cite[\S 2.2]{LMM21}).

By \cite{Sak91} (cf. \cite{Gus83}), any singular point on the boundary of a quadrature domain $\Omega$ is a cusp or a double point. In fact, these singularities can be read from the uniformizing meromorphic map $f$ associated with $\Omega$.

We will use the notation of Proposition~\ref{schwarz_alg_prop}. Let $K$ be a component of $\Omega^\complement$ and $P_K$ be the component of $P$ that is mapped onto $\partial K$ by $f$. Further, let $\beta\in\partial K$ be a singular point of $\partial K$. 
\smallskip

\noindent\textbf{Case I: Cusps.} Suppose that $\beta$ is not a cut-point of $\partial K$.  Then, there exists a unique point $w\in P_k$ with $f(w)=\beta$. The facts that $P_K$ is a non-singular real-analytic curve, $f$ is analytic, and $\beta$ is a singular point of $\partial K$, together imply that $w$ is a critical point of $f$. Further, the conformality of $f$ on $W^+$ implies that $w$ is a simple critical point of $f$. Thus, the point $\beta$ is a cusp singularity of $\partial K$ and the curve $\partial K$ points towards $\Omega$ at $\beta$. 
In particular, each cusp on $\partial\Omega$ is a critical value of $f$ with an associated critical point on~$P$. 
\smallskip

\noindent\textbf{Case II: Double points.}
Now let $\beta\in\partial K$ be a cut-point of $\partial K$. Then, there are distinct points $w_1,\cdots, w_r\in P_K$, $r\geq 2$, such that $f(w_i)=\beta$, $i\in\{1,\cdots,r\}$. Since $f:W^+\to\Omega$ is a conformal isomorphism and $P_K$ is a non-singular curve, it follows that $r=2$ and neither $w_1$, nor $w_2$, is a critical point of $f$. Thus, the pieces of $P_K$ near $w_1, w_2$ are mapped conformally by $f$, and hence for $\epsilon>0$ small enough, $\partial K\cap B(\beta,\epsilon)$ is the union of two non-singular real-analytic arcs meeting tangentially at $\beta$. Hence, $\beta$ is a double point singularity of $\partial K$.

\subsection{Critical points of Schwarz reflections and uniformizing meromorphic maps}\label{critical_subsec}
We denote the critical points of $f$ and $\sigma$ by $\mathrm{crit}(f)$ and $\mathrm{crit}(\sigma)$, respectively.
The following elementary fact will play a fundamental role in the proof of our main results. We recall the notation $k=\mathrm{conn}(\Omega)$.
\begin{prop}\label{schwarz_crit_pnt_prop} 
We have
$$
\#_m\ \mathrm{crit}(f)=2d_f+2k-4,
$$
where $\#_m$ means counted with multiplicity.
In particular, $\sigma$ has at most $2d_f+2k-4$ critical points in $\Omega$, counted with multiplicity.
\end{prop}
\begin{proof}
Recall that $\Sigma$ has genus $k-1$, and $f:\Sigma\to\widehat{\C}$ has degree $d_f$. The result is a simple consequence of the Riemann-Hurwitz formula.
\end{proof}

By the description of $\sigma$ given in Proposition~\ref{schwarz_alg_prop}, we have that
$$
\mathrm{crit}(\sigma)= f(\eta(\mathrm{crit(f)\cap W^-})).
$$
Further, the critical values of $\sigma$ are also critical values of~$f$. 

As mentioned in Section~\ref{singularity_subsec}, a critical point of $f$ on $P$ creates a cusp on $\partial\Omega$. Hence, points of $\mathrm{crit}(f)\cap P$ do not give rise to critical points of the Schwarz reflection $\sigma$.
Consequently, a cusp of $\partial\Omega$ is a critical value of $f$, but in general not a critical value of $\sigma$. In fact, a cusp $\beta\in\partial\Omega$ is a critical value of $\sigma$ if and only if there exists a critical point of $f$ in $f^{-1}(\beta)\cap W^-$; i.e., when there are at least two distinct critical points of $f$ in the fiber $f^{-1}(\beta)$: one in $P$ and another in $W^-$.

\subsubsection{Critical points in the pole set}\label{pole_crit_points_subsec}
Let us record the connection between critical points of $f$ (respectively, $\sigma$) and the poles of $f$ (respectively, $\sigma$) for future reference.

We first consider the case of an unbounded quadrature domain $\Omega$. As $\infty\in\Omega$ and $\sigma\colon\sigma^{-1}(\Omega)\xrightarrow{d_f-1:1}\Omega$ is a branched covering, we have that $\sigma$ has $d_f-1$ many poles (counted with multiplicity). Consider the set $P_\sigma$ consisting of the poles of $\sigma$. If $\sigma$ has a pole $\omega$ of order $m_\omega>1$, then it is a critical point of $\sigma$ of multiplicity $(m_\omega-1)$. So the total number of critical points of $\sigma$ (counted with multiplicity) in the pole set $P_\sigma$ is equal to
$$
\sum_{\omega \in P_\sigma}(m_\omega-1) = (d_f-1)-\# P_\sigma = (d_f-1)-n_\Omega.
$$
Hence, in the unbounded case, the number of critical points of $f$ (counted with multiplicity) in $f^{-1}(\infty)=\eta((f\vert_{W^+})^{-1}(P_\sigma))$ is also equal to $d_f-n_\Omega-1$.

Now suppose that $\Omega$ is a bounded quadrature domain. In this case, $\infty\in\Int{\Omega^\complement}$ and $\sigma\colon\sigma^{-1}(\Int{\Omega^\complement})\xrightarrow{d_f:1}\Int{\Omega^\complement}$ is a branched covering, so $\sigma$ has $d_f$ many poles (counted with multiplicity). As before, let $P_\sigma$ be the pole set of $\sigma$. Since $\# P_\sigma = n_\Omega$, it follows that the total number of critical points of $\sigma$ (counted with multiplicity) in the pole set $P_\sigma$ is $d_f-n_\Omega$.
Hence, in the bounded case, the number of critical points of $f$ (counted with multiplicity) in $f^{-1}(\infty)=\eta((f\vert_{W^+})^{-1}(P_\sigma))$ is also equal to $d_f-n_\Omega$.

\subsection{Dynamics of Schwarz reflections}\label{schwarz_dyn_basic_subsec}

For a quadrature domain $\Omega$, we denote the set of singular points (cusps and double points) on $\partial\Omega$ by $\mathcal{S}$, and set 
$$
T(\sigma):=\widehat{\C}\setminus\Omega,\ \mathrm{and}\ T^0(\sigma):=T(\sigma)\setminus\mathcal{S}.
$$ 
We call the set $T(\sigma)$ the \emph{droplet} of $\sigma$, and the set $T^0(\sigma)$ (obtained by removing the singular points from the droplet boundary) the \emph{desingularized droplet/fundamental tile} of $\sigma$. We remark that $T^0(\sigma)$ is neither open, nor closed. 
  
The \emph{escaping/tiling set} of $\sigma$ is
$$
T^\infty(\sigma):=\displaystyle\bigcup_{n=0}^\infty \sigma^{-n}(T^0(\sigma)).
$$ 
It is the set of all points that eventually land in the fundamental tile $T^0(\sigma)$. Connected components of $\sigma^{-n}(T^0(\sigma))$ are called \emph{tiles of rank $n$}. When $T^0(\sigma)$ is disconnected, the tiling set admits a further partition based on which component of $T^0(\sigma)$ a point escapes to:
$$
T^\infty(\sigma)=\bigsqcup_{\substack{\mathrm{components\ K_0}\\ \mathrm{of\  T^0(\sigma)}}} \widehat{T_{K_0}^\infty}(\sigma),\quad \widehat{T_{K_0}^\infty}(\sigma):=\displaystyle\bigcup_{n=0}^\infty \sigma^{-n}(K_0).
$$
We will denote the component of $\widehat{T_{K_0}^\infty}(\sigma)$ containing $K_0$ by $T_{K_0}^\infty(\sigma)$.

The \emph{non-escaping set} of $\sigma$ is defined as $\mathscr{K}(\sigma):=\widehat{\C}\setminus T^\infty(\sigma)$. The escaping and non-escaping sets provide a $\sigma$-invariant partition of $\widehat{\C}$.

\begin{prop}
    The tiling set $T^\infty(\sigma)$ is open and the non-escaping set $\mathscr{K}(\sigma)$ is compact.
\end{prop}
\begin{proof}
    The proof of \cite[Proposition~2.3]{LMM24} applies verbatim to the current setting.
\end{proof}

\section{Dynamics near cusps and double points}\label{cusp_dp_dyn_sec}

In this section, we will expound the local dynamics of a Schwarz reflection map $\sigma:\overline{\Omega}\to\widehat{\C}$ near the set $\mathcal{S}$ consisting of the singular points of $\partial\Omega$.

\subsection{Cusp dynamics and critical points}

Let $p$ be a cusp of $\partial \Omega$ of type $(n,2)$, for some odd integer $n\geq 3$; i.e., there is a local diffeomorphic change of coordinates that brings $\partial\Omega$ (near $p$) to the cuspidal curve $y^2=x^n$ (near $0$). By \cite[Proposition~A.3]{LMM24}, we have the following asymptotic development of $\sigma^{\circ 2}$ near $p$:
\begin{equation}
\sigma^{\circ 2}(z)= z+ c\cdot (z-p)^{n/2}+o(\vert z-p\vert^{n/2}),
\label{cusp_expansion_eqn}
\end{equation}
where $c\neq 0$ (for a suitable choice of the branch of square root). By \cite[Proposition~A.3]{LMM24}, $\sigma^{\circ 2}$ has $n-2$ invariant directions in $\Omega$, and these directions are attracting and repelling in an alternating manner. Standard arguments from the local fixed point theory of holomorphic parabolic germs now show that there exist $n-2$ `attracting and repelling petals' of $\sigma^{\circ 2}$ at $p$, one for each attracting/repelling direction. We collect some basic properties of attracting/repelling petals below. For the proofs of analogous statements in the classical setting of parabolic germs, we refer the reader to \cite[\S 10]{Mil06}. For the analysis of local dynamics of a Schwarz reflection map at a cusp (i.e., the local fixed theory of a tangent-to-identity Puiseux series), see \cite[\S 4.2.1]{LLMM1}, \cite[Theorem~5.4]{LLMM3}, \cite[Appendix~A]{LMM24}.

\subsubsection*{Attracting petals}
For each attracting direction $\vec{v}$ at $p$, there is an open set $\cP$ containing the attracting direction $\vec{v}$ (called an \emph{attracting petal} associated with $\vec{v}$) such that $p\in\partial\cP$, the map $\sigma^{\circ 2}$ is injective on $\cP$, $\sigma^{\circ 2}(\cP)\subset\cP$, and the $\sigma^{\circ 2}$-orbits of all points in $\cP$ converge to $p$ asymptotic to $\vec{v}$. Further, $\cP$ contains a wedge of angle arbitrarily close to $2\pi/(n-2)$ based at $p$ such that a change of coordinate of the form
\begin{equation}
z\mapsto \frac{c_1}{(z-p)^{\frac{n-2}{2}}}
    \label{pre_fatou_formula}
\end{equation}
(for suitable $c_1\neq 0$ and a choice of a branch of square root) carries $\cP$ to an approximate right half-plane and conjugates $\sigma^{\circ 2}\vert_\cP$ to a map of the form $\zeta\mapsto \zeta+1+O(1/\zeta)$ near $\infty$. Thus, the proof of the existence of Fatou coordinates for parabolic germs can be applied mutatis mutandis to the current situation to prove the existence of a conformal map on $\cP$ that conjugates $\sigma^{\circ 2}$ to the translation $\zeta\mapsto\zeta+1$ on a right half-plane (cf. \cite[Theorem~10.9]{Mil06}). With this conformal conjugacy at our disposal, we can apply the arguments of \cite[Theorem~10.15]{Mil06} to conclude that every attracting direction at $p$ corresponds to an immediate attracting basin (the set of points that converge to $p$ asymptotic to this attracting direction under iterates of $\sigma^{\circ 2}$, and has $p$ on its boundary), and such an immediate attracting basin contains a critical point of $\sigma^{\circ 2}$. Clearly, such an immediate attracting basin of $\sigma^{\circ 2}$ is either fixed by $\sigma$ or it forms a $2-$cycle under $\sigma$. It follows that every cycle (under $\sigma$) of immediate attracting basins at $p$ contains a critical point of $\sigma$.

\subsubsection*{Repelling petals}
Similarly, for each repelling direction $\vec{v}$ at $p$, there is an open set $\cP$ containing the repelling direction $\vec{v}$ (called a \emph{repelling petal} associated with $\vec{v}$) such that $p\in\partial\cP$, a branch $g$ of $\sigma^{-2}$ is well-defined on $\cP$, $g(\cP)\subset\cP$, and the $g$-orbits of all points in $\cP$ converge to $p$ asymptotic to $\vec{v}$. The repelling petal also contains a wedge of angle arbitrarily close to $2\pi/(n-2)$ based at $p$ such that a change of coordinate of the form~\eqref{pre_fatou_formula} carries $\cP$ to an approximate left half-plane conjugating $\sigma^{\circ 2}\vert_{g(\cP)}$ to a map of the form $\zeta\mapsto \zeta+1+O(1/\zeta)$ near $\infty$.

The union of the attracting and repelling petals (along with the point $p$) covers the domain of definition of $\sigma^{\circ 2}$ near $p$.

\begin{lem}\label{cusp_tiling_dyn_lem}
Let $U\subset T^\infty(\sigma)$ be an open set such that there exists an inverse branch $g$ of $\sigma^{\circ 2}$ on $U$ with $g(U)\subset U$ and $p\in\partial g(U)$. Then, there exists $z\in U$ such that $U\ni g^{\circ k}(z)\xrightarrow{k\to\infty} p$.    
\end{lem}
\begin{proof}
As $g(U)$ is an open set having $p$ on its boundary and such that $\sigma^{\circ 2}$ is defined on $g(U)$, it must intersect at least one petal $\cP$. The fact that $T^\infty(\sigma)$ is totally invariant under $\sigma$ implies that $g(U)\subset T^\infty(\sigma)$.
Let $N:= U\cap P$. If $P$ is an attracting petal, then points in $N$ would converge non-trivially to $p$ under iterates of $\sigma^{\circ 2}$. However, this is not possible as $\sigma^{\circ 2}$-iterates of points in $U\subset T^\infty(\sigma)$ land in $T^0(\sigma)$ in finite time.
Thus, $P$ is a repelling petal for $\sigma^{\circ 2}$. Since $g$ is an inverse branch of $\sigma^{\circ 2}$, it now follows from the properties of repelling petals mentioned above that for $z\in N$, the orbit $\{g^{\circ k}(z)\} \xrightarrow{k\to\infty} p$. The $g$-invariance of $U$ implies that $g^{\circ k}(z)\in U$, for all~$k\in\N$.
\end{proof}

\subsubsection*{Critical orbits associated with attracting petals}
\begin{lem}\label{higher_order_cusp_lem}
    Set $\delta_p=\lfloor n/4\rfloor$. Then, there are $\delta_p$ distinct critical orbits of $\sigma$ that converge non-trivially to $p$. In particular, these critical orbits lie in $\Int{\mathscr{K}(\sigma)}$.
\end{lem}
\begin{proof}
   Recall that each $\sigma-$cycle of immediate attracting basins at $p$ contains a critical point of $\sigma$. Thus, it suffices to show that there are $\delta_p$ cycles (under $\sigma$) of attracting directions at $p$. By a local conformal change of coordinates, we may assume that $p=0$, the cusp points towards the positive real axis, and that $c\in\R\setminus\{0\}$ in Equation~\eqref{cusp_expansion_eqn}. This ensures that the positive real axis is an invariant direction under $\sigma$, and the second iterate $\sigma^{\circ 2}$ has $n-3$ non-real invariant directions (at the origin), each of which forms a $2-$cycle under $\sigma$ (see \cite[Proposition~A.4]{LMM24}).
   We consider the following two cases.
   \smallskip

   \noindent\textbf{Case I: $n=4j+3$, for some $j\geq 0$.} In this case, by \cite[Proposition~A.5]{LMM24}, the positive real axis is a repelling direction for $\sigma$. There are $4j$ non-real invariant directions for $\sigma^{\circ 2}$, and they form $2j$ many $2-$cycles under $\sigma$. Since the attracting and repelling directions are arranged in an alternating manner, precisely $j=\lfloor n/4\rfloor$ of these $2-$cycles correspond to attracting directions.
      \smallskip

   \noindent\textbf{Case II: $n=4j+1$, for some $j\geq 1$.} According to \cite[Proposition~A.5]{LMM24}, the positive real axis is an attracting direction for $\sigma$ in this case. There are $4j-2$ non-real invariant directions for $\sigma^{\circ 2}$, and they form $2j-1$ many $2-$cycles under $\sigma$. Since the attracting and repelling directions are arranged in an alternating manner, it is easily seen that $j-1$ of these $2-$cycles correspond to attracting directions. Thus, in this case, there exists a total of $1+(j-1)=j=\lfloor n/4\rfloor$ cycles of attracting directions at the origin.
   
  By the preceding analysis, we have the desired number of distinct critical orbits of $\sigma$ converging non-trivially to $p$ in each case. Finally, it is evident that the $\sigma-$cycles of immediate basins containing these $\delta_p$ critical orbits are disjoint and are contained in $\Int{\mathscr{K}(\sigma)}$.
\end{proof}

\subsection{Double point dynamics and critical points}

Let $p$ be a double point on $\partial\Omega$; i.e., $p$ is the point of touching (more precisely, tangential intersection) of two local non-singular pieces $\gamma^\pm$ of $\partial\Omega$. Suppose further that the curve germs $\gamma^\pm$ have a contact of order $n\geq 1$ at $p$. Note that since $\gamma^\pm$ do not cross each other at $p$, the order of contact is necessarily odd.

\subsubsection*{A convenient change of coordinates}
To study the dynamics of $\sigma$ near $p$, we may assume, possibly after performing a local conformal change of coordinates,  that $\gamma^-$ is contained in the real line, and $\gamma^+$ is a non-singular real-analytic curve contained in the (closure of the) upper half-plane and touching the real line at the origin. This change of coordinates conjugates the map $\sigma$, restricted to a neighborhood of $p$, to the piecewise anti-conformal map near the origin that acts as the complex conjugation map $\iota$ in the lower half-plane, and as the classical Schwarz reflection map $\iota^+$ in $\gamma^+$ above the curve $\gamma^+$ (see Figure~\ref{dp_second_iterate_fig}).

\subsubsection*{Order of contact, and Taylor series expansion}
Let us denote the (signed) curvature of $\gamma^+$ at a point $z_0\in\gamma^+$ by $\kappa(z_0)$, and its $r$-th  derivative (with respect to arc-length) by $\kappa^{(r)}(z_0)$, with the convention that $\kappa^{(0)}=\kappa$. The assumption that $\gamma^+$ and $\gamma^-$ have a contact of order $n=2l+1$ at $p$, for some  non-negative integer $l$, implies that if $n\geq 3$, then
$$
\kappa(0)=\cdots=\kappa^{(n-2)}(0)=0,\ \kappa^{(n-1)}(0)\neq 0;
$$
and if $n=1$, then $\kappa(0)\neq 0$. (Note that the order of contact is preserved by a conformal change of coordinates.)

For any point $z_0\in\gamma^+$, the Taylor series expansion of $\iota^+$ is given by
$$
\iota^+(z_0+\epsilon)= z_0+\overline{\epsilon}+\sum_{r\geq 2} b_r(z_0)\cdot \overline{\epsilon}^r,\ \epsilon\approx 0,
$$
where $b_r(z_0)$ is a polynomial function of $\kappa(z_0),\cdots,\kappa^{(r-2)}(z_0)$ with a trivial constant term (see \cite[\S 7]{Dav74} for details). In particular, we have that if $n\geq 3$, then
$$
b_2(0)=\cdots=b_n(0)=0,\ b_{n+1}(0)\neq 0;
$$
and if $n=1$, then $b_2(0)\neq 0$. Thus, the Taylor series expansion of $\iota^+$ at the origin is given by
\begin{equation}
\iota^+(\epsilon)= \overline{\epsilon}+ b_{n+1}(0)\cdot \overline{\epsilon}^{n+1} + O(\overline{\epsilon}^{n+2}),\ \epsilon\approx 0.
\label{schwarz_taylor_series_eqn}
\end{equation}
\smallskip

\subsubsection*{Attracting and repelling petals for the second iterate $\sigma^{\circ 2}$}
We refer the reader to Figure~\ref{dp_second_iterate_fig} for an illustration of the following discussion.
Note that the curve $\iota(\gamma^+)$ lies below the real line, while the curve $\iota^+(\gamma^-)$ lies above $\gamma^+$. Near the origin, the second iterate of $\sigma$ takes the form
$$
\sigma^{\circ 2}\equiv
\begin{cases}
    \iota\circ\iota^+,\quad \textrm{above}\ \iota^+(\gamma^-),\\
    \iota^+\circ\iota,\quad \textrm{below}\ \iota(\gamma^+).
\end{cases}
$$
By Equation~\eqref{schwarz_taylor_series_eqn}, we have $\left(\iota\circ\iota^+\right)(\epsilon)= \epsilon+ c\cdot \epsilon^{n+1} + O(\epsilon^{n+2})$ for $\epsilon\approx 0$, where $c\neq 0$ (cf. \cite[Proposition~4.6]{LMM25}). Thus, $\iota\circ\iota^+$ is a parabolic germ of multiplicity $n+1$ fixing the origin (cf. \cite[\ 10]{Mil06}). Clearly, the same is true for $ \iota^+\circ\iota$. Further, $\iota$ is a topological conjugacy between $ \iota\circ\iota^+$ and $ \iota^+\circ\iota$.

\begin{figure}[ht!]
\captionsetup{width=0.98\linewidth}
\begin{tikzpicture}
\node[anchor=south west,inner sep=0] at (0,0) {\includegraphics[width=0.6\textwidth]{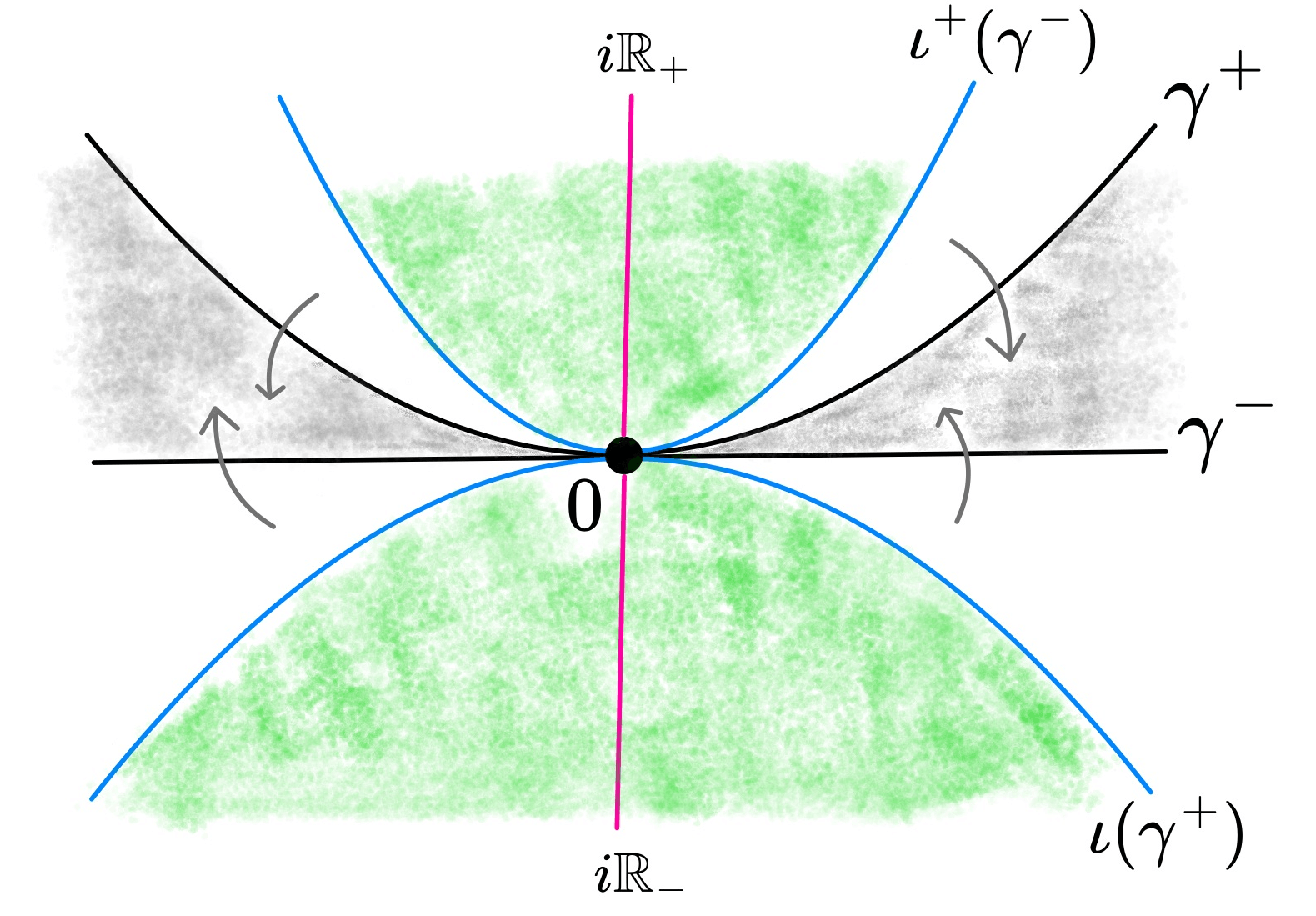}};
\end{tikzpicture}
\caption{The non-singular real-analytic curves $\gamma^\pm\subset\partial\Omega$ touch tangentially to form a double point at the origin. The gray shaded region is contained in $\Omega^\complement$. The second iterate $\sigma^{\circ 2}$ is defined in the green shaded region near the origin. The parabolic germs $\iota\circ\iota^+, \iota^+\circ\iota$ that define $\sigma^{\circ 2}$ near the origin (where $\iota, \iota^+$ are Schwarz reflections in the curves $\gamma^-, \gamma^+$, respectively) have the positive and negative imaginary axes (marked as $i\R_+, i\R_-$) as  invariant directions.}
\label{dp_second_iterate_fig}
\end{figure}

Let $\vec{v}$ be an attracting (respectively, repelling) direction for $\iota\circ\iota^+$ lying above $\iota^+(\gamma^-)$ (cf. \cite[\S 10]{Mil06}).  The fact that $\iota$ conjugates $\iota\circ\iota^+$ to $\iota^+\circ\iota$ implies that $\iota(\vec{v})$ is an attracting (respectively, repelling) direction for $\iota^+\circ\iota$ lying below $\iota(\gamma^+)$.
By the description of the map $\sigma^{\circ 2}$ given above, both $\vec{v}$ and $\iota(\vec{v})$ are attracting (respectively, repelling) directions for $\sigma^{\circ 2}$. The $\sigma^{\circ 2}-$invariant directions $\vec{v}, \iota(\vec{v})$ form a $2-$cycle under the antiholomorphic map $\sigma$.
Similarly, if $\cP$ is an attracting (respectively, repelling) petal of $\iota\circ\iota^+$ associated with $\vec{v}$, then $\iota(\cP)$ is an attracting (respectively, repelling) petal of $\iota^+\circ\iota$ associated with $\iota(\vec{v})$. Further, $\cP$ and $\iota(\cP)$ are also attracting (respectively, repelling) petals for $\sigma^{\circ 2}$.
The union of all these attracting and repelling petals of $\sigma^{\circ 2}$ cover the domain of definition of $\sigma^{\circ 2}$ near $0$. 

Thanks to the local dynamical description of $\sigma^{\circ 2}$ given above, the proof of Lemma~\ref{cusp_tiling_dyn_lem} applies verbatim to yield an analogous result for double points.

\begin{lem}\label{dp_tiling_dyn_lem}
Let $U\subset T^\infty(\sigma)$ be an open set such that there exists an inverse branch $g$ of $\sigma^{\circ 2}$ on $U$ with $g(U)\subset U$ and $p\in\partial g(U)$. Then, there exists $z\in U$ such that $U\ni g^{\circ k}(z)\xrightarrow{k\to\infty} p$.    
\end{lem}

\subsubsection*{Critical orbits associated with attracting petals}

\begin{lem}\label{higher_order_dp_lem}
    Set $\delta_p=\lfloor n/2\rfloor$. Then, there are $\delta_p$ distinct critical orbits of $\sigma$ that converge non-trivially to $p$. In particular, these critical orbits lie in $\Int{\mathscr{K}(\sigma)}$.
\end{lem}
\begin{proof}
Since $n=2l+1$ is odd, the germ $\iota\circ\iota^+$ has $n$ invariant directions above $\iota^+(\gamma^-)$, and the germ $\iota^+\circ\iota$ has $n$ invariant directions below $\iota(\gamma^+)$.
Further, the positive and negative imaginary axes $i\R_+$ and $i\R_-$ are invariant directions for $\sigma^{\circ 2}$.
We will use this information to count the number of 2-cycles of attracting directions of $\sigma$. We consider the following two cases.
\smallskip

\noindent\textbf{Case I: $n=4j+1$, for some $j\geq0$.}
In this case, if $i\R_+$ is an attracting direction for $\sigma^{\circ 2}$, then so is $i\R_-$, and there are
$4j/2+1=2j+1$
attracting directions for $\sigma^{\circ 2}$ above as well as below the real axis, hence
$2j+1=\lfloor n/2\rfloor +1$
many $2$-cycles (under $\sigma$) corresponding to them. On the other hand, if $i\R_+$ is a repelling direction for $\sigma^{\circ 2}$, then so is $i\R_-$, and there are
$4j/2=2j$
attracting directions for $\sigma^{\circ 2}$ above as well as below the real axis, and hence
$2j=\lfloor n/2\rfloor$
many 2-cycles (under $\sigma$) corresponding to them.
\smallskip

\noindent\textbf{Case II: $n=4j+3$, for some $j\geq0$.}
In this case, if $i\R_+$ is an attracting direction for $\sigma^{\circ 2}$, then there are
$4j/2+1=2j+1$
attracting directions for $\sigma^{\circ 2}$ above as well as below the real axis, and hence
$2j+1=\lfloor n/2\rfloor$
many 2-cycles (under $\sigma$) corresponding to them. On the other hand, if $i\R_+$ is a repelling direction for $\sigma^{\circ 2}$, then there are
$4j/2+2=2j+2$
attracting directions for $\sigma^{\circ 2}$ above as well as below the real axis, and hence
$2j+2=\lfloor n/2\rfloor +1$
many 2-cycles (under $\sigma$) corresponding to them.
\smallskip

In all cases, we have at least $\lfloor n/2\rfloor$ many 2-cycles of attracting directions for $\sigma$. By the proof of \cite[Theorem~10.15]{Mil06}, each such $2$-cycle of attracting directions has an associated $2$-cycle of immediate attracting basins containing a critical value of $\sigma$. This yields the
desired number of distinct critical orbits of $\sigma$ converging non-trivially to $p$, and it is evident that the $\sigma-$cycles of immediate basins containing these $\delta_p$ critical orbits are disjoint and are contained in $\Int{\mathscr{K}(\sigma)}$.
\end{proof}

\section{Bound on connectivity and double points}\label{special_case_sec}

In this section, we will prove the following weaker version of Theorem~\ref{main_thm_1}.

\begin{theorem}\label{special_case_thm}
$$
\mathrm{conn}(\Omega) + \# D \leq 2d_f-4.
$$
\end{theorem}

\subsection{From individual contributions to global upper bound}\label{component_wise_reduction_subsec}

The crucial idea in the proof of Theorem~\ref{special_case_thm} is to associate critical points with each component of $\Omega^\complement$, as stated in the following key proposition.
Given a component $K$ of $\Omega^\complement$, we denote the connected components of its `desingularization' $K\setminus\mathcal{S}$ by $K_0^0,\cdots,K_0^n$. Recall the notation $T_{K_0^j}^\infty(\sigma)$ from Section~\ref{schwarz_dyn_basic_subsec}.

\begin{prop}\label{individual_contribution_prop}
Every component $K$ of $\Omega^\complement$ accounts for at least $(\# D_K + 3)$ critical points of $f$, where $\# D_K$ is the number of double points on $\partial K$. More precisely, there are at least $(\# D_K + 3)$ critical points of $f$ (counted with multiplicity) in $\displaystyle f^{-1}(\partial K\cap\mathcal{S})\sqcup\bigsqcup_{j=0}^n f^{-1}(T_{K_0^j}^\infty(\sigma))$.
\end{prop}

Most of this section will be devoted to the proof of Proposition~\ref{individual_contribution_prop}. Before we delve into the proof, let us observe that Theorem~\ref{special_case_thm} is an easy consequence of this proposition.

\begin{proof}[Proof of Theorem~\ref{special_case_thm}]\label{pf_special_case_thm}
We recall the notation $k=\mathrm{conn}(\Omega)$.
By Proposition~\ref{schwarz_crit_pnt_prop} and Proposition~\ref{individual_contribution_prop},
\begin{align*}
\sum\limits_{\substack{\mathrm{components\ K}\\ \mathrm{of\ \Omega^\complement}}} & (\# D_K + 3) \leq 2d_f+2k-4\\
\implies &\# D + 3k \leq 2d_f+2k-4\\
\implies &\# D + k \leq 2d_f-4.
\end{align*}
This completes the proof.
\end{proof}

\subsection{Associating critical points with each component of $\Omega^\complement$}\label{crit_points_per_component_subsec}

The proof of Proposition~\ref{individual_contribution_prop} will be divided into various cases. For clarity of exposition, we record these cases as separate statements. The background on local dynamics of Schwarz reflections near cusps and double points from Section~\ref{cusp_dp_dyn_sec} will be extensively used in the following proofs.

\subsubsection{The non-singular case}

\begin{lem}\label{non_sing_lem}
Let $K$ be a component of $\Omega^\complement$ such that $\partial K$ is non-singular. Then, there are at least $3$ critical points of $f$ (counted with multiplicity) in~$\displaystyle f^{-1}(T_{K}^\infty(\sigma))$.
\end{lem}
\begin{remark}
We note that Lemma~\ref{non_sing_lem} (combined with the proof of Theorem~\ref{special_case_thm} given above) proves the desired upper bound $2d_f-4$ on $\mathrm{conn}(\Omega)$ in the case when $\partial\Omega$ is non-singular. This gives an alternative proof of the same result first established in \cite[\S 4]{LM16} using quasiconformal surgery techniques.
\end{remark}
\begin{proof}
Since $\partial K$ is non-singular, there exists a relative neighborhood $N_\epsilon(\partial K)\cap\overline{\Omega}$ of $\partial K$ (in $\overline{\Omega}$) such that for each $z\in N_\epsilon(\partial K)\cap\overline{\Omega}$, we have that $\sigma(z)\in K$  (cf. \cite[\S 4]{LM16}).
Let $A_0$ be the component of $\sigma^{-1}(K)$ containing this relative neighborhood $N_\epsilon(\partial K)\cap\overline{\Omega}$. By construction, $\partial K\subset\partial A_0\subset A_0$ (see Figure~\ref{fig:4.1}).
\begin{figure}[ht]
\captionsetup{width=0.98\linewidth}
    \centering
    \includegraphics[width=0.5\linewidth]{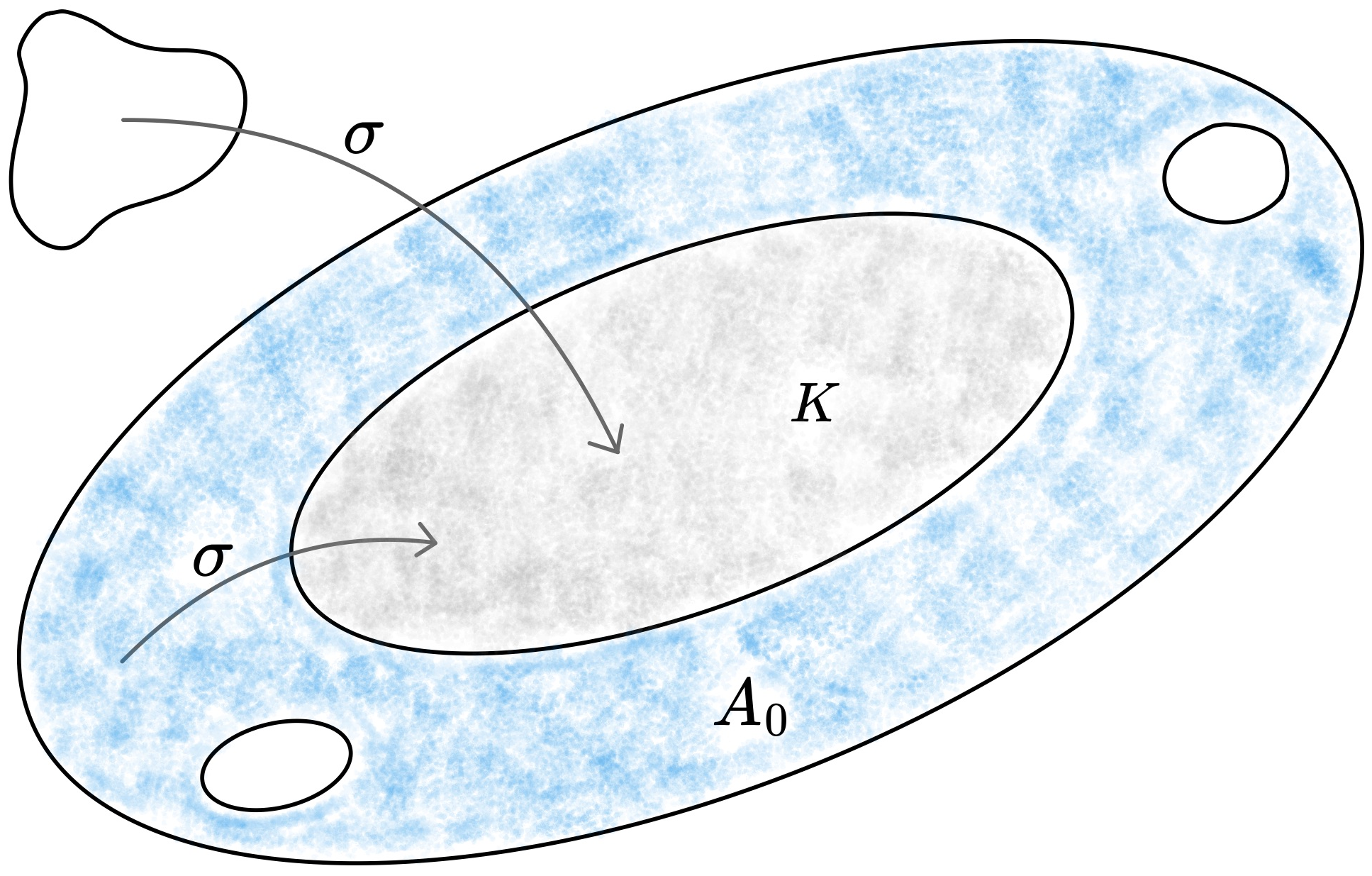}
    \caption{Depicted are the component $K$ of $\Omega^\complement$ and the neighboring component $A_0$ of $\sigma^{-1}(K)$.}
    \label{fig:4.1}
\end{figure}

\textbf{Step I: At least two critical points of $f$ over $K$.} We claim that $A_0$ is not simply connected. To see this, we assume otherwise and note that under this assumption, $A_0$ would be a topological disk. As the boundary $\partial K$ is contained in $\partial A_0$, we must have that $A_0=\widehat{\C}\setminus\Int{K}=\overline{\Omega}$ and hence $\sigma^{-1}(\Omega)=\emptyset$. This is a contradiction to the hypothesis that $\deg\left(\sigma:\sigma^{-1}(\Omega)\to\Omega\right)=d_f - 1 \geq 2$.

Since $A_0$ is not simply connected and $K$ is a closed Jordan disk, it follows that $\sigma : \Int{A_0} \to \Int{K}$ is a branched covering that is not a conformal isomorphism, and hence $d_{A_0} := \deg\left(\sigma: \Int{A_0} \to \Int{K}\right) \geq 2$. By the Riemann-Hurwitz formula,
\begin{equation}
\#_m\ \mathrm{crit}(\sigma\vert_{\Int{A_0}}) = d_{A_0}\cdot\chi(\Int{K}) - \chi(\Int{A_0}) \geq 2 - \chi(\Int{A_0}) \geq 2,
\label{rh_1_eqn}
\end{equation}
as $\chi(\Int{K}) = 1$ and $\chi(\Int{A_0}) \leq 0$ (here and elsewhere, $\#_m$ stands for the number of critical points counted with multiplicity).
Thus, $\sigma$ has at least two critical points in $\Int{A_0}$. This provides us with at least two critical points of $f$ in $f^{-1}(\Int{K})$. 
\smallskip

\textbf{Step II: Construction of annular rank $n$ tiles and successive homeomorphisms.} 
By way of contradiction, let us assume that there is no further critical point of $f$ in $f^{-1}(T_K^\infty(\sigma))$; i.e, there is no further critical point of $\sigma$ in $T_K^\infty(\sigma)$.

If $A_0$ is not an annulus, then $\chi(\Int{A_0}) \leq -1$ and it follows from Equation~\eqref{rh_1_eqn} that $\#_m\ \mathrm{crit}(\sigma\vert_{\Int{A_0}}) \geq 3$; i.e., we are done. So, we consider the case that $A_0$ is an annulus. Again, if $d_{A_0} \geq 3$, then $\#_m\ \mathrm{crit}(\sigma\vert_{A_0}) \geq 3$ and are done. So, we only need to consider the case $d_{A_0} = 2$. Under this assumption, $\sigma: A_0 \xrightarrow{2:1} K$ is a branched covering map and hence $\sigma: \partial A_0 \xrightarrow{2:1} \partial K$ is an unbranched covering of degree two (see Figure~\ref{fig:4.2}).

Let $A_1$ be the component of $\sigma^{-1}(A_0)$ containing the outer boundary of $A_0$; i.e., $\partial A_0 \setminus \partial K \subseteq \partial A_1$. Since $A_0$ is an annulus, it follows from the Riemann-Hurwitz formula that $A_1$ is not simply connected. If $A_1$ is not an annulus, then
\begin{equation}
\#_m\ \mathrm{crit}(\sigma\vert_{\Int{A_1}}) = d_{A_1}\cdot\chi(\Int{A_0}) - \chi(\Int{A_1}) = - \chi(\Int{A_1}) \geq 1.
\label{rh_2_eqn}
\end{equation}
This would give a critical value of $\sigma$, and hence of $f$, in $\Int{A_0}$ and we are done. So, we consider $A_1$ to be an annulus. We will argue that in this case, $\sigma: A_1\rightarrow A_0$ is an isomorphism. Indeed, the boundary of the annulus $A_0$ has two (non-singular) components: $\partial K$ and $\partial A_0 \setminus \partial K$, and $\sigma$ is the identity on $\partial K$, so $\sigma : \partial K \to \partial K$ has degree $1$. As  $\sigma: \partial A_0 \xrightarrow{2:1} \partial K$ is a covering of degree two, the map $\sigma$ must carry $\partial A_0 \setminus \partial K$ homeomorphically onto $\partial K$. Further, as $\sigma: \Int{A_1} \rightarrow \Int{A_0}$ is an annulus-to-annulus covering of degree $d_{A_1}$, the degree of $\sigma$ restricted to each of the components of $\partial A_1$ is also $d_{A_1}$. It follows that $d_{A_1}=1$ (see Figure~\ref{fig:4.2}).

Let $A_2$ be the component of $\sigma^{-1}(A_1)$ containing the outer boundary of $A_1$; i.e., $\partial A_1 \setminus \partial A_0 \subseteq \partial A_2$. As before, the absence of critical points of $\sigma$ in $A_2$ forces $A_2$ to be an annulus with non-singular boundary components. This fact, in turn, implies that the annulus-to-annulus covering $\sigma: A_2 \rightarrow A_1$ is an isomorphism (see Figure~\ref{fig:4.2}).

Proceeding thus, we obtain a sequence of isomorphisms $\{\sigma : A_n \to A_{n-1}\}_{n \geq 1}$ as depicted in Figure~\ref{fig:4.2}.
\begin{figure}[ht]
\captionsetup{width=0.98\linewidth}
    \includegraphics[width=0.8\linewidth]{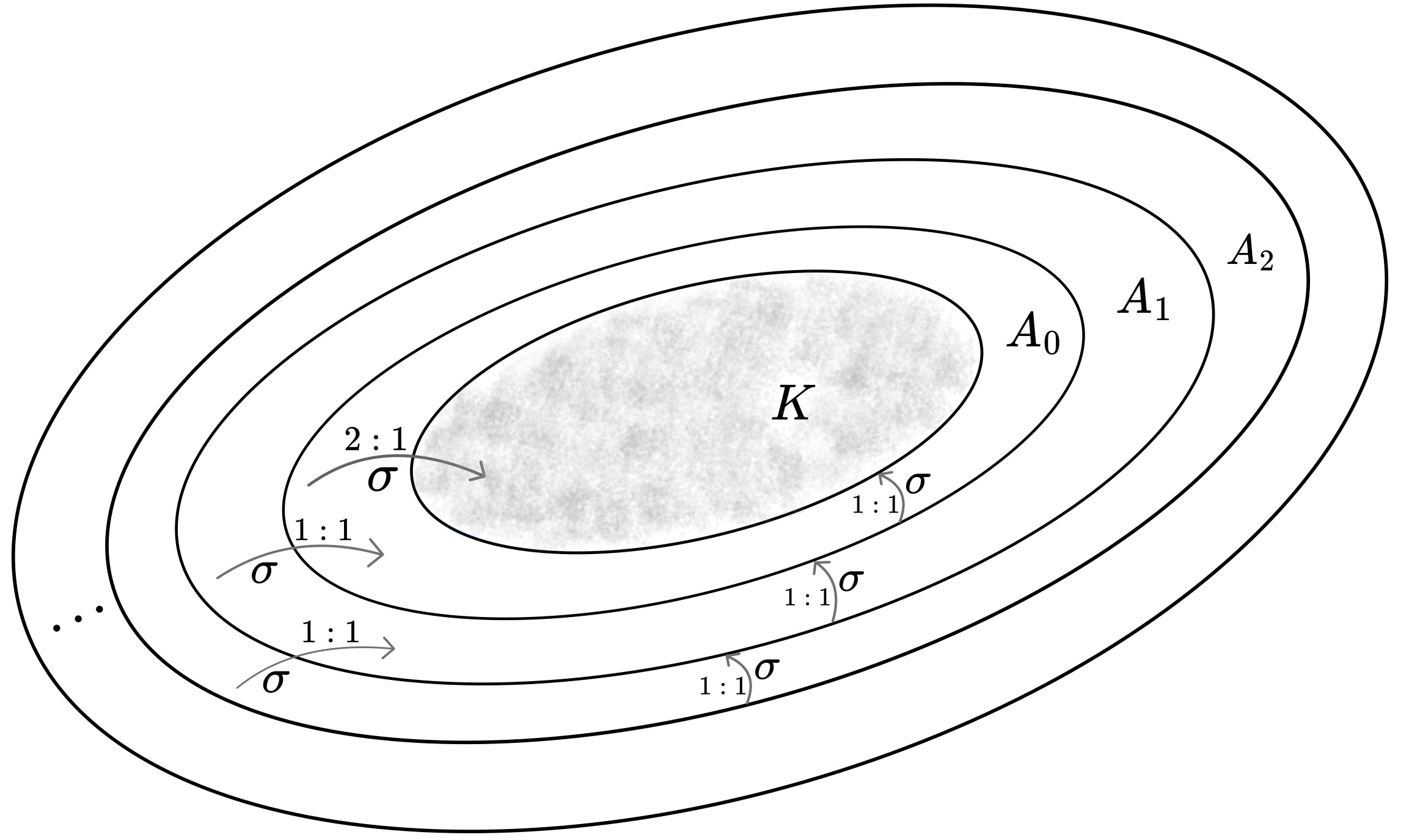}
    \caption{The annular tiles of various ranks and their mapping degrees under the Schwarz reflection map $\sigma$ are shown.}
    \label{fig:4.2}
\end{figure}
\smallskip

\textbf{Step III: A modulus estimate and negligibility of the non-escaping set.}
Let $A_\infty := \bigcup\limits_{n \geq 0}A_n$. Then, 
$$
\Int{A_\infty}=\bigcup_{n=1}^\infty\left(\Int{\bigcup_{i=0}^n A_i}\right);
$$
i.e., $\Int{A_\infty}$ is an increasing union of open topological annuli, and hence an open topological annulus itself.
As $\sigma : A_n \to A_{n-1}$ is an isomorphism, we have that $\mathrm{mod}(A_n) = \mathrm{mod}(A_0)$, $n\geq 1$ (where, $\mathrm{mod}(B)$ stands for the modulus of an annulus $B$). By \cite[Proposition~5.4]{BH92},
\[
\mathrm{mod}(A_\infty) \geq \sum_{n}\mathrm{mod}(A_n) = \sum_{n}\mathrm{mod}(A_0).
\]
Hence, $\Int{A_\infty}$ is an annulus of infinite modulus and by \cite[Proposition 5.5]{BH92}, the boundary component of $\Int{A_\infty}$ different from $\partial K$ is a singleton. Hence, $\widehat{\C}\setminus T_K^\infty(\sigma)=\widehat{\C} \setminus \left(A_\infty \cup K\right) = \{p\}$, for some $p\in\widehat{\C}$.

It follows from the above discussion that $K$ is the desingularized droplet, and that the tiling set $T^\infty(\sigma)$ is equal to $K\cup A_\infty$. Hence, the non-escaping set $\mathscr{K}(\sigma)$ is the singleton $\{p\}$. 
\smallskip

\textbf{Step IV: Contradicting the degree of $\sigma$.}
As the tiling set and non-escaping set are both completely invariant under $\sigma$, we have $\sigma^{-1}(p)~=~\{p\}$.

Finally, as $\deg\left(\sigma:\sigma^{-1}(\Omega)\to\Omega\right) \geq 2$, we conclude that $\sigma$ maps $p$ to itself with local degree at least two, and hence $p\in \mathrm{crit}(\sigma)$.
In fact, $p$ is a super-attracting fixed point of $\sigma$ and hence has a basin of attraction $\mathcal{A}_\sigma(p)$, which is an open subset of $\mathscr{K}(\sigma)$ (cf. \cite[\S 9]{Mil06}). This is not possible, as $\mathscr{K}(\sigma)=\{p\}$, and we obtain a contradiction. This completes the proof of the lemma.
\end{proof}

Recall that the quadrature function $R_\Omega$ and the Schwarz reflection $\sigma$ have the same poles of the same multiplicity (by \cite[Lemma~3.1]{LM16}). The number of such distinct poles is given by $n_\Omega$. We record a conditional improvement of Lemma~\ref{non_sing_lem} that will be useful in the next section.

\begin{cor}\label{non_sing_pole_cor}
Assume that one of the two following conditions holds true.
\noindent\begin{enumerate}
    \item $\Omega$ is unbounded, $K$ is a non-singular component of $\Omega^\complement$, and $\infty\in\widehat{T_K^\infty}(\sigma)$. 
    \item $\Omega$ is bounded, $K$ is a non-singular component of $\Omega^\complement$, and $\infty\in\Int{K}$.    
\end{enumerate}    
Then, at least $\max\{d_f-n_\Omega+1,3\}$ critical points of $f$ (counted with multiplicity) lie in~$f^{-1}(\widehat{T_K^\infty}(\sigma))$.
\end{cor}

The proof of part $\mathrm{(2)}$ of Corollary~\ref{non_sing_pole_cor} will use the following elementary topological~result. 

\begin{lem}\label{val_lem}
Let $D_1$ and $D_2$ be two domains in $\widehat{\C}$, such that, $D_2$ is simply connected and $D_1$ is not simply connected. Then, any branched covering $g \colon D_1 \to D_2$ has at least two distinct critical values in $D_2$.
\end{lem}
\begin{proof}
Let $d$ be the degree of $g$. By the Riemann-Hurwitz formula,
\begin{equation}
\#_m\ \mathrm{crit}(g) = d\chi(D_2) - \chi(D_1) = d - \chi(D_1) \geq d,
\label{rh_3_eqn}
\end{equation}
(counted with multiplicity), as $\chi(D_2) = 1$ and $\chi(D_1) \leq 0$.

Assume that $g$ has a unique critical value in $D_2$, say $\omega$. Then, every critical point of $g$ maps to $\omega$. Let $\mathrm{crit}(g) = \{c_1,\cdots,c_r\}$, and $m_i$ be the multiplicity of the critical point $c_i$. By Inequality~\ref{rh_3_eqn}, we have $\sum_{i=1}^r m_i\geq d$.

We denote the valency (i.e., the local degree) of $g$ at a point $z\in D_1$ by $v(z)$. Note that $v(c_i)=m_i+1$. Then,
\begin{equation}
d=\sum\limits_{z \in g^{-1}\omega} v(z) \implies d\geq \sum\limits_{i=1}^{r} v(c_i) \implies d\geq \sum_{i=1}^r m_i +r \geq d+r 
\label{val_eqn}
\end{equation}
(cf. \cite[\S 2.5]{Bea91}). 
This implies that $r=0$, which is impossible by Inequality~\ref{rh_3_eqn}. Hence, $g$ must have at least two distinct critical values in $D_2$.
\end{proof}

\begin{proof}[Proof of Corollary~\ref{non_sing_pole_cor}\ \upshape (Part (1)).]
By Step I in the proof of Lemma~\ref{non_sing_lem}, we have at least two critical points of $f$ (counted with multiplicity) in $f^{-1}(\Int{K})$, say $c_1$ and $c_2$, which are not poles of $f$ as $f(c_1), f(c_2) \in K$ but $\infty \notin K$.
By  Section~\ref{pole_crit_points_subsec} and our assumption, there are $(d_f-n_\Omega-1)$ critical points of $f$ (counted with multiplicity) in $f^{-1}(\infty)\subset f^{-1}(\widehat{T_K^\infty}(\sigma))$. Counting them along with $c_1$ and $c_2$, we get at least
$$
(d_f-n_\Omega-1) + 2 = d_f-n_\Omega+1
$$
critical points of $f$ (counted with multiplicity) in $f^{-1}(\widehat{T_K^\infty}(\sigma))$.
\end{proof}

\begin{proof}[Proof of Corollary~\ref{non_sing_pole_cor}\upshape (Part (2)).]
Note that $\Int{K}$ is simply connected, and by Step~I of the proof of Lemma~\ref{non_sing_lem}, the interior of the neighboring component $A_0$ of $\sigma^{-1}(K)$ is not simply connected. By Lemma~\ref{val_lem}, the restriction of the Schwarz reflection map $\sigma\vert_{\Int{A_0}}$ has a critical value in $\Int{K}$ distinct from $\infty$, and hence a critical point in $\Int{A_0}$, say $c$, which is not a pole. Counting $\eta((f\vert_{W^+})^{-1}(c))$ with the $(d_f-n_\Omega)$ critical points of $f$ in $f^{-1}(\infty)\subset f^{-1}(\Int K)\subset f^{-1}(\widehat{T_{K}^\infty}(\sigma))$ gives us the desired $(d_f-n_\Omega+1)$ critical~points.
\end{proof}

\subsubsection{The case of no double points}\label{4.2.2}

Recall from Section~\ref{singularity_subsec} that every cusp on $\partial\Omega$ is a critical value of $f$. Hence, if $\partial K$ has no double points and at least $3$ cusps, then it trivially satisfies the conclusion of Proposition~\ref{individual_contribution_prop}. In what follows, we consider the cases where $\partial K$ has no double points and one/two cusp(s). Observe that when $\partial K$ contains no double points, the desingularization $K\setminus\mathcal{S}$ is connected; i.e., $K\setminus\mathcal{S}$ consists of a single component~$K_0$.

\begin{lem}[The case of two cusps]\label{two_cusp_lem}
Let $K$ be a component of $\Omega^\complement$ such that $\partial K$ has no double points and exactly two cusps. Then, there are at least $3$ critical points of $f$ (counted with multiplicity) in $\displaystyle f^{-1}(\partial K\cap\mathcal{S})\sqcup f^{-1}(T_{K_0}^\infty(\sigma))$.
\end{lem}
\begin{proof}
In this case, $f$ has two distinct critical values given by the cusps on $\partial K$, and our goal is to associate another critical point of $f$ with $K$, by locating a critical point of $\sigma$ escaping to $K_0$.
\smallskip

\textbf{Step I: The desingularized droplet.}
Let $p_1$ and $p_2$ be the two cusps on $\partial K$. Then, we have that $K_0 = K\setminus\{p_1,p_2\}$, and $\partial K_0 \setminus \{p_1,p_2\}$ is a non-singular curve consisting of two components, say $\gamma^+$ and $\gamma^-$ (see Figure~\ref{fig:4.3}).

As in Lemma~\ref{non_sing_lem}, we assume by way of contradiction that there is no critical point of $f$ in $\displaystyle f^{-1}(T_{K_0}^\infty(\sigma))$; i.e, there is no critical value of $f$ (or equivalently of $\sigma$) in $T_{K_0}^\infty(\sigma)$.
\smallskip

\textbf{Step II: Construction of simply connected rank $n$ tiles and successive homeomorphisms.}
Let $K_1^+$ and $K_1^-$ be the components of $\sigma^{-1}(K_0)$ adjacent to $\gamma^+$ and $\gamma^-$ respectively; i.e., $\gamma^\pm\subset\partial K_1^\pm$. At this point, we do not know that $K_1^+\neq K_1^-$, but this will be established shortly. By Proposition~\ref{schwarz_deg_prop}, $\sigma:\sigma^{-1}(\Int{\Omega^\complement})\to\Omega^\complement$ is a branched covering. Let $d_1^\pm$ be the degrees of $\sigma:K_1^\pm \to K_0$, respectively.

By our hypothesis, $K_1^+ \cup K_1^-$ does not contain any critical point of $\sigma$. Then $\sigma:K_1^\pm \to K_0$ are unbranched covering maps. As $K_0$ is simply connected, it follows that $\sigma:K_1^\pm \to K_0$ are homeomorphisms. Hence, $d_1^\pm=1$ and $K_1^\pm$ are simply connected. In particular, $\sigma:\partial K_1^\pm\to\partial K_0$ are also homeomorphisms. If $K_1^+=K_1^-$, then the above observations, combined with the fact that $\sigma\vert_{\partial K_0}$ is the identity map, imply that $\partial K_0=\partial K_1^\pm$. But this would force the equality $K_0\cup K_1^\pm=\widehat{\C}\setminus\{p_1,p_2\}$ to hold true, implying that $\deg(\sigma:\sigma^{-1}(\Omega)\to\Omega)=0$, which contradicts our hypothesis. Hence, $K_1^+$ and $K_1^-$ are disjoint simply connected sets such that $\partial K_1^\pm$ are non-singular away from $p_1, p_2$.

Observe that $\sigma\vert_{K_1^+ \sqcup K_1^-}$ is an anti-conformal reflection across $\partial K$. Hence, there are components $K_2^+$ and $K_2^-$ of $\sigma^{-1}(K_1^+ \sqcup K_1^-)\subset\sigma^{-2}(K_0)$, adjacent to $K_1^+$ and $K_1^-$ respectively, such that $\sigma$ maps $K_2^+$ onto $K_1^-$ and $K_2^-$ onto $K_1^+$, say with degrees $d_2^+$ and $d_2^-$ respectively (see Figure~\ref{fig:4.3}). 

Once again, our hypothesis implies that $K_2^+ \sqcup K_2^-$ does not contain any critical point of $\sigma$, so by the previous argument, $d_2^\pm=1$ and $K_2^\pm$ are simply connected. Further, $\partial K_2^\pm$ are non-singular away from $p_1, p_2$.

Proceeding thus, we obtain a pair of sequences of simply connected sets $\{K_n^+\}_{n \geq 1}$ and $\{K_n^-\}_{n \geq 1}$, such that,
\[
\sigma:K_n^+ \to K_{n-1}^- \text{ and } \sigma:K_n^-\to K_{n-1}^+  
\]
are homeomorphisms, $K_n^\pm$ are adjacent to $K_{n-1}^\pm$, and $\partial K_n^\pm\setminus\{p_1, p_2\}$ are non-singular for all $n \geq 1$ (see Figure~\ref{fig:4.3}).

\begin{figure}[ht]
\captionsetup{width=0.98\linewidth}
    \includegraphics[width=0.66\linewidth]{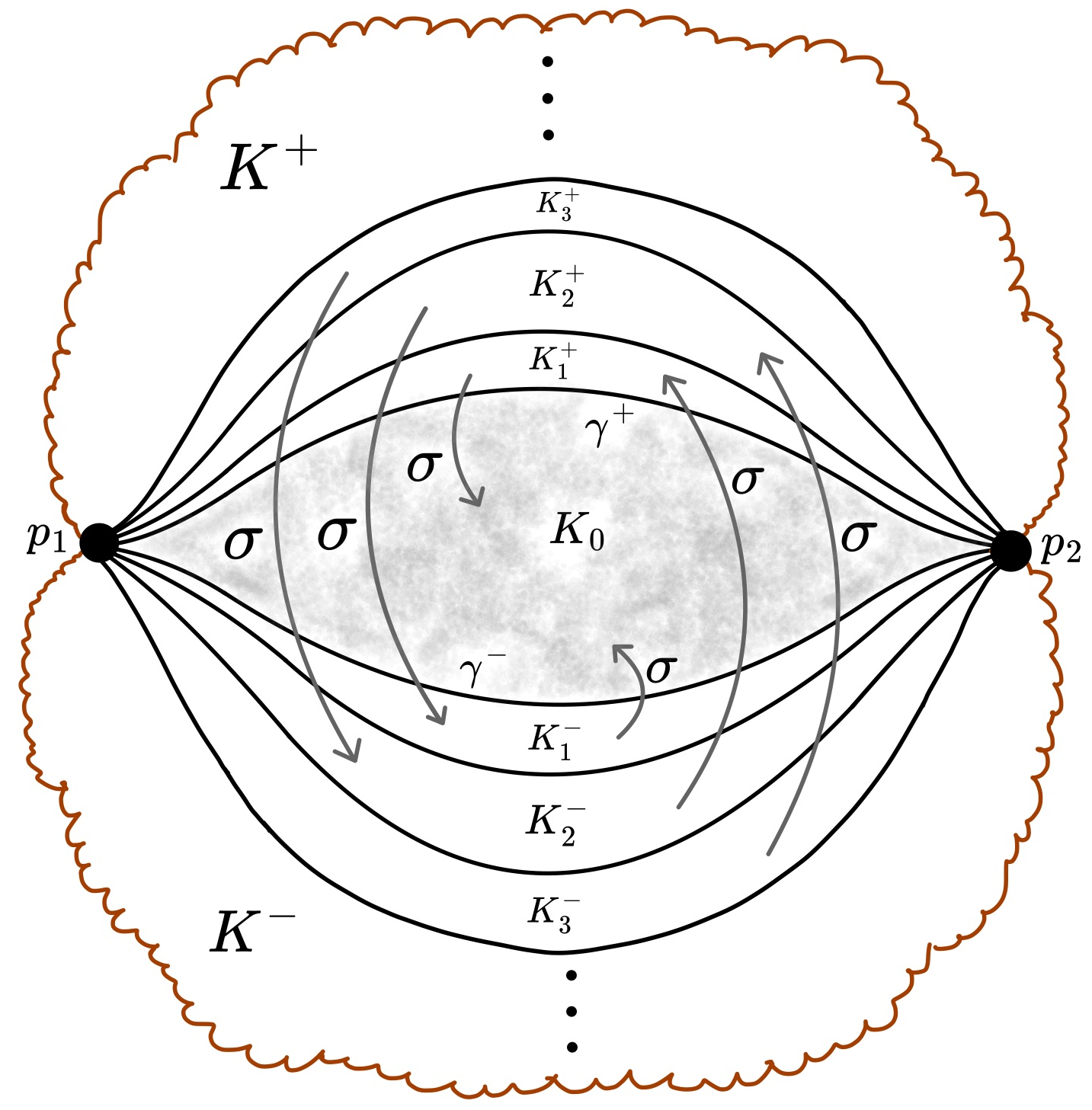}
    \caption{Illustrated is the dynamics of the Schwarz reflection $\sigma$ on the simply connected tiling component $T_{K_0}^\infty(\sigma)$, which has the brown curve as its boundary. The gray region is the rank $0$ tile $K_0$ having two cusps on its boundary. The part of $T_{K_0}^\infty(\sigma)$ above (respectively, below) $K_0$ is the set $K^+$ (respectively, $K^-$).}
    \label{fig:4.3}
\end{figure}
\smallskip

\textbf{Step III: The simply connected, forward-invariant tiling component $T_{K_0}^\infty(\sigma)$.}  
Let $K^+ := \bigcup\limits_{n \geq 1}K_n^+$ and $K^- := \bigcup\limits_{n \geq 1}K_n^-$. Then
\[
T_{K_0}^\infty(\sigma) = K_0 \cup K^+ \cup K^-,
\]
and it is an invariant component of the tiling set $T^\infty(\sigma)$ containing $K_0$.
Observe that 
$$\sigma^{\circ2}: K_n^+\to K_{n-2}^+,\ \mathrm{and}\ \sigma^{\circ 2}: K_n^-\to K_{n-2}^-
$$ 
are homeomorphisms for all $n \geq 2$ (where we use the convention $K_0^+=K_0^-=K_0$). Hence,
\[
\sigma^{\circ2}: X^+:=\Int{\bigcup\limits_{n \geq 2}K_n^+} \longrightarrow Y^+:=\left(\bigcup\limits_{n \geq 0}K_n^+\right) \setminus\, \gamma^-
\]
is a biholomorphism. Further, we have that $X^+\subsetneq  Y^+$.
\smallskip

\textbf{Step IV: A contraction mapping.}
Consider the biholomorphism 
$$
g := \sigma^{-2} \colon Y^+ \to X^+.
$$ 
As $X^+, Y^+$ are hyperbolic Riemann surfaces, it follows by the Schwarz--Pick Theorem that $g$ is an isometry with respect to the hyperbolic metrics on $X^+$ and $Y^+$. Hence,
\begin{equation}
d_{Y^+}(z,w) = d_{X^+}(g(z),g(w))\quad \forall\quad z,w \in Y^+.
\label{isom_eqn}
\end{equation}
On the other hand, since $X^+ \subsetneq Y^+$, it again follows by the Schwarz--Pick Theorem that the inclusion map $(X^+,d_{X^+}) \hookrightarrow (Y^+,d_{Y^+})$ is a strict contraction; i.e.,
\begin{equation}
d_{X^+}(z,w) > d_{Y^+}(z,w)\quad \forall\quad z \neq w \in X^+.
\label{contract_eqn}
\end{equation}
From Equations~\eqref{isom_eqn} and~\eqref{contract_eqn}, we have
\[
d_{Y^+}(g(z),g(w)) < d_{Y^+}(z,w)\quad \forall\quad z \neq w \in Y^+.
\]
Hence, $g \colon (Y^+,d_{Y^+}) \to (Y^+,d_{Y^+})$ is a strict contraction.
\smallskip

\textbf{Step V: Dynamics at $p_i$, $i\in\{1,2\}$.}
Now, we investigate the dynamics of $g$, or equivalently $\sigma^{\circ2}$, at the cusp points $p_1$ and $p_2$. We will refer to Section~\ref{cusp_dp_dyn_sec} for the same.
\begin{figure}[h!]
\captionsetup{width=0.98\linewidth}
    \includegraphics[width=0.96\linewidth]{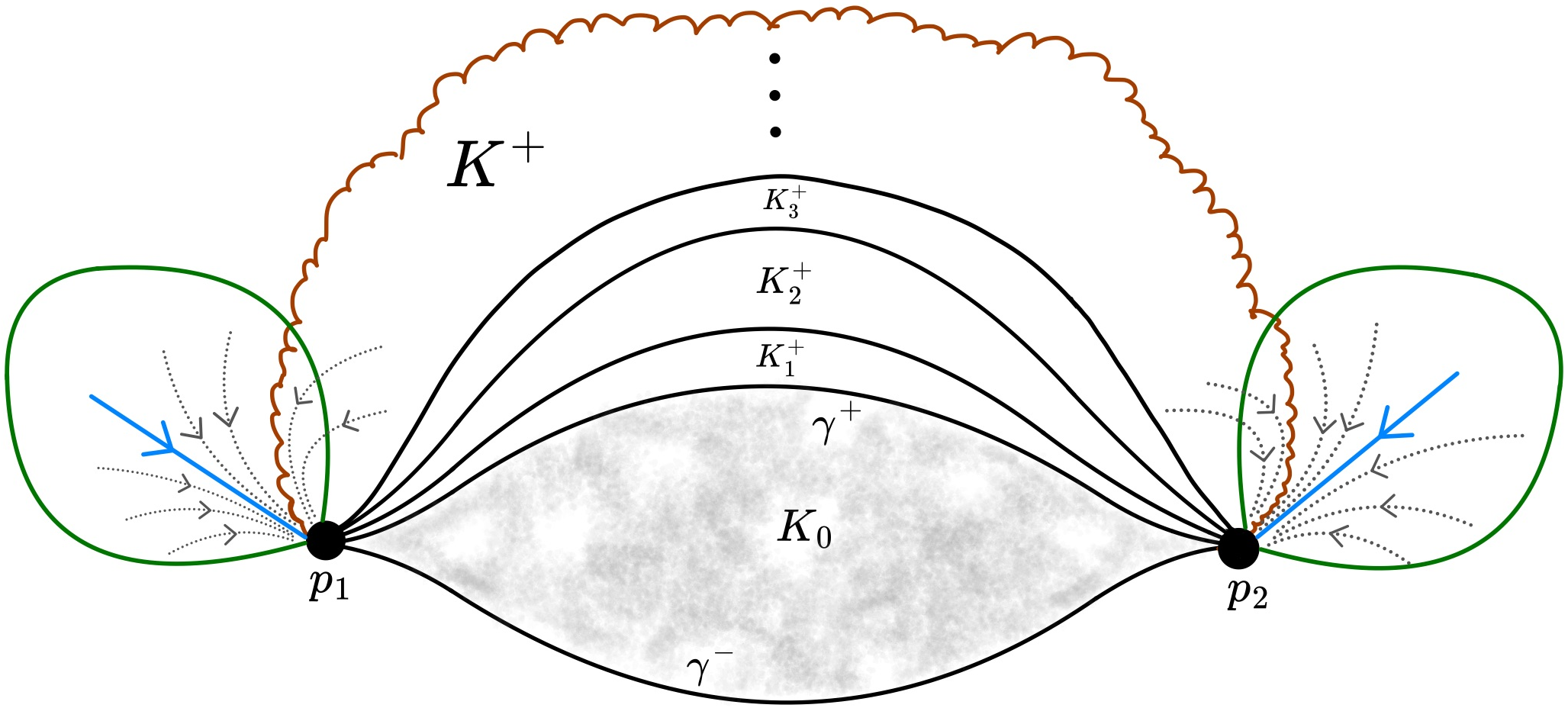}
    \caption{The dynamics of the branch $g$ of $\sigma^{-2}$ near the cusps $p_1, p_2$ is shown. The blue directions are repelling for $\sigma^{\circ 2}$ and hence attracting for $g$. The domains enclosed by the green curves are the corresponding attracting petals for $g$. Some $g-$orbits converging to $p_i$ through these petals are shown in gray.}
    \label{fig:4.4}
\end{figure}
Clearly, the set $Y^+$ satisfies all the conditions of Lemma~\ref{cusp_tiling_dyn_lem} for both cusps $p_1, p_2$. Hence, there exist $z_i\in Y^+$ such that $g^{\circ n}(z_i)\xrightarrow{n\to\infty} p_i$, $i\in\{1,2\}$. (See Figure~\ref{fig:4.4}.)
\smallskip

\textbf{Step VI: Contradicting hyperbolic contraction.}
By the previous paragraph, there exist $z_i\in Y^+$ with $g^{\circ n}(z_i)\xrightarrow{n\to\infty} p_i\in\partial Y^+$, $i\in\{1,2\}$. Let $d_{Y^+}(z_1,z_2)=\delta>0$. Since $g$ is a hyperbolic contraction, it follows that $d_{Y^+}(g^{\circ n}(z_1),g^{\circ n}(z_2))\leq\delta$, for all $n\geq 0$. The fact that $g^{\circ n}(z_1)\xrightarrow{n\to\infty} p_1\in\partial Y^+$ implies that the closed hyperbolic balls $\overline{B_{Y^+}}(g^{\circ n}(z_1),\delta)$ also converge to the cusp $p_1$ (as the Euclidean diameter of $\overline{B_{Y^+}}(g^{\circ n}(z_1),\delta)$ goes to $0$, cf. \cite[Theorem~3.4]{Mil06}). Thus, we must have that $g^{\circ n}(z_2)\xrightarrow{n\to\infty} p_1$ as well, which contradicts the fact that~$p_1\neq p_2$.
\end{proof}

\begin{cor}\label{two_cusp_pole_cor}
Let $K$ be a component of $\Omega^\complement$ such that $\partial K$ has no double points and at least two cusps. 
\noindent\begin{enumerate}
    \item If $\Omega$ is unbounded and $\infty\in\widehat{T_{K_0}^\infty}(\sigma)$, then at least $\max\{d_f-n_\Omega+1,3\}$ critical points of $f$ (counted with multiplicity) lie in $\displaystyle f^{-1}(\partial K\cap\mathcal{S})\sqcup f^{-1}(\widehat{T_{K_0}^\infty}(\sigma))$.
    \item If $\Omega$ is bounded and $\infty\in\Int{K}$, then at least $\max\{d_f-n_\Omega+2,3\}$ critical points of $f$ (counted with multiplicity) lie in $\displaystyle f^{-1}(\partial K\cap\mathcal{S})\sqcup f^{-1}(\widehat{T_{K_0}^\infty}(\sigma))$.   
\end{enumerate}    
\end{cor}
\begin{proof}[Proof of Part $\mathrm{(1)}$]
The meromorphic map $f$ has at least two distinct critical values given by the cusps on $\partial K$ and hence two associated critical points, say $c_1$ and $c_2$, which are not poles of $f$ as $f(c_1) = p_1, f(c_2) = p_2 \in \partial K$ but $\infty \in \Ext{K}$. Similarly as in Corollary~\ref{non_sing_pole_cor}, the critical points $c_1, c_2$ along with the $(d_f-n_\Omega-1)$ critical points of $f$ in $f^{-1}(\infty)\subset f^{-1}(\widehat{T_{K_0}^\infty}(\sigma))$ (cf. Section~\ref{pole_crit_points_subsec}) yield $(d_f-n_\Omega+1)$ critical points of $f$ (counted with multiplicity) in $\displaystyle f^{-1}(\partial K\cap\mathcal{S})\sqcup f^{-1}(\widehat{T_{K_0}^\infty}(\sigma))$.
\end{proof}
\begin{proof}[Proof of Part $\mathrm{(2)}$]
Once again, the two cusps $p_1, p_2\in \partial K$ give rise to two distinct critical points $c_1, c_2$ of $f$ with $f(c_i)=p_i$, $i\in\{1,2\}$. Further, $c_1, c_2$ are not poles of $f$ as $f(c_1) = p_1, f(c_2) = p_2 \in \partial K$ but $\infty \in \Int{K}$.
Hence, counting $c_1$ and $c_2$ with the $(d_f-n_\Omega)$ critical points of $f$ in $f^{-1}(\infty)\subset f^{-1}(\Int K)\subset f^{-1}(\widehat{T_{K_0}^\infty}(\sigma))$ gives us the desired $(d_f-n_\Omega+2)$ critical points. 
\end{proof}

\begin{lem}[The case of one cusp]\label{one_cusp_lem}
Let $K$ be a component of $\Omega^\complement$ such that $\partial K$ has no double points and exactly one cusp. Then, there are at least $3$ critical points of $f$ (counted with multiplicity) in $\displaystyle f^{-1}(\partial K\cap\mathcal{S})\sqcup f^{-1}(T_{K_0}^\infty(\sigma))$.
\end{lem}
\begin{proof}
In this case,  $f$ has a critical value given by the cusp on $\partial K$, and our goal is to associate two more critical points of $f$ with $K$.
\smallskip

\textbf{Step I: The desingularized droplet contains a critical value of $f$.}
Let $p$ be the cusp on $\partial K$. We have the algebraic droplet $K$ and consider the desingularized droplet $K_0 = K\setminus\{p\}$. Let $\gamma := \partial K_0 \setminus \{p\}$. Then, $\gamma$ is a non-singular real-analytic curve (see Figure~\ref{fig:4.5}).

Let $K_1$ be the component of $\sigma^{-1}(K_0)$ adjacent to $K_0$. Since $d_f\geq 3$, we have that $K_0\cup K_1\subsetneq\widehat{\C}\setminus\{p\}$. By Proposition~\ref{schwarz_deg_prop}, $\sigma:\sigma^{-1}(\Int{\Omega^\complement})\to\Omega^\complement$ is a branched covering. Let $d_1$ be the degree of $\sigma : \Int{K_1} \to \Int{K_0}$. As every point on $\gamma$ has at least two pre-images under $\sigma$; one on $\partial K_1 \setminus \partial K_0$ and another on itself, it follows that $\sigma : \partial K_1 \to \partial K_0$ has degree at least two and hence $d_1 \geq 2$. By the Riemann-Hurwitz formula, we have:
\[\#_m\ \mathrm{crit}(\sigma\vert_{\Int{K_1}}) = d_1\cdot\chi(\Int{K_0}) - \chi(\Int{K_1}) = d_1 - \chi(\Int{K_1}) \geq 2 - 1 = 1.\]
So, there is at least one critical point of $\sigma$ in $\Int{K_1}$, and hence at least one critical point of $f$ in $f^{-1}(\Int{K_0})$.
\smallskip

\textbf{Step II: Simple connectedness of $K_1$ and the degree of $\sigma$ on $K_1$.}
We assume by way of contradiction that there is no critical point of $f$, other than the one already accounted for, in $f^{-1}(T_{K_0}^\infty(\sigma))$; or equivalently, $\sigma$ has a unique, simple critical point in $\Int{K_1}$ and no further critical point in in $T_{K_0}^\infty(\sigma)$.

If $K_1$ is not simply connected or if $d_1 \geq 3$, then there are at least two critical points of $\sigma$ in $K_1$, which contradicts our hypothesis. Thus, $K_1$ is simply connected and $d_1 = 2$. Therefore, $\sigma: \Int{K_1} \xrightarrow{2:1} \Int{K_0}$ is a degree two branched covering. Further, $\partial K_1\setminus\{p\}$ is non-singular.
\smallskip

\textbf{Step III: Construction of simply connected rank $n$ tiles and successive homeomorphisms.}
Let $K_2$ be the component of $\sigma^{-1}(K_1)$ adjacent to $K_1$ and $d_2$ be the degree of $\sigma : \Int{K_2} \to \Int{K_1}$. Once again, the Riemann-Hurwitz formula gives:
\[
\#_m\ \mathrm{crit}(\sigma\vert_{\Int{K_2}}) = d_2\cdot\chi(\Int{K_1}) - \chi(\Int{K_2}) = d_2 - \chi(\Int{K_2}).
\]
Since $K_2$ does not contain any critical point of $\sigma$, we must have $d_2 = \chi(K_2) = 1$ (note that $d_2 \geq 1$ and $\chi(\Int{K_2}) \leq 1$). So, $K_2$ is simply connected, $\sigma : \Int{K_2} \xrightarrow{1:1} \Int{K_1}$ is a conformal isomorphism, and $\partial K_2\setminus\{p\}$ is non-singular.

Proceeding thus, we obtain a sequence of simply connected domains $\{\Int{K_n}\}_{n \geq 1}$ and a sequence of homeomorphisms $\{\sigma : K_{n+1} \to K_n\}_{n \geq 1}$ (see Figure~\ref{fig:4.5}).
\begin{figure}[ht]
\captionsetup{width=0.98\linewidth}
    \includegraphics[width=0.72\linewidth]{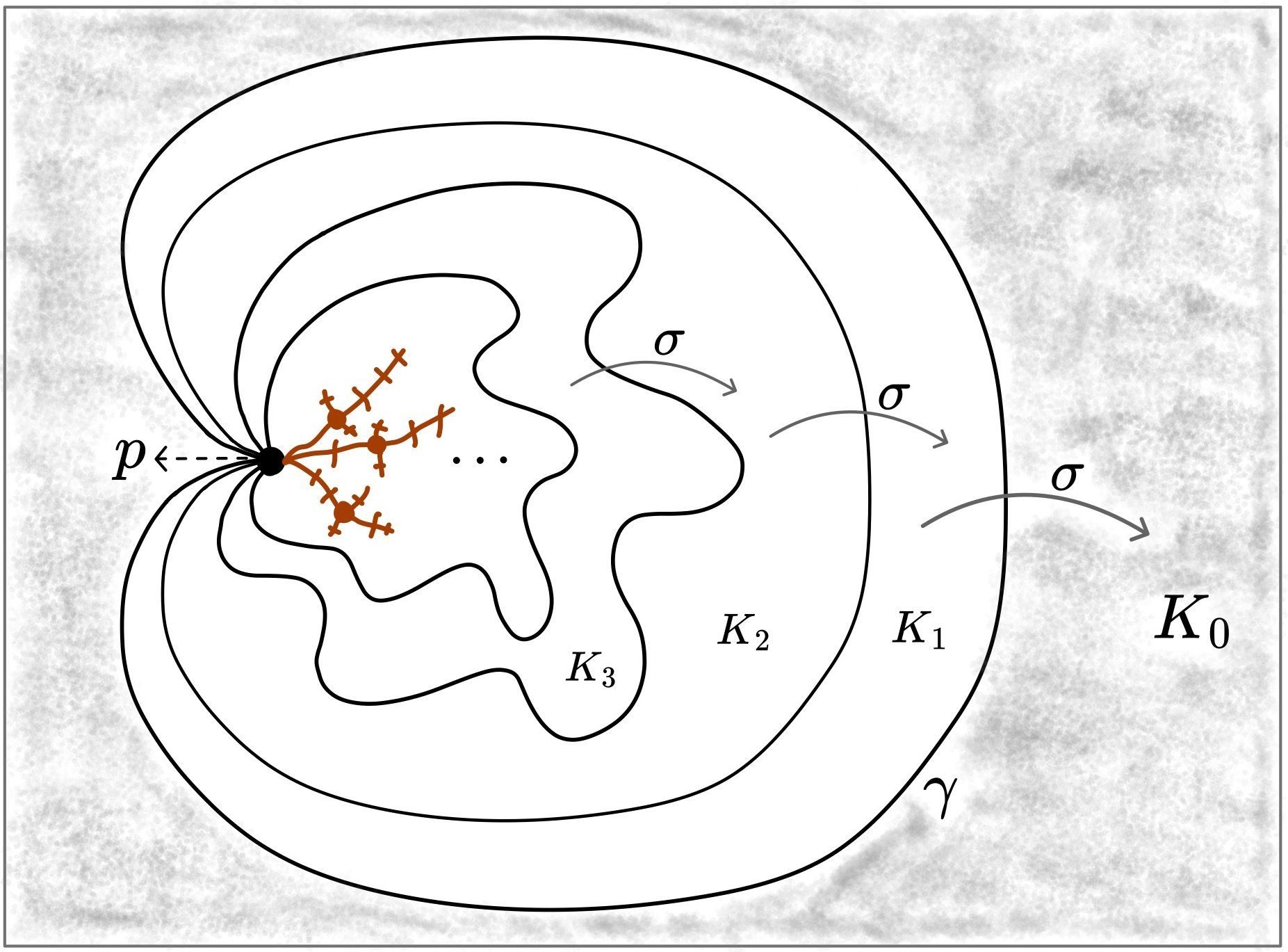}
    \caption{Illustrated is the dynamics of the Schwarz reflection $\sigma$ on the simply connected tiling component $T_{K_0}^\infty(\sigma)$, whose boundary is drawn in brown. The gray region is the rank $0$ tile $K_0$ that has a unique cusp on its boundary. The white region is the domain $V$, which maps into itself under a branch $g$ of $\sigma^{-2}$.}
    \label{fig:4.5}
\end{figure}
\smallskip

\textbf{Step IV: The simply connected, forward-invariant tiling component $T_{K_0}^\infty(\sigma)$ and a contraction mapping.}
Evidently,
\[
T^\infty_{K_0}(\sigma) = \bigcup\limits_{n \geq 0}K_n
\]
is a simply connected, forward invariant component of the tiling set $T^\infty(\sigma)$.
Observe that $\sigma^{\circ2} : K_{n+2} \to K_n$ is an isomorphism for all $n \geq 1$.
We define
\[
U := T_{K_0}^\infty(\sigma)\setminus\left(K_0 \cup K_1 \cup K_2\right) = \Int{\bigcup\limits_{n \geq 3}K_n},
\]
and
\[
V := T_{K_0}^\infty(\sigma)\setminus K_0 = \Int{\bigcup\limits_{n \geq 1}K_n} \supsetneq U.
\]
Evidently, $\sigma^{\circ2} : U \to V$ is a biholomorphism. Let $g := \sigma^{-2} \colon V \to U\subsetneq V$. As in the proof of Lemma~\ref{two_cusp_lem}, the map $g$ is a strict contraction with respect to the hyperbolic metric on $V$.
\smallskip

\textbf{Step V: The conformal type of $T_{K_0}^\infty(\sigma)$.}
\smallskip

\noindent\textbf{Claim: $T_{K_0}^\infty(\sigma) \subsetneq \widehat{\C} \setminus \{p\}$ and hence $T_{K_0}^\infty(\sigma)$ is a hyperbolic surface. In particular, $\partial T_{K_0}^\infty(\sigma)\setminus\{p\} \neq\emptyset$.}
\begin{proof}[Proof of Claim]
Suppose that $T_{K_0}^\infty(\sigma) = \widehat{\C} \setminus \{p\}$. Then $T_{K_0}^\infty(\sigma)$ is the entire tiling set $T^\infty(\sigma)$ of $\sigma$, and in particular, $\Omega^\complement = K$. It follows that $\Omega = T^\infty(\sigma) \setminus K_0 = V$ and $\sigma^{-1}(\Omega) = T^\infty(\sigma)\setminus\left(K_0 \cup K_1\right) = \Int{\bigcup\limits_{n \geq 2}K_n}$. Hence, $\sigma : \sigma^{-1}(\Omega) \xrightarrow{1:1} \Omega$ is a biholomorphism and $d_f = 2$, which contradicts our hypothesis. Thus, $T_{K_0}^\infty(\sigma) \subsetneq \widehat{\C} \setminus p$, and consequently, $T_{K_0}^\infty(\sigma)$ is conformally equivalent to $\D$.   
\end{proof}
\smallskip

\textbf{Step VI: $p$ is a cut-point of $\partial V$.}
\smallskip

\noindent\textbf{Claim: $\partial V\setminus\{p\}$ is disconnected.}
\begin{proof}[Proof of Claim]
By construction, $\gamma=\partial K_0\setminus\{p\}$ is separated from $\partial T_{K_0}^\infty(\sigma)\setminus\{p\}$ by $\Int{K_1}\cup\{p\}$.
Using this fact, we will show that $\partial V = \partial T_{K_0}^\infty(\sigma) \cup \partial K_0$. From basic topological arguments, it follows that $\partial V \subseteq \partial T_{K_0}^\infty(\sigma) \cup \partial K_0$. So, we only need to prove the reverse containment $\partial T_{K_0}^\infty(\sigma) \cup \partial K_0 \subseteq \partial V$.
Any neighborhood $N(z)$ of a point $z (\neq p) \in \partial K_0$ contains points of $\Int{K_0}\subset\Ext{V}$ as well as points of $\Int{K_1} \subset V$, and $p \in \partial V$, so $\partial K_0 \subseteq \partial V$. Any neighborhood $N(w)$ of a point $w \in \partial T_{K_0}^\infty(\sigma)$ contains points of $\Ext {T_{K_0}^\infty(\sigma)}\subset\Ext {V}$ as well as points of $\Int{K_n} (\subset V)$ for some $n\in\N$; so $\partial T_{K_0}^\infty(\sigma) \subseteq \partial V$. Hence, we have $\partial T_{K_0}^\infty(\sigma) \cup \partial K_0 \subseteq \partial V$.

Further, $p \in \partial T_{K_0}^\infty(\sigma) \cap \partial K_0$. Once again, the separation of $\gamma$ from $\partial T_{K_0}^\infty(\sigma)\setminus\{p\}$ implies that $p$ is a cut point of $\partial V= \partial T_{K_0}^\infty(\sigma) \cup \partial K_0$.
\end{proof}
\smallskip

\textbf{Step VII: Dynamics at $p$.}
For $\epsilon>0$ sufficiently small, the domain $B(p,\epsilon)\cap V$ (here $B(p,\epsilon)$ is a Euclidean disk) has at least two components (as $p$ is a cut-point of $\partial V$). Let $U_1, U_2$ be the components of $B(p,\epsilon)\cap V$ intersecting $K_1$.
Then $U_j\subset T_{K_0}^\infty(\sigma)$, $U_j$ is forward-invariant under $g$ (this follows from the local dynamics of $\sigma^{\circ 2}$ near cusps), and $p\in\partial g(U_j)$, $j\in\{1,2\}$.
By Lemma~\ref{cusp_tiling_dyn_lem}, there exist $x_j\in U_j$, such that $\{g^{\circ n}(x_j)\}_{n\geq 0}$ converges to $p$ through $U_j$, $j\in\{1,2\}$ (see Figure~\ref{fig:4.7}). 
\smallskip

\textbf{Step VIII: Dynamics on the disk and contradicting hyperbolic contraction.}
Observe that $V (\subsetneq T_{K_0}^\infty(\sigma))$ is an increasing union of simply connected domains and hence is a simply connected hyperbolic surface.

Now, consider a simple loop $\Gamma$ in $\overline{K_1}$ based at $q\in\Int{K_1}$, such that $\Gamma$ passes through $p$ and $\Gamma \setminus \{p\} \subseteq\Int{K_1}$ (see Figure~\ref{fig:4.7}). Let $\varphi : \D \to V$ be a Riemann uniformizing map, normalized to send $0$ to $q$. As $\Gamma$ passes through $p$, it intersects both $U_1$ and $U_2$. Let $\Gamma_k$ be the segment of $\Gamma$ from $q$ to $p$ passing through $U_k$, $k\in\{1,2\}$. By \cite[Theorem~4.18]{BF14}, $\varphi^{-1}\circ\Gamma_k$ lands at some point $\tau_k\in\partial\D$.

We claim that $\tau_1 \neq \tau_2$. To this end, note that $\Gamma_1$ and $\Gamma_2$ are not homotopic in $V$; for if $\Gamma_1 \sim_V \Gamma_2$, then $\Gamma = \Gamma_1 \cup \Gamma_2$ would bound a Jordan domain in $V$, which contradicts the fact that $p$ is a cut-point of $\partial V$ (see Step~VI). Hence, $\Gamma_1$ and $\Gamma_2$ do not lie in same access of $V$ to $p$, and hence by \cite[Theorem~4.18]{BF14}, $\tau_1 \neq \tau_2$.
\begin{figure}[ht]
\captionsetup{width=0.98\linewidth}
    \includegraphics[width=1\linewidth]{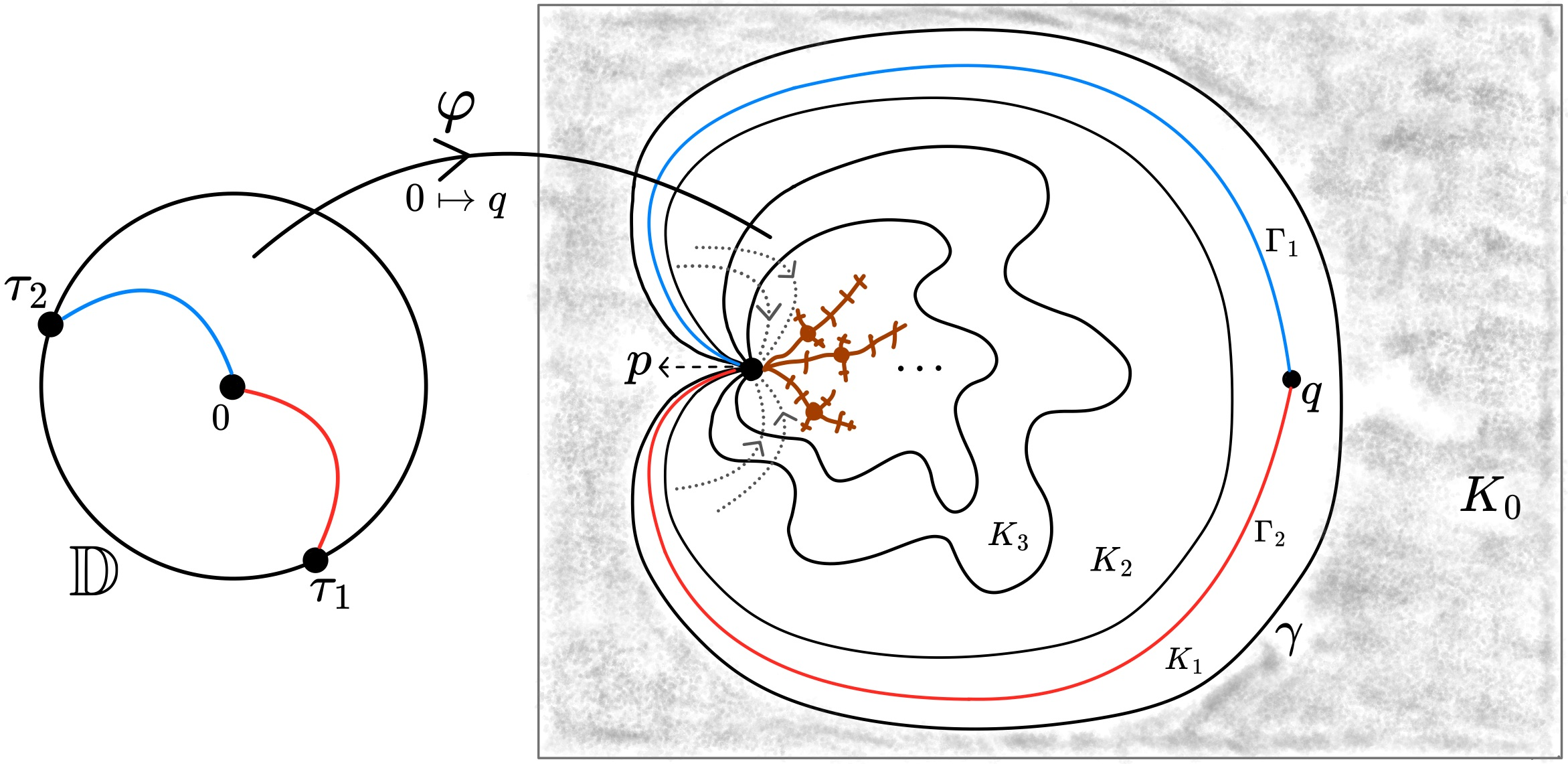}
    \caption{Pictured is the loop $\Gamma=\Gamma_1\cup\Gamma_2\subset\overline{K_1}$ passing through the cusp $p$ and an interior point $q$ of $K_1$. Under the conformal map $\phi:\D\to V$ (where $V$ is the white region), the loop $\Gamma$ lifts to two distinct accesses to $\partial\D$. Some $g-$orbits (where $g$ is a branch of $\sigma^{-2}$ preserving $V$) converging to $p$ are shown in gray. Under the uniformization $\phi$, these orbits give rise to orbits of $\widetilde{g}=\varphi^{-1} \circ g \circ \varphi:\D\to\D$ that converge to two distinct points on $\partial\D$.}
    \label{fig:4.7}
\end{figure}
\[
\begin{tikzcd}
(\D,d_\D) \arrow[r,"\widetilde{g}"] \arrow[d,swap,"\varphi"]
& (\D,d_\D) \arrow[d,"\varphi"] \\
(V,d_V) \arrow[r,swap,"g"]
& (V,d_V)
\end{tikzcd}
\]
Finally, consider the holomorphic self-map $\widetilde{g} := \varphi^{-1} \circ g \circ \varphi:\D\to\D$ and two distinct points $y_k := \varphi^{-1}(x_k)$, $k\in\{ 1,2\}$. Since $g^{\circ n}(x_k)\xrightarrow{n\to\infty} p$ through $U_k$, it now follows that $\widetilde{g}^{\circ n}(y_k) \xrightarrow{n\to\infty} \tau_k$, $k\in\{1,2\}$. As in Step~VI of Lemma~\ref{two_cusp_lem}, this is a contradiction to hyperbolic contraction of $\widetilde{g}$.
\end{proof}

\begin{cor}\label{one_cusp_pole_cor}
Let $K$ be as in Lemma~\ref{one_cusp_lem}, and assume that one of the two following conditions holds true.
\noindent\begin{enumerate}
    \item $\Omega$ is unbounded and $\infty\in\widehat{T_{K_0}^\infty}(\sigma)$. 
    \item $\Omega$ is bounded and $\infty\in\Int{K}$.
\end{enumerate} 
Then, at least $\max\{d_f-n_\Omega+1,3\}$ critical points of $f$ (counted with multiplicity) lie in $\displaystyle f^{-1}(\partial K\cap\mathcal{S})\sqcup f^{-1}(\widehat{T_{K_0}^\infty}(\sigma))$.
\end{cor}
\begin{proof}[Proof of $\mathrm{(1)}$.]
The map $f$ has a critical value given by the cusp on $\partial K$ and hence an associated critical point, say $c_1$, which is not a pole of $f$ as $f(c_1) = p \in \partial K$ but $\infty \in \Ext{K}$. By Step I of the proof of Lemma~\ref{one_cusp_lem}, we obtain a critical point of $\sigma$ in $\Int{K_1}$, and this provides us with a critical point $c_2$ of $f$ in $f^{-1}(\Int K)$. Clearly, $c_2$ is not a pole of $f$ as $f(c_2)\in\Int{K}$ but $\infty \in \Ext{K}$.
Hence, the points $c_1, c_2$, and the $(d_f-n_\Omega-1)$ critical points of $f$, counted with multiplicity (see Section~\ref{pole_crit_points_subsec}), in $f^{-1}(\infty)\subset f^{-1}(\widehat{T_{K_0}^\infty}(\sigma))$ account for the $(d_f-n_\Omega+1)$ critical points of $f$ in $\displaystyle f^{-1}(\partial K\cap\mathcal{S})\sqcup f^{-1}(\widehat{T_K^\infty}(\sigma))$.
\end{proof}

\begin{proof}[Proof of $\mathrm{(2)}$.]
Once again, the cusp $p \in \partial K$ provides a critical point $c_1$ of $f$, which is not a pole of $f$ as $f(c_1) = p \in \partial K$ but $\infty \in \Int{K}$.
Hence, counting $c_1$ with the $(d_f-n_\Omega)$ critical points of $f$ in $f^{-1}(\infty)\subset f^{-1}(\Int K)\subset f^{-1}(\widehat{T_{K_0}^\infty}(\sigma))$ gives us the desired $(d_f-n_\Omega+1)$ critical points.
\end{proof}

\subsubsection{The case of double points}\label{4.2.3}

We begin with a preparatory graph-theoretic result.

\begin{lem}\label{tree_counting_lem}
Let $\mathscr{T}$ be a tree. Let $n := \#$ edges of $\mathscr{T}$, and $n_i := \#$ vertices of $\mathscr{T}$ of valence $i$, where $i\in\{1,2\}$. Then,
\begin{equation}\label{gt_lem_eqn}
2n_1+n_2 \geq n+3.
\end{equation}
Further, the bound is sharp and is attained by the chain tree.
\end{lem} 
\begin{proof}
We will use induction to prove the lemma.

First, note that the statement is true for $n=1$, for in that case, $n_1=2$ and $n_2=0$ and hence
\[2n_1+n_2 = 4 = n+3.\]
Now, suppose that the statement is true for any tree with $n=\alpha$.
We will show that the Inequality~\eqref{gt_lem_eqn} holds for $n=\alpha+1$. Let $\mathscr{T}$ be a tree with $\alpha+1$ edges.
Clearly, we have $n_1\geq2$. Consider a vertex $v_0$ of valence $1$ and let $\widehat{v}$ be the vertex of $\mathscr{T}$ adjacent to $v_0$.
Let $\widehat{\mathscr{T}}$ be the tree obtained by pruning the edge $[\widehat{v},v_0]$ from $\mathscr{T}$. Then, $\widehat{\mathscr{T}}$ is a tree with $\alpha$ many edges. Denote by $\widehat{n}_i := \#$ vertices of $\widehat{\mathscr{T}}$ of valence $i$. By induction hypothesis,
\begin{equation}\label{gt_pf_eqn}
2\widehat{n}_1+\widehat{n}_2\geq\alpha+3.
\end{equation}
Further, $\val_{\widehat{\mathscr{T}}}(\widehat{v}) = \val_\mathscr{T}(\widehat{v})-1$. We now consider the following cases.
\smallskip

\noindent\textbf{Case I: $\val_\mathscr{T}(\widehat{v}) = 2$.} In this case, $\widehat{n}_1=n_1$, $\widehat{n}_2=n_2-1$.
From Inequality~\eqref{gt_pf_eqn}, we have
\[
2n_1+n_2-1\geq\alpha+3 \implies 2n_1+n_2\geq(\alpha+1)+3.
\]

\noindent\textbf{Case II: $\val_\mathscr{T}(\widehat{v}) = 3$.} In this case, $\widehat{n}_1=n_1-1$, $\widehat{n}_2=n_2+1$.
From Inequality~\eqref{gt_pf_eqn}, we have
\[
2(n_1-1)+n_2+1\geq\alpha+3 \implies 2n_1+n_2\geq(\alpha+1)+3.
\]

\noindent\textbf{Case III: $\val_\mathscr{T}(\widehat{v})> 3$.} In this case, $\widehat{n}_1=n_1-1$, $\widehat{n}_2=n_2$.
Inequality~\eqref{gt_pf_eqn} now gives
\[
2(n_1-1)+n_2\geq\alpha+3 \implies 2n_1+n_2\geq \alpha
+5 >(\alpha+1)+3.
\]
In all cases, Inequality~\eqref{gt_lem_eqn} is true for $n=\alpha+1$ and we have proven the lemma.

It is easy to see that the bound is attained by chain trees as in this case, $n_1=2$, $n_2=n-1$; and hence, $2n_1+n_2=4+(n-1)=n+3$.
\end{proof}

Let $K$ be a component of $\Omega^\complement$. We will associate a tree with $K$ whose vertices correspond to the components of $\Int{K}$. The graph-theoretic lemma proved above will later be applied to this tree.

\begin{lem}\label{droplet_tree_lem}
There is a tree $\mathscr{T}(K)$ associated with $K$ such that the following are satisfied.
\begin{enumerate}
    \item There exists a planar embedding of $\mathscr{T}(K)$ which has a unique vertex in each component of $\Int{K}$ and no other vertex. Further, the said planar embedding lives in $K$.
    \item $K$ deformation retracts to the above planar embedding of $\mathscr{T}(K)$.
\end{enumerate} 
\end{lem}
\begin{proof}
Since $\partial\Omega$ is a real-algebraic curve, it has at most finitely many double points (cf. \cite{Gus88}). Let $\alpha$ be the number of double points on $\partial K$. 

Since $K$ is a complementary component of an open connected set in $\widehat{\C}$, it is full and compact. Further, each double point of $\partial K$ is a cut-point of $K$ of valence two; i.e., the removal of a double point disconnects $K$ into two components (see Figure~\ref{fig:4.8}). Hence, $\Int{K}$ has $\alpha+1$ components, each of which is a topological disk.

Consider a vertex in each component of $\Int{K}$ and draw an edge connecting the vertices corresponding to any pair of components whose boundaries intersect at a double point (see Figure~\ref{fig:4.8}). Evidently, this defines a graph $\mathscr{T}(K)$ with $\alpha+1$ vertices. Fullness of $K$ implies that the graph $\mathscr{T}(K)$ has $\alpha$ edges; i.e., it is a tree. The desired properties of $\mathscr{T}(K)$ are easily verified.
\end{proof}

\begin{figure}[ht]
\captionsetup{width=0.98\linewidth}
    \includegraphics[width=0.72\linewidth]{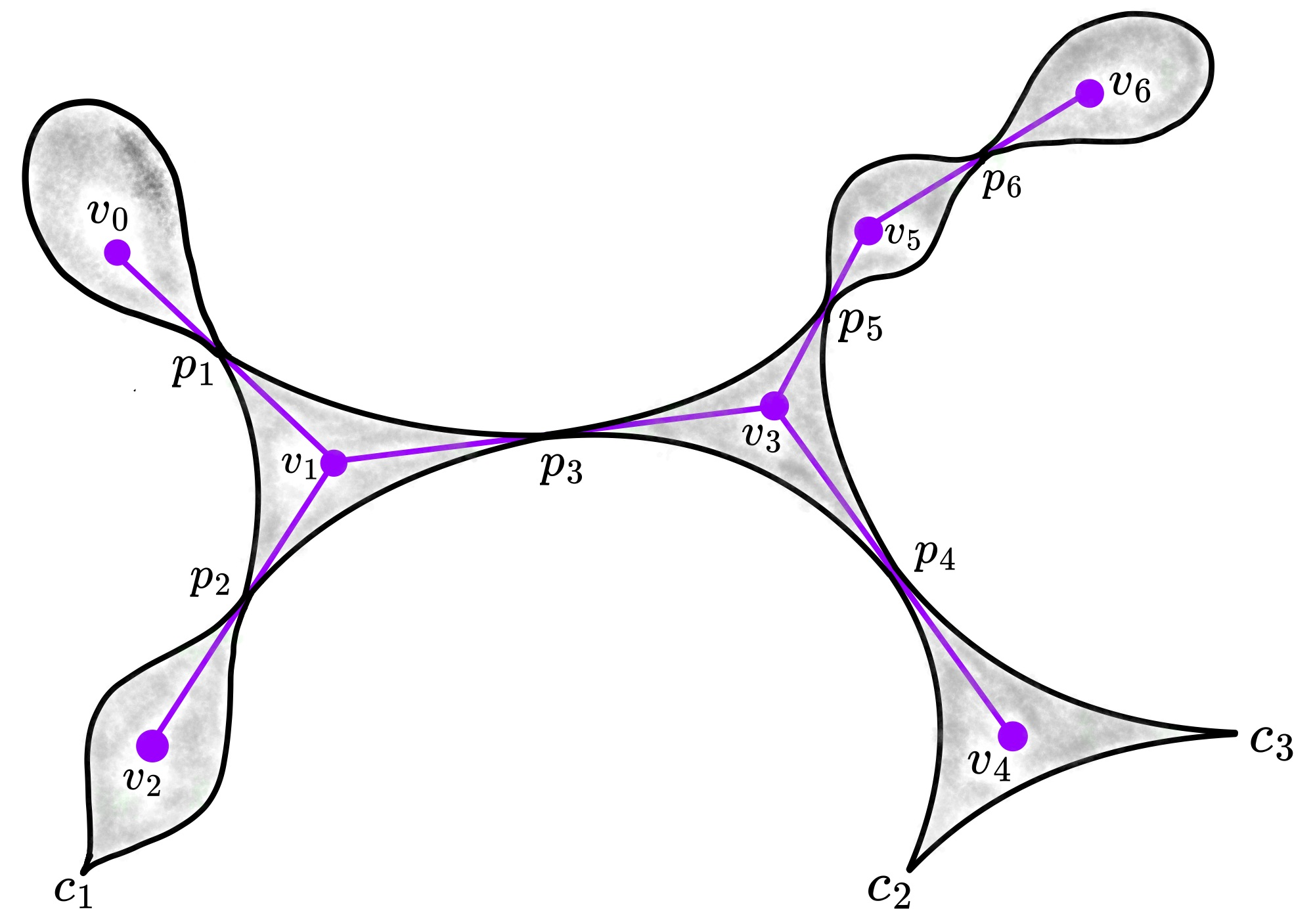}
    \caption{Depicted is the tree $\mathscr{T}(K)$ associated with a component $K$ of $\Omega^\complement$.}
    \label{fig:4.8}
\end{figure}

\begin{lem}\label{dp_lem}
Let $K$ be a component of $\Omega^\complement$ such that $\partial K$ has double points. Then, there are at least $3+\alpha$ critical points of $f$ (counted with multiplicity) in $\displaystyle f^{-1}(\partial K\cap\mathcal{S})\sqcup\bigsqcup_{j=0}^{\alpha} f^{-1}(T_{K_0^j}^\infty(\sigma))$, where $\alpha := \# D_K$, and $K_0^0,\cdots,K_0^\alpha$ are the components of the desingularization $K\setminus\mathcal{S}$.
\end{lem}
\begin{proof}
Due to the graph-theoretic Lemma~\ref{tree_counting_lem}, it is enough to consider the components of the desingularization $K\setminus\mathcal{S}$ corresponding to valence $1$ and $2$ vertices of $\mathscr{T}(K)$ to prove the above lemma. In particular, proving the following claims would provide the required lower bound on the number of critical points associated with~$K$.

\noindent\textbf{Claim I:} Let $K_0^i$ be a component of $K\setminus\mathcal{S}$ such that the vertex of $\mathscr{T}(K)$ in $\Int{K_0^i}$ has valence 2. Then, there is at least 1 critical point of $f$ in $\displaystyle f^{-1}(T_{K_0^i}^\infty(\sigma))$.

\noindent\textbf{Claim II:} Let $K_0^j$ be a component of $K\setminus\mathcal{S}$ such that the vertex of $\mathscr{T}(K)$ in $\Int{K_0^j}$ has valence 1. Then, there are at least 2 critical points of $f$ in $\displaystyle f^{-1}(T_{K_0^j}^\infty(\sigma))$.

\begin{proof}[Proof of Claim I]
Recall that $T_{K_0^i}^\infty(\sigma)$ is the component of $T^\infty(\sigma)$ containing $K_0^i$.
As in the case of exactly two cusps (Lemma~\ref{two_cusp_lem}, Step~I), we assume that there is no critical point of $f$ in $\displaystyle f^{-1}(T_{K_0^i}^\infty(\sigma))$; i.e, there is no critical value of $f$ (or equivalently of $\sigma$) in $T_{K_0^i}^\infty(\sigma)$. Then, Steps~II--IV of Lemma~\ref{two_cusp_lem} apply verbatim to the present setting to give a self-map $g$, which is an inverse branch of $\sigma^{\circ 2}$, on a simply connected subset of $T_{K_0^i}^\infty(\sigma)$ such that $g$ is a hyperbolic contraction (see Figure~\ref{fig:4.9} (right)).
\begin{figure}
\captionsetup{width=0.98\linewidth}
    \includegraphics[width=0.98\linewidth]{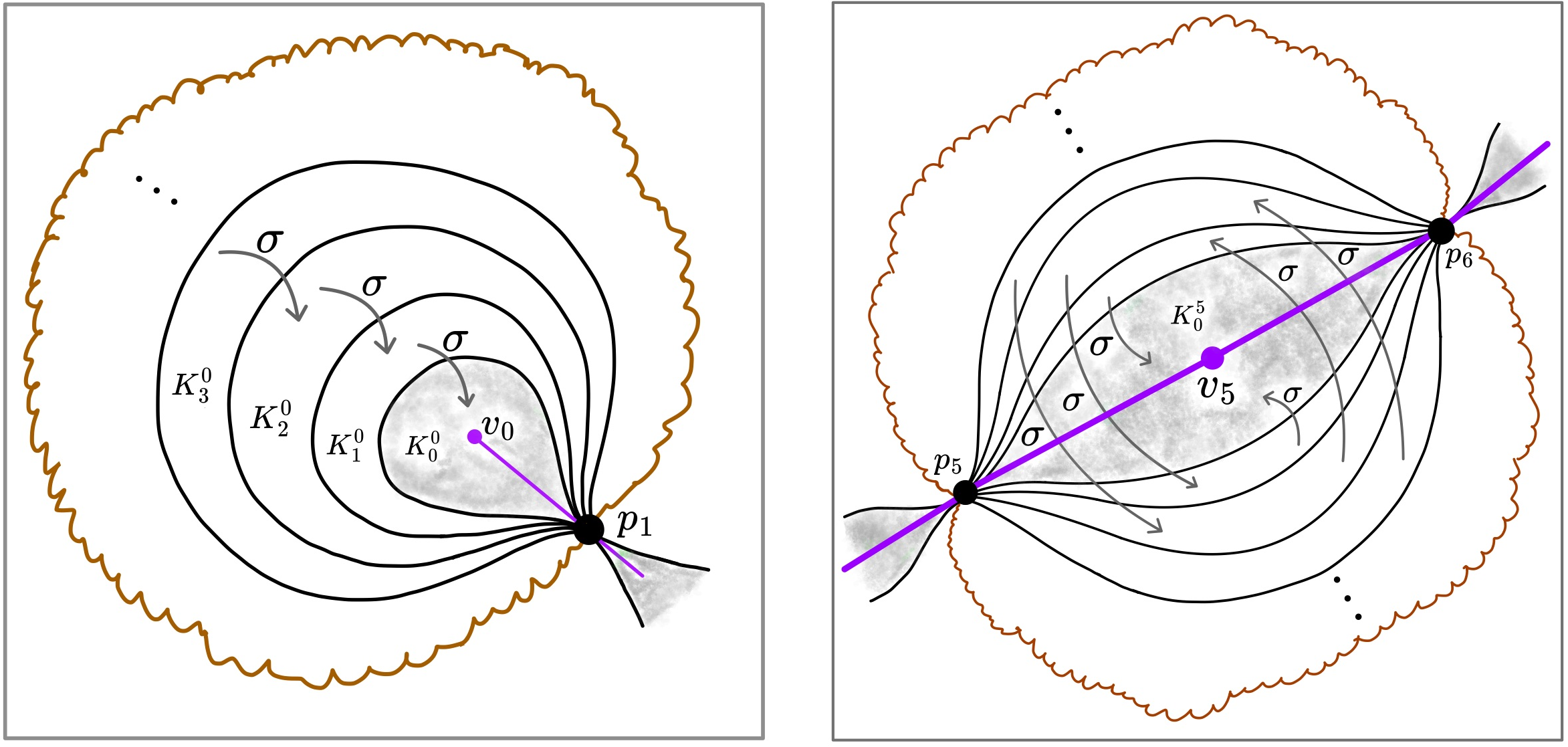} 
    \caption{Depicted is the dynamics of $\sigma$ on the component $K_0^0$ (respectively, $K_0^5$) of $K\setminus\mathcal{S}$, corresponding to the vertex $v_0$ (respectively, $v_5$) of valence $1$ (respectively, $2$) of the graph $\mathscr{T}(K)$ in Figure~\ref{fig:4.8}.}
    \label{fig:4.9}
\end{figure}
To apply the arguments of Step~V of Lemma~\ref{two_cusp_lem}, we invoke Lemma~\ref{dp_tiling_dyn_lem}, which gives us two distinct $g$-orbits converging to the two distinct double points on the boundary $\partial K_0^i$ (in Figure~\ref{fig:4.9} (right), the relevant component is $K_0^5$ and the double points on $\partial K_0^5$ are $p_5,p_6$). This contradicts the hyperbolic contraction of $g$ as explained in Step~VI of Lemma~\ref{two_cusp_lem}.
\end{proof}

\begin{proof}[Proof of Claim II]
The proof follows the scheme of the case of exactly one cusp (see Lemma~\ref{one_cusp_lem}).
As in Steps~I-II of that lemma, we have at least one critical value of $\sigma$, and hence of $f$, in $\Int{K_0}$, and assume that there is no other critical point of $f$ in $f^{-1}(T_{K_0^j}^\infty(\sigma))$. 
We now consider the dynamics of $\sigma$ on the component $T_{K_0^j}^\infty(\sigma)$ of the tiling set (see Figure~\ref{fig:4.9} (left)), and note that the Steps~III--VIII of Lemma~\ref{one_cusp_lem} can be applied to the situation at hand with exactly one modification: we need to appeal to Lemma~\ref{dp_tiling_dyn_lem} (instead of Lemma~\ref{cusp_tiling_dyn_lem}) in Step~VII.
This provides us with a contraction mapping on the hyperbolic disk with two orbits converging to two distinct points on $\partial\D$. This is a contradiction, which proves the claim.
\end{proof}
\end{proof}

\begin{cor}\label{dp_pole_cor}
Let $K$ be as in Lemma~\ref{dp_lem}, and assume that one of the two following conditions holds true.
\noindent\begin{enumerate}
    \item $\Omega$ is unbounded and $\displaystyle\infty\in\bigsqcup_{j=0}^\alpha\widehat{T_{K_0^j}^\infty}(\sigma)$. 
    \item $\Omega$ is bounded and $\displaystyle\infty\in\bigsqcup_{j=0}^\alpha\Int{K_0^j}$.
\end{enumerate} 
Then, at least $\max\{\alpha + d_f-n_\Omega+1,\alpha+3\}$ critical points of $f$ (counted with multiplicity) lie in $\displaystyle f^{-1}(\partial K\cap\mathcal{S})\sqcup\bigsqcup_{j=0}^\alpha f^{-1}(\widehat{T_{K_0^j}^\infty}(\sigma))$.
\end{cor}

\begin{proof}[Proof of Part~$\mathrm{(1)}$]
Let $K_0^i$ be the component of the desingularization $K\setminus\mathcal{S}$ such that $\infty\in\widehat{T_{K_0^i}^\infty}(\sigma)$ and $v_i$ be the vertex of the tree $\mathscr{T}(K)$ corresponding to $\Int{K_0^i}$ (see Lemma~\ref{droplet_tree_lem}). We consider the following sub-cases.

\noindent\textbf{Case 1: $v_i$ has valence at least $3$.}
Referring to the proof of Lemma~\ref{dp_lem}, as the component $K_0^i$ corresponds to a vertex of valence at least $3$ in the tree $\mathscr{T}(K)$, the meromorphic map $f$ has at least $3+\alpha$ critical points (counted with multiplicity) in $\displaystyle f^{-1}(\partial K\cap\mathcal{S})\sqcup\bigsqcup_{\substack{j=0\\ j \neq i}}^{\alpha} f^{-1}(T_{K_0^j}^\infty(\sigma))$. None of these critical points are poles of $f$; indeed, as $\infty\in\widehat{T_{K_0^i}^\infty}(\sigma)$, the poles of $f$ lie in $f^{-1}\left(\widehat{T_{K_0^i}^\infty}(\sigma)\right)$. Counting them along with the $(d_f-n_\Omega-1)$ critical points of $f$ in $f^{-1}(\infty)$ (cf. Section~\ref{pole_crit_points_subsec}) yield $(\alpha + d_f-n_\Omega+2)$ critical points of $f$ (counted with multiplicity) in $\displaystyle f^{-1}(\partial K\cap\mathcal{S})\sqcup\bigsqcup_{j=0}^\alpha f^{-1}(\widehat{T_{K_0^j}^\infty}(\sigma))$.

\noindent\textbf{Case 2: $v_i$ has valence $2$.}
Similarly as in Case 1, we refer to the proof of Lemma~\ref{dp_lem}. As the component $K_0^i$ corresponds to a vertex of valence $2$ in the tree $\mathscr{T}(K)$, the map $f$ has at least $(3+\alpha)-1 = \alpha+2$ critical points (counted with multiplicity) in $\displaystyle f^{-1}(\partial K\cap\mathcal{S})\sqcup\bigsqcup_{\substack{j=0\\ j \neq i}}^{\alpha} f^{-1}(T_{K_0^j}^\infty(\sigma))$, none of which are poles of $f$. Counting them with the $(d_f-n_\Omega-1)$ critical points of $f$ in $f^{-1}(\infty)\subset f^{-1}\left(\widehat{T_{K_0^i}^\infty}(\sigma)\right)$ yield $(\alpha + d_f-n_\Omega+1)$ critical points of $f$ (counted with multiplicity) in $\displaystyle f^{-1}(\partial K\cap\mathcal{S})\sqcup\bigsqcup_{j=0}^\alpha f^{-1}(\widehat{T_{K_0^j}^\infty}(\sigma))$.

\noindent\textbf{Case 3: $v_i$ has valence $1$.}
Again, we refer to the proof of Lemma~\ref{dp_lem}. As the component $K_0^i$ corresponds to a vertex of valence $1$ in the tree $\mathscr{T}(K)$, the map $f$ has at least $(3+\alpha)-2 = \alpha+1$ critical points (counted with multiplicity) in $\displaystyle f^{-1}(\partial K\cap\mathcal{S})\sqcup\bigsqcup_{\substack{j=0\\ j \neq i}}^{\alpha} f^{-1}(T_{K_0^j}^\infty(\sigma))$, none of which are poles of $f$. By Claim II in the proof of Lemma~\ref{dp_lem}, we obtain a critical point of $\sigma$ in $\Int{K_1^i}$, which provides us with a critical point $c$ of $f$ in $f^{-1}(\Int K_0^i)$. Clearly, $c$ is not a pole of $f$ as $f(c)\in\Int{K_0^i}$ but $\infty \in \Ext{K_0^i}$. Counting the critical points obtained so far with the $(d_f-n_\Omega-1)$ critical points of $f$ in $f^{-1}(\infty)\subset f^{-1}\left(\widehat{T_{K_0^i}^\infty}(\sigma)\right)$ yield $(\alpha + d_f-n_\Omega+1)$ critical points of $f$ (counted with multiplicity) in $\displaystyle f^{-1}(\partial K\cap\mathcal{S})\sqcup\bigsqcup_{j=0}^\alpha f^{-1}(\widehat{T_{K_0^j}^\infty}(\sigma))$.
\end{proof}

\begin{proof}[Proof of Part~$\mathrm{(2)}$]
Let $K_0^i$ be the component of the desingularization $K\setminus\mathcal{S}$ such that $\infty\in\Int{K_0^i}$, and let $v_i$ be the vertex of the tree $\mathscr{T}(K)$ corresponding to $K_0^i$. As in the Part~(1), we consider the following sub-cases. 

\noindent\textbf{Case 1: $v_i$ has valence at least $3$.}
Arguing as in the unbounded case (using the proof of Lemma~\ref{dp_lem}), we can find at least $3+\alpha$ critical points of $f$ (counted with multiplicity) in $\displaystyle f^{-1}(\partial K\cap\mathcal{S})\sqcup\bigsqcup_{\substack{j=0\\ j \neq i}}^{\alpha} f^{-1}(T_{K_0^j}^\infty(\sigma))$ and $(d_f-n_\Omega)$ critical points of $f$ in $f^{-1}(\infty)\subset f^{-1}\left(\widehat{T_{K_0^i}^\infty}(\sigma)\right)$ (cf. Section~\ref{pole_crit_points_subsec}). This yields $(\alpha + d_f-n_\Omega+3)$ critical points of $f$ (counted with multiplicity) in $\displaystyle f^{-1}(\partial K\cap\mathcal{S})\sqcup\bigsqcup_{j=0}^\alpha f^{-1}(\widehat{T_{K_0^j}^\infty}(\sigma))$.

\noindent\textbf{Case 2: $v_i$ has valence at least $2$.}
Similarly as in Case 1 (referring to the proof of Lemma~\ref{dp_lem}), the map $f$ has at least $\alpha+2$ critical points (counted with multiplicity) in $\displaystyle f^{-1}(\partial K\cap\mathcal{S})\sqcup\bigsqcup_{\substack{j=0\\ j \neq i}}^{\alpha} f^{-1}(T_{K_0^j}^\infty(\sigma))$, and $(d_f-n_\Omega)$ critical points in $f^{-1}(\infty)\subset f^{-1}\left(\widehat{T_{K_0^i}^\infty}(\sigma)\right)$. Thus, we obtain $(\alpha + d_f-n_\Omega+2)$ critical points of $f$ (counted with multiplicity) in $\displaystyle f^{-1}(\partial K\cap\mathcal{S})\sqcup\bigsqcup_{j=0}^\alpha f^{-1}(\widehat{T_{K_0^j}^\infty}(\sigma))$.

\noindent\textbf{Case 3: $v_i$ has valence at least $1$.}
As in the third case of Part~(1) of this corollary, the map $f$ has at least $\alpha+1$ critical points (counted with multiplicity) in $\displaystyle f^{-1}(\partial K\cap\mathcal{S})\sqcup\bigsqcup_{\substack{j=0\\ j \neq i}}^{\alpha} f^{-1}(T_{K_0^j}^\infty(\sigma))$, and $(d_f-n_\Omega)$ critical points in $f^{-1}(\infty)\subset f^{-1}\left(\widehat{T_{K_0^i}^\infty}(\sigma)\right)$. This gives the desired number of critical points; i.e., $(\alpha + d_f-n_\Omega+1)$ (counted with multiplicity), in $\displaystyle f^{-1}(\partial K\cap\mathcal{S})\sqcup\bigsqcup_{j=0}^\alpha f^{-1}(\widehat{T_{K_0^j}^\infty}(\sigma))$.
\end{proof}

\subsubsection{Proof of Proposition~\ref{individual_contribution_prop}}
Let $K$ be a component of $\Omega^\complement$. 

If $K$ is non-singular, then $\partial K\cap\mathcal{S}=\emptyset$, and hence $\# D_K=0$ and $K\setminus\mathcal{S}=K$. In this case, the existence of the desired critical points of $f$ follows from Lemma~\ref{non_sing_lem}.

If $\partial K$ has cusps, but no double points, then $\# D_K=0$ and $K_0:=K\setminus\mathcal{S}$ is connected. As mentioned in Section~\ref{4.2.2}, if there are at least three cusps on $\partial K$, then the conclusion of Proposition~\ref{individual_contribution_prop} is trivially satisfied for the component $K$. On the other hand, if $K$ has at most two cusps, then Lemmas~\ref{two_cusp_lem} and~\ref{one_cusp_lem} provide us with the required critical points of $f$.

Finally, if $\partial K$ has at least one double point, then the result follows from Lemma~\ref{dp_lem}.

\section{Proofs of the main theorems}\label{pf_main_thm_sec}

\subsection{Bounding connectivity and double points}\label{pf_main_thm_1_subsec}

In this subsection, we will prove Theorem~\ref{main_thm_1} as an improvement of Theorem~\ref{special_case_thm}, by taking node data and orders of singular points into account. Our goal here is to obtain the following inequality.
\begin{equation}
\mathrm{conn}(\Omega) + \# D + \sum_{p\in D} \delta_p + \sum_{p\in C} \delta_p \leq 
\min \{d_f+n_\Omega-2, 2d_f-4\}.
\label{main_thm_1_eqn}
\end{equation}
We will further improve the above upper bound in the case that $\Omega$ has a node at $\infty$ as follows:
\begin{equation}
\mathrm{conn}(\Omega) + \# D + \sum_{p\in D} \delta_p + \sum_{p\in C} \delta_p \leq d_f +n_\Omega - 3.
\label{infty_case_eqn}
\end{equation}
Recall that,
\[\delta_p=
\begin{cases}
\lfloor n/4\rfloor, & \text{ if $p$ is a cusp of type } (n,2),\\
\lfloor n/2\rfloor, & \text{ is $p$ a double point with contact order } n,
\end{cases}
\]
in the above inequality.

Observe that
$$
d_f+n_\Omega-2\leq 2d_f-4 \iff n_\Omega\leq d_f-2.
$$
In light of this observation, we will assume that $n_\Omega\leq d_f-2$ in the proof of Theorem~\ref{main_thm_1}. Before proceeding with the proof, we introduce the following sets of critical points of $f$.

\[
\crit_C(f) \coloneq \{c\in\crit(f)\ |\ f(c)\ \textrm{is a cusp on } \partial\Omega\}.
\]

\begin{align*}
&\crit_\cS(f) \coloneq\\ 
& \{c\in\crit(f)\ |\ \{\sigma^{\circ n}(f(c))\}_n\ \textrm{ is an infinite sequence, and }\lim_{n\rightarrow\infty}\sigma^{\circ n}(f(c))\in\cS\},
\end{align*}
which is the set of all critical points of $f$ whose corresponding critical values converge non-trivially to a point in $\cS$ under iterates of $\sigma$.

\[
\crit_T(f) \coloneq \{c\in\crit(f)\ |\ f(c)\in\ T^\infty(\sigma)\}.
\]

\[
\crit_P(f) \coloneq \{c\in\crit(f)\ |\ f(c)=\infty\}.
\]
In what follows, the notation $\#_m \mathrm{crit}_\bullet(f)$ will stand for the number of critical points of $f$ in $\mathrm{crit}_\bullet(f)$, counted with multiplicity.
The next lemma establishes certain relations among the above types of critical points. 
\begin{lem}\label{crit_relation_lem}
The sets $\crit_\cS(f)$, $\crit_T(f)$, and $\crit_C(f)$ are pairwise disjoint, and $\crit_P(f)\cap\crit_C(f)=\emptyset$. Further, we have the following assertions.
\begin{enumerate}
\item If $\infty\in\mathscr{K}(\sigma)$, then the sets $\crit_P(f)$ and $\crit_T(f)$ are disjoint. Further, if $n_\Omega\leq d_f-2$, then 
$$
\#_m\left(\crit_P(f)\cup\crit_\cS(f)\right)\geq \left(d_f-n_\Omega-2\right)+\left(\sum_{p\in D} \delta_p + \sum_{p\in C} \delta_p\right).
$$
Additionally, if $\Omega$ has a node at $\infty$, then $\crit_P(f)\cap\crit_\cS(f)=\emptyset$, and hence
$$
\#_m\left(\crit_P(f)\cup\crit_\cS(f)\right)\geq \left(d_f-n_\Omega-1\right)+\left(\sum_{p\in D} \delta_p + \sum_{p\in C} \delta_p\right).
$$
\item If $\infty\in T^\infty(\sigma)$, then the sets $\crit_\cS(f)$ and $\crit_P(f)$ are disjoint.
\end{enumerate}
\end{lem}
\begin{proof}
The first statement follows from the observations that $f(\crit_\cS(f))\subset \Int{\mathscr{K}(\sigma)}$, $f(\crit_T(f))\subset T^\infty(\sigma)$, $f(\crit_C(f))\subset \partial T^\infty(\sigma)\cap\partial\Omega$, and that the latter sets are mutually disjoint.
The second statement is a consequence of the fact that $f(\crit_P(f))=\{\infty\}$, and $\infty\notin\partial\Omega$.

(1) As $f(\crit_P(f))=\{\infty\}\subset \mathscr{K}(\sigma)$, $f(\crit_T(f))\subset T^\infty(\sigma)$, and the latter sets are disjoint, it follows that $\crit_P(f)\cap\crit_T(f)=\emptyset$.

Let us now assume that $n_\Omega\leq d_f-2$. By Section~\ref{pole_crit_points_subsec}, the set $\crit_P(f)$ contains at least $(d_f-n_\Omega-1)$ critical points of $f$, counted with multiplicity. By definition, these $(d_f-n_\Omega-1)$ critical points of $f$ are mapped to $\infty$ under $f$. On the other hand, by Lemmas~\ref{higher_order_cusp_lem} and~\ref{higher_order_dp_lem}, the set $\crit_{\cS}(f)$ contains $\Delta:=\left(\sum_{p\in D} \delta_p + \sum_{p\in C} \delta_p\right)$ many distinct critical points and their $f-$images lie in distinct attracting petals. Hence, at most one of these $\Delta$ critical points in $\crit_{\cS}(f)$ can lie in $\crit_P(f)$. Thus, we get at least
$$
(d_f-n_\Omega-1)+(\Delta-1)=(d_f-n_\Omega-2)+\Delta
$$
critical points in the union $\crit_P(f)\cup\crit_\cS(f)$, counted with multiplicity.

Additionally, suppose that $\Omega$ has a node at $\infty$. By way of contradiction, let $c\in\crit_P(f)\cap\crit_\cS(f)$. Then,  $\sigma(f(c))=\sigma(\infty)=\infty$, since the poles of $R_\Omega$ are also poles of $\sigma$ (cf. \cite[Lemma~3.1]{LM16}). On the other hand, as $c\in\crit_\cS(f)$, the orbit $\{\sigma^{\circ n}(\infty)\}_n$ must be infinite, which is a contradiction. Hence, $\crit_P(f)\cap\crit_\cS(f)=\emptyset$. The desired lower bound on $\#_m\left(\crit_P(f)\cup\crit_\cS(f)\right)$ now follows as in the previous paragraph (using Section~\ref{pole_crit_points_subsec} and Lemmas~\ref{higher_order_cusp_lem},~\ref{higher_order_dp_lem}).

(2) Since $f(\crit_\cS(f))\subset\mathscr{K}(\sigma)$ and $f(\crit_P(f))=\{\infty\}\subset T^\infty(\sigma)$, we conclude that $\crit_\cS(f)\cap\crit_P(f)=\emptyset$.
\end{proof}

\begin{proof}[Proof of Theorem~\ref{main_thm_1}]
We consider the cases of bounded and unbounded $\Omega$ separately. Let us set 
$$
\Delta:=\left(\sum_{p\in D} \delta_p + \sum_{p\in C} \delta_p\right).
$$
\smallskip

\noindent\textbf{Unbounded case.}
We consider the following sub-cases.
\smallskip

\noindent\textbf{Case U1:} $\infty \in \mathscr{K}(\sigma)$.
The critical points of $f$ obtained in Proposition~\ref{individual_contribution_prop} lie in the set $\mathrm{crit}_T(f)\sqcup\mathrm{crit}_C(f)$. Recall from Section~\ref{pole_crit_points_subsec} that $\#_m \mathrm{crit}_P(f)=d_f-n_\Omega-1$. Therefore, by Proposition~\ref{individual_contribution_prop} and Lemma~\ref{crit_relation_lem},
\begin{align*}
& \#_m\left(\mathrm{crit}_{T}(f)\cup\mathrm{crit}_C(f)\cup\mathrm{crit}_P(f)\cup\mathrm{crit}_{\cS}(f)\right) = \\
& \#_m\mathrm{crit}_{T}(f)+\#_m\mathrm{crit}_C(f)+\#_m\left(\mathrm{crit}_P(f)\cup\mathrm{crit}_{\cS}(f)\right)\\
& \geq \sum\limits_{\substack{\mathrm{components\ K}\\ \mathrm{of\ \Omega^\complement}}} (\# D_K + 3) + (d_f-n_\Omega-2) + \Delta.
\end{align*}

Thanks to Proposition~\ref{schwarz_crit_pnt_prop} and the above inequality, we conclude that
\begin{align*}
\sum\limits_{\substack{\mathrm{components\ K}\\ \mathrm{of\ \Omega^\complement}}} (\# D_K + 3) + (d_f-n_\Omega-2) + \Delta & \leq 2d_f+2k-4\\
\implies (\# D + 3k) + (d_f-n_\Omega-2) + \Delta & \leq 2d_f+2k-4\\
\implies \# D + k + \Delta & \leq d_f+n_\Omega-2,
\end{align*}
where $k=\mathrm{conn}(\Omega)$.

Note that if $\Omega$ has a node at $\infty$, then Lemma~\ref{crit_relation_lem} implies that 
\begin{align*}
& \#_m\left(\mathrm{crit}_{T}(f)\cup\mathrm{crit}_C(f)\cup\mathrm{crit}_P(f)\cup\mathrm{crit}_{\cS}(f)\right) = \\
& \#_m\mathrm{crit}_{T}(f)+\#_m\mathrm{crit}_C(f)+\#_m\left(\mathrm{crit}_P(f)\cup\mathrm{crit}_{\cS}(f)\right)\\
& \geq \sum\limits_{\substack{\mathrm{components\ K}\\ \mathrm{of\ \Omega^\complement}}} (\# D_K + 3) + (d_f-n_\Omega-1) + \Delta.\\
&\mathrm{So}, \sum\limits_{\substack{\mathrm{components\ K}\\ \mathrm{of\ \Omega^\complement}}} (\# D_K + 3) + (d_f-n_\Omega-1) + \Delta  \leq 2d_f+2k-4\\
&\implies \# D + k + \Delta  \leq d_f+n_\Omega-3,
\end{align*}
as required.
\smallskip

\noindent\textbf{Case U2: $\infty \in T^\infty(\sigma)$.}
In this case, there exist $n\geq 1$ and a component $K_0$ of $T^0(\sigma)$ such that $\sigma^{\circ n}(\infty)\in K_0$. Note that $\overline{K_0}$ is contained in some component $K$ of~$\Omega^\complement$.
\smallskip

\noindent\textbf{Case U2a: $\partial K$ has no double points.} Under this assumption, we have that $K=\overline{K_0}$. We refer to case $(1)$ of Corollaries~\ref{non_sing_pole_cor}, \ref{two_cusp_pole_cor} and \ref{one_cusp_pole_cor} for this subcase. The critical points of $f$ obtained in these corollaries lie in the set $\mathrm{crit}_T(f)\cup\mathrm{crit}_P(f)\cup\mathrm{crit}_C(f)$, and we have
\begin{align*}
&\#_m \mathrm{crit}_T(f)\cup\mathrm{crit}_P(f)\cup\mathrm{crit}_C(f)\\
\geq & \max\{d_f-n_\Omega+1,3\} +  \sum\limits_{\substack{\mathrm{components\ K' \mathrm{of\ \Omega^\complement}}\\ K'\neq K}}  (\# D_{K'} + 3).
\end{align*} 
Note that by Lemma~\ref{crit_relation_lem}, the set $\mathrm{crit}_{\cS}(f)$ is disjoint from $\mathrm{crit}_T(f)\cup\mathrm{crit}_P(f)\cup\mathrm{crit}_C(f)$. Combining the above facts with Lemmas~\ref{higher_order_cusp_lem} and~\ref{higher_order_dp_lem}, we see that
\begin{align*}
& \#_m\left(\mathrm{crit}_{T}(f)\cup\mathrm{crit}_C(f)\cup\mathrm{crit}_P(f)\cup\mathrm{crit}_{\cS}(f)\right)=\\
& \#_m \left(\mathrm{crit}_{T}(f)\cup\mathrm{crit}_C(f)\cup\mathrm{crit}_P(f)\right) + \#_m\mathrm{crit}_{\cS}(f) \\
\geq & \max\{d_f-n_\Omega+1,3\} +  \sum\limits_{\substack{\mathrm{components\ K' \mathrm{of\ \Omega^\complement}}\\ K'\neq K}}  (\# D_{K'} + 3) + \Delta.\\
\mathrm{So},\ & (d_f-n_\Omega+1) + (\# D + 3(k-1)) + \Delta \leq 2d_f+2k-4,\\
\implies & \# D + k + \Delta \leq d_f+n_\Omega-2.
\end{align*}
(In the above computation, we used the hypothesis that $n_\Omega\leq d_f-2$.)
\smallskip

\noindent\textbf{Case U2b: $\partial K$ has $\alpha>0$ many double points.} In this case, we have that $\overline{K_0}\subsetneq K$. We refer to case $(1)$ of Corollary~\ref{dp_pole_cor} for this subcase. The critical points of $f$ obtained in this corollary lie in the set $\mathrm{crit}_T(f)\cup\mathrm{crit}_P(f)\cup\mathrm{crit}_C(f)$. As in the previous case, this observation, in combination with Lemmas~\ref{higher_order_cusp_lem},~\ref{higher_order_dp_lem}, and~\ref{crit_relation_lem}, yield the following inequality:
\begin{align*}
& \#_m\left(\mathrm{crit}_{T}(f)\cup\mathrm{crit}_C(f)\cup\mathrm{crit}_P(f)\cup\mathrm{crit}_{\cS}(f)\right)\\
\geq & \max\{\alpha+d_f-n_\Omega+1,\alpha+3\} +  \sum\limits_{\substack{\mathrm{components\ K' \mathrm{of\ \Omega^\complement}}\\ K'\neq K}}  (\# D_{K'} + 3) + \Delta,\\
\implies & (\alpha+d_f-n_\Omega+1) + ((\# D - \alpha) + 3(k-1)) + \Delta \leq 2d_f+2k-4,\\
\implies & \# D + k  + \Delta \leq d_f+n_\Omega-2.
\end{align*}
(Once again, we used in the above computation the hypothesis that $n_\Omega\leq d_f-2$.)
\smallskip

\noindent\textbf{Bounded case.} Let $K$ be a component of $\Omega^\complement$ such that $\infty\in\Int{K}$.
The proof in this case is similar to Case U2 when $\Omega$ is unbounded, combining case (2) of Corollaries~\ref{non_sing_pole_cor}, \ref{two_cusp_pole_cor}, \ref{one_cusp_pole_cor} and \ref{dp_pole_cor} with Lemmas~\ref{higher_order_cusp_lem},~\ref{higher_order_dp_lem}, and~\ref{crit_relation_lem}.
\end{proof}

\subsection{Bounding the number of weighted singularities}\label{sing_point_linear_bound_subsec}

We will now establish the upper bound on the number of singular points stated in Theorem~\ref{main_thm_2}.

\begin{proof}[Proof of Theorem~\ref{main_thm_2}]
By the discussion in Section~\ref{singularity_subsec}, we know that $\# C=\# \mathrm{crit}_C(f)$. Also recall that $\#_m\mathrm{crit}_P(f)\geq d_f-n_\Omega-1$ (see Section~\ref{pole_crit_points_subsec}), and $\#\mathrm{crit}_{\cS}(f)\geq \Delta$ by Lemmas~\ref{higher_order_cusp_lem} and~\ref{higher_order_dp_lem}, where $\Delta=\left(\sum_{p\in D} \delta_p + \sum_{p\in C} \delta_p\right)$.

By Lemmas~\ref{higher_order_cusp_lem},~\ref{higher_order_dp_lem}, and~\ref{crit_relation_lem}, 
\begin{align*}
&\#_m\left(\mathrm{crit}_C(f)\cup\mathrm{crit}_P(f)\cup\mathrm{crit}_{\cS}(f)\right)=\# \mathrm{crit}_C(f) + \#_m\left(\mathrm{crit}_P(f)\cup\mathrm{crit}_{\cS}(f)\right)\\
\geq &\
\begin{cases}
\# C + (d_f-n_\Omega-2) + \Delta & \mathrm{if}\quad n_\Omega \leq d_f-2,\\
\# C + \Delta & \mathrm{if}\quad n_\Omega>d_f-2.
\end{cases}
\end{align*}
Finally, by Proposition~\ref{schwarz_crit_pnt_prop},
\begin{align}
\notag &\# C + M + \Delta \leq 2d_f+2k-4,\\
\notag &\textrm{where }
M\coloneq
\begin{cases}
d_f-n_\Omega-2 & \mathrm{if}\quad n_\Omega \leq d_f-2,\\
0 & \mathrm{if}\quad n_\Omega>d_f-2,
\end{cases}\\
\label{thm_2_step_1_eqn}
\implies &\ \# C +\Delta  \leq m_1,\\
\notag &\textrm{where }
m_1\coloneq
\begin{cases}
d_f+n_\Omega+2k-2 & \mathrm{if}\quad n_\Omega \leq d_f-2,\\
2d_f+2k-4 & \mathrm{if}\quad n_\Omega>d_f-2,
\end{cases}
\end{align}
and $k=\mathrm{conn}(\Omega)$.

By Theorem~\ref{main_thm_1},
\begin{align}\label{thm_2_step_2_eqn}
\notag & k + \# D + \Delta \leq 
\min \{d_f+n_\Omega-2, 2d_f-4\}\\
\implies & k \leq m_2 - \# D - \Delta,\\
\notag &\textrm{where }
m_2\coloneq
\begin{cases}
d_f+n_\Omega-2 & \mathrm{if}\quad n_\Omega \leq d_f-2,\\
2d_f-4 & \mathrm{if}\quad n_\Omega>d_f-2.
\end{cases}
\end{align}
Putting Inequalities~\eqref{thm_2_step_1_eqn} and~\eqref{thm_2_step_2_eqn} together, we get
\begin{align*}
&\# C + \Delta \leq m,\\
\notag \textrm{where }
m &\coloneq
\begin{cases}
d_f+n_\Omega+2\cdot\left(m_2-\# D-\Delta\right)-2 & \mathrm{if}\quad n_\Omega \leq d_f-2,\\
2d_f+2\cdot\left(m_2-\# D-\Delta\right)-4 & \mathrm{if}\quad n_\Omega>d_f-2,
\end{cases}\\
&=
\begin{cases}
3d_f+3 n_\Omega-2\# D-2\Delta-6 & \mathrm{if}\quad n_\Omega \leq d_f-2,\\
6d_f-2\# D-2\Delta-12 & \mathrm{if}\quad n_\Omega>d_f-2,
\end{cases}\\
\implies &\# C + 2\cdot \# D+ 3 \Delta \leq  \min \{ 3d_f+3 n_\Omega-6, 6d_f-12\}.
\end{align*}
When $\infty$ is a node, Inequality~\ref{thm_2_step_2_eqn} can be promoted to
$$
k  \leq d_f+n_\Omega-\# D-\Delta- 3
$$
(see Section~\ref{pf_main_thm_1_subsec}).
Plugging this into Inequality~\eqref{thm_2_step_1_eqn}, we get
$$
\# C + 2\cdot \# D+ 3 \Delta \leq \min \{ 3d_f+3n_\Omega-8, 4d_f+2n_\Omega-10\}.
$$
This completes the proof of the theorem.
\end{proof}

\section{Constructing quadrature domains with prescribed connectivity\\ and double points: Special cases}\label{sharpness_sec}

In this final section, we will establish sharpness of the upper bound given by Theorem~\ref{special_case_thm} in low degrees; i.e., for $d_f\in\{3, 4\}$, by constructing special examples of quadrature domains with prescribed connectivity and number of double points. Specifically, we will produce Schwarz reflection dynamical systems using surgery and uniformization techniques from conformal dynamics, and this will give rise to the desired quadrature domains. 

We will also show how the above dynamical ideas can be used to give alternative constructions of non-singular quadrature domains $\Omega$ of maximal connectivity; i.e., $\mathrm{conn}(\Omega)=2d_f-4$, for arbitrary $d_f\geq 3$ (cf. \cite[Theorem~A, Theorem~B]{LM14}). 

We make the following remarks.
\begin{itemize}[leftmargin=4mm]
\item The methods for constructing multiply connected quadrature domains available in the literature either use Riemann surface theory  \cite{Bel04,CM04} or Hele-Shaw flow tools \cite{LM14}, which are substantially different from the dynamical techniques introduced in this section.

\item As the boundary of a quadrature domain is real-algebraic, the constructions illustrated below give a recipe to manufacture real-algebraic curves with controlled topology; we refer the reader to \cite{Vir84,Vir90} for constructions of such curves from an algebro-geometric point of view.
\end{itemize}

\subsection{The ingredients}\label{ingredients_subsec}

We start with the description of some piecewise analytic circle coverings that will be used in the construction of multiply connected quadrature domains.

\subsubsection{An expanding/hyperbolic circle map}\label{pre_deltoid_external_map_subsec}
Consider the maps
$$
\widecheck{\pmb{f}}(z)=z+\frac{1}{2z^2},\quad \mathrm{and}\quad \pmb{f}(z)=z+\frac{1}{4z^2}.
$$
Note that $\pmb{f}(z)=\frac{1}{\sqrt[3]{2}}\widecheck{\pmb{f}}(\sqrt[3]{2}z)$.
By \cite[Proposition~B.1]{LLMM3}, the map $\widecheck{\pmb{f}}$ is injective on $\overline{\D^*}$, where $\D^*:=\widehat{\C}\setminus\overline{\D}$, and hence $\pmb{f}$ also enjoys the same property.
Let $\pmb{\Omega}:=\pmb{f}(\D^*)$. By construction, $\pmb{\Omega}$ is an unbounded Jordan quadrature domain with the associated Schwarz reflection map $\pmb{\sigma}\vert_{\overline{\pmb{\Omega}}} \equiv \pmb{f}\circ \kappa\circ \left(\pmb{f}\vert_{\overline{\D^*}}\right)^{-1}$, where $\kappa(z)= 1/\overline{z}$.

Note that $\mathrm{crit}(\pmb{f})=\{0, 1/\sqrt[3]{2}, \omega/\sqrt[3]{2}, \omega^2/\sqrt[3]{2}\}$, where $\omega$ is a primitive third root of unity. A direct computation shows that for each non-zero critical point $c$ of $\pmb{f}$, we have that $\pmb{f}^{-1}(\pmb{f}(c))\subset\D$. Hence, for $c\in\mathrm{crit}(\pmb{f})\setminus\{0\}$, the corresponding critical value $\pmb{f}(c)$ lies in $\Int{\pmb{\Omega}^\complement}$; while $f(0)=\infty\in\pmb{\Omega}$.

We record the following properties of the map $\pmb{\sigma}$ (cf. \cite[\S 4.1]{LM23}).
\begin{itemize}[leftmargin=4mm]
\item Since no critical point of $\pmb{f}$ lies in $\mathbb{S}^1$, it follows that $\partial\pmb{\Omega}$ is a non-singular real-analytic Jordan curve, and hence $\pmb{\sigma}^{-1}(\pmb{\Omega})$ is compactly contained in $\pmb{\Omega}$. This is depicted in Figure~\ref{pre_deltoid_external_map_fig} (right); where the central light green Jordan domain is $\Int{\pmb{\Omega}^\complement}$, and the adjacent blue annular region is $\pmb{\sigma}^{-1}(\Int{\pmb{\Omega}^\complement})$.

\item The map $\pmb{\sigma}$ has a super-attracting fixed point at $\infty$. The other three critical points of $\pmb{\sigma}$ are mapped under $\pmb{\sigma}$ to $\Int{\pmb{\Omega}^\complement}$ (cf. Section~\ref{critical_subsec}).

\item The branched covering $\pmb{\sigma}:\pmb{\sigma}^{-1}(\pmb{\Omega})\to\pmb{\Omega}$ has degree two and a unique, simple critical point. Thus, by the Riemann-Hurwitz formula, the domain $\pmb{\sigma}^{-1}(\pmb{\Omega})$ is simply connected. In fact, the Jordan curve $\partial\pmb{\Omega}$ does not contain any critical value of $\pmb{\sigma}$, and hence $\pmb{\sigma}$ is locally injective on $\partial\pmb{\sigma}^{-1}(\pmb{\Omega})$. This implies that $\pmb{\sigma}^{-1}(\pmb{\Omega})$ is also a Jordan domain. 

We conclude that $\pmb{\sigma}:\pmb{\sigma}^{-1}(\pmb{\Omega})\to\pmb{\Omega}$ is an anti-polynomial-like map of degree two (cf. \cite[Chapter~I]{DH85}, \cite[\S 5]{IM21}). The filled Julia set of this anti-polynomial-like map is precisely the non-escaping set $\mathscr{K}(\pmb{\sigma})$ of the Schwarz reflection map (shown in gray in Figure~\ref{pre_deltoid_external_map_fig} (right)).
Since $\pmb{\sigma}$ has a fixed critical point, it follows by the straightening theorem for anti-polynomial-like maps that $\pmb{\sigma}\vert_{\mathscr{K}(\pmb{\sigma})}$ is topologically conjugate (conformally on the interior) to~$\overline{z}^2\vert_{\overline{\D}}$.

\item As $\pmb{\Omega}, \pmb{\sigma}^{-1}(\pmb{\Omega})$ are Jordan domains with $\pmb{\sigma}^{-1}(\pmb{\Omega})$ compactly contained in $\pmb{\Omega}$, we have that $\pmb{\Omega}\setminus\overline{\pmb{\sigma}^{-1}(\pmb{\Omega})}$ is an annulus (the blue annulus adjacent to $\pmb{\Omega}^\complement$ in Figure~\ref{pre_deltoid_external_map_fig} (right)) that maps as a degree three branched cover onto $\Int{\pmb{\Omega}^\complement}$ under $\pmb{\sigma}$. Alternatively, one can use the facts that $\Int{\pmb{\Omega}^\complement}$ is simply connected and that the degree three branched covering $\pmb{\sigma}:\pmb{\sigma}^{-1}(\Int{\pmb{\Omega}^\complement})\to\Int{\pmb{\Omega}^\complement}$ has precisely three critical points to deduce that $\pmb{\sigma}^{-1}(\Int{\pmb{\Omega}^\complement})$ is an annulus.
\end{itemize}
\begin{figure}[ht!]
\captionsetup{width=0.98\linewidth}
\begin{tikzpicture}
\node[anchor=south west,inner sep=0] at (0,0) {\includegraphics[width=0.96\textwidth]{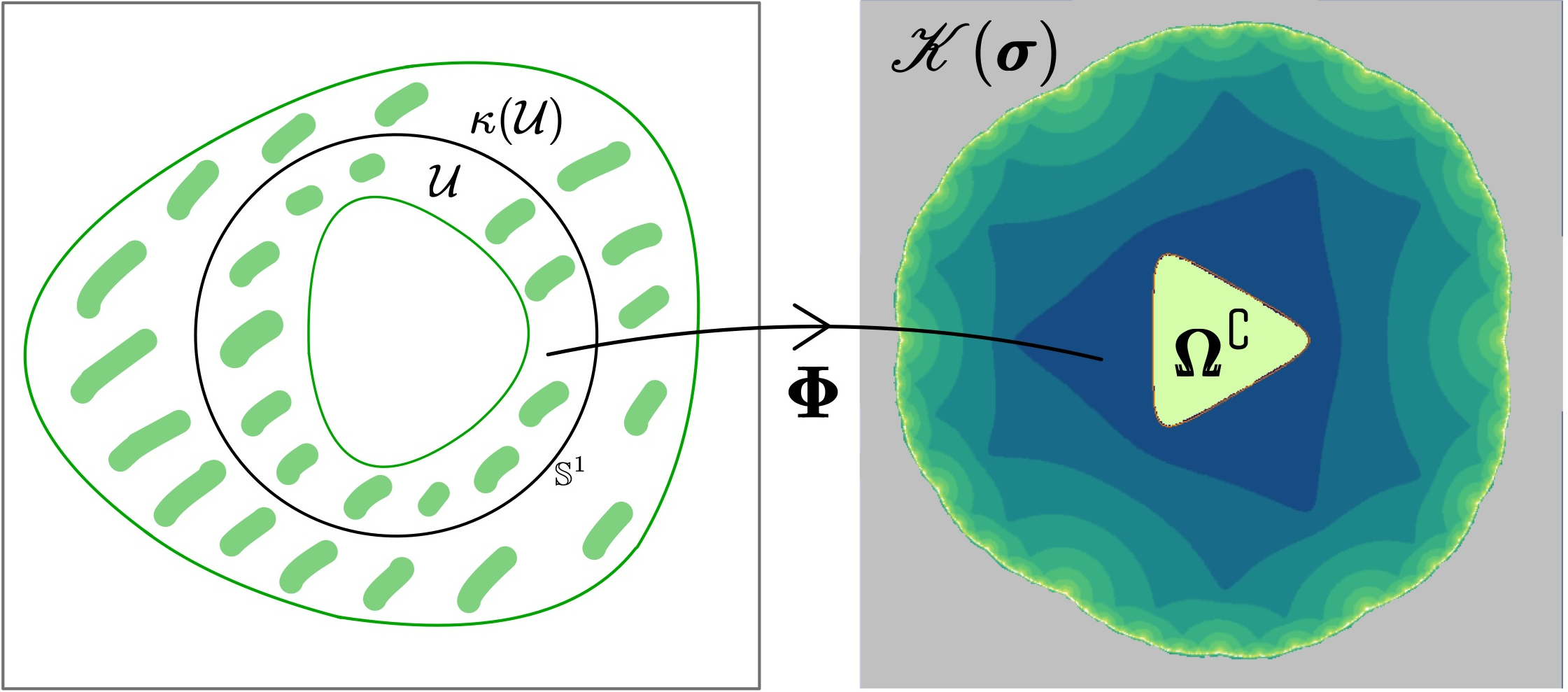}};
\end{tikzpicture}
\caption{Right: Illustrated is the dynamical plane of $\pmb{\sigma}$ with the non-escaping set $\mathscr{K}(\pmb{\sigma})$ in gray and the escaping set $T^\infty(\pmb{\sigma})$ in green/blue. Left: The domain of definition $\cV=\cU\cup\kappa(\cU)\cup\mathbb{S}^1$ of the map $\mathcal{E}$ is shaded in green. The restriction of $\mathcal{E}$ on $\mathbb{S}^1$ is an orientation-reversing, expanding, real-analytic double covering of itself, while $\mathcal{E}$ is the identity map on $\partial\cV$.}
\label{pre_deltoid_external_map_fig}
\end{figure}

We now look at the \emph{external map} of the above anti-polynomial-like map (cf. \cite[Chapter~I, \S 2]{DH85}). Specifically, let $\pmb{\Phi}:\D\to T^\infty(\pmb{\sigma})$ be a conformal isomorphism. The topological conjugacy between $\pmb{\sigma}\vert_{\mathscr{K}(\pmb{\sigma})}$ and $\overline{z}^2\vert_{\overline{\D}}$ implies that $\partial \mathscr{K}(\pmb{\sigma})=\partial T^\infty(\pmb{\sigma})$ is a Jordan curve, and hence $\pmb{\Phi}$ extends continuously to a homeomorphism between $\overline{\D}$ and $\overline{T^\infty(\pmb{\sigma})}$. Set $\cU:=\pmb{\Phi}^{-1}(\pmb{\Omega}\cap T^\infty(\pmb{\sigma}))$, and
$$
\mathcal{E}:\overline{\cU}\to\overline{\D},\ \mathcal{E}:= \pmb{\Phi}^{-1}\circ\pmb{\sigma}\circ\pmb{\Phi}.
$$
Note that $\partial\cU=\mathbb{S}^1\sqcup \pmb{\Phi}^{-1}(\partial\pmb{\Omega})$. The restriction of the map $\mathcal{E}$ to $\mathbb{S}^1$ is an orientation-reversing double covering, while $\mathcal{E}$ acts as the identity map on $\pmb{\Phi}^{-1}(\partial\pmb{\Omega})$. Using the Schwarz reflection principle, we extend $\mathcal{E}$ to a continuous map  $\mathcal{E}:\overline{\cV}\to\widehat{\C}$ such that $\mathcal{E}$ is antiholomorphic on $\cV$, where $\cV=\cU\cup\kappa(\cU)\cup\mathbb{S}^1$ is an annulus. It follows that $\mathcal{E}:\mathbb{S}^1\to\mathbb{S}^1$ is a real-analytic, expanding, orientation-reversing double covering. Further, 
\begin{enumerate}[leftmargin=6mm]
\item\label{identity_prop} $\mathcal{E}$ acts as the identity map on $\partial\cV=\pmb{\Phi}^{-1}(\partial\pmb{\Omega})\sqcup\kappa(\pmb{\Phi}^{-1}(\partial\pmb{\Omega}))$, 

\item\label{degree_prop} $\mathcal{E}:\pmb{\Phi}^{-1}(\pmb{\sigma}^{-1}(\Int{\pmb{\Omega}^\complement}))\to\pmb{\Phi}^{-1}(\Int{\pmb{\Omega}^\complement})$ and $\mathcal{E}:\kappa(\pmb{\Phi}^{-1}(\pmb{\sigma}^{-1}(\Int{\pmb{\Omega}^\complement})))\to\kappa(\pmb{\Phi}^{-1}(\Int{\pmb{\Omega}^\complement}))$ are degree $3$ branched coverings, and

\item\label{non_sing_prop} $\partial\cV$ consists of a pair of non-singular, real-analytic Jordan curves (this follows from the construction of $\cV$ and the fact that $\partial\pmb{\Omega}$ is non-singular)
 \end{enumerate}

\subsubsection{Two expansive/parabolic circle maps}\label{nielsen_anti_farey_subsec}
The dynamics of Schwarz reflections in quadrature domains is closely related to certain piecewise analytic maps associated with reflection groups acting on the hyperbolic plane (cf. \cite{LM23}). We will now recall two such maps that will facilitate the construction of quadrature domains with controlled topology.
\smallskip

\noindent\textbf{The Nielsen map.}
Let $C_j$ be the hyperbolic geodesic in $\D$ connecting $\exp{(\frac{2i\pi (j-1)}{3})}$ and $\exp{(\frac{2i\pi j}{3})}$, for $j\in\{1,2,3\}$. The geodesics $C_1, C_2, C_3$ bound an ideal triangle $\Pi\subset\D$, where $\Pi$ is closed in the topology of $\D$. We denote the components of $\D\setminus\Pi$ by $D_1, D_2, D_3$, where $C_j\subset \partial D_j$, for $j\in\{1,2,3\}$. We denote the anti-M{\"o}bius reflection in $C_j$ by $\rho_j$, and note that $\rho_j$ is an anti-conformal involution of $\D$. The maps $\rho_1,\rho_2$, and $\rho_3$ generate a discrete subgroup of the group of conformal and anti-conformal automorphisms of $\D$. This group is called the \emph{ideal triangle reflection group}.
The \emph{Nielsen map} of the ideal triangle reflection group is defined as
$$
\mathcal{N}:\overline{\D}\setminus\Int{\Pi}\to \overline{\D},\quad z\mapsto \rho_j(z),\ \mathrm{if}\ z\in\overline{D_j}.
$$
(See Figure~\ref{nielsen_farey_fig} (left).) The following properties of $\cN$ will be useful for us.
\begin{itemize}[leftmargin=4mm]
\item The map $\cN$ fixes $\partial\Pi$ pointwise.
\item The restriction $\cN\vert_{\mathbb{S}^1}$ is an orientation-reversing, piecewise analytic, $C^1$, degree two expansive covering map. Hence, $\cN\vert_{\mathbb{S}^1}$ is topologically conjugate to~$\overline{z}^2\vert_{\mathbb{S}^1}$.
\item $\cN:\cN^{-1}(\Int{\Pi})\to\Int{\Pi}$ is a $3:1$ covering map.
\end{itemize}
(We refer the reader to \cite[\S 2]{LLMM1}, \cite[\S 3.1]{LM23} for details on the Nielsen map.)
\begin{figure}[ht!]
\captionsetup{width=0.98\linewidth}
\begin{tikzpicture}
\node[anchor=south west,inner sep=0] at (0,0) {\includegraphics[width=0.4\textwidth]{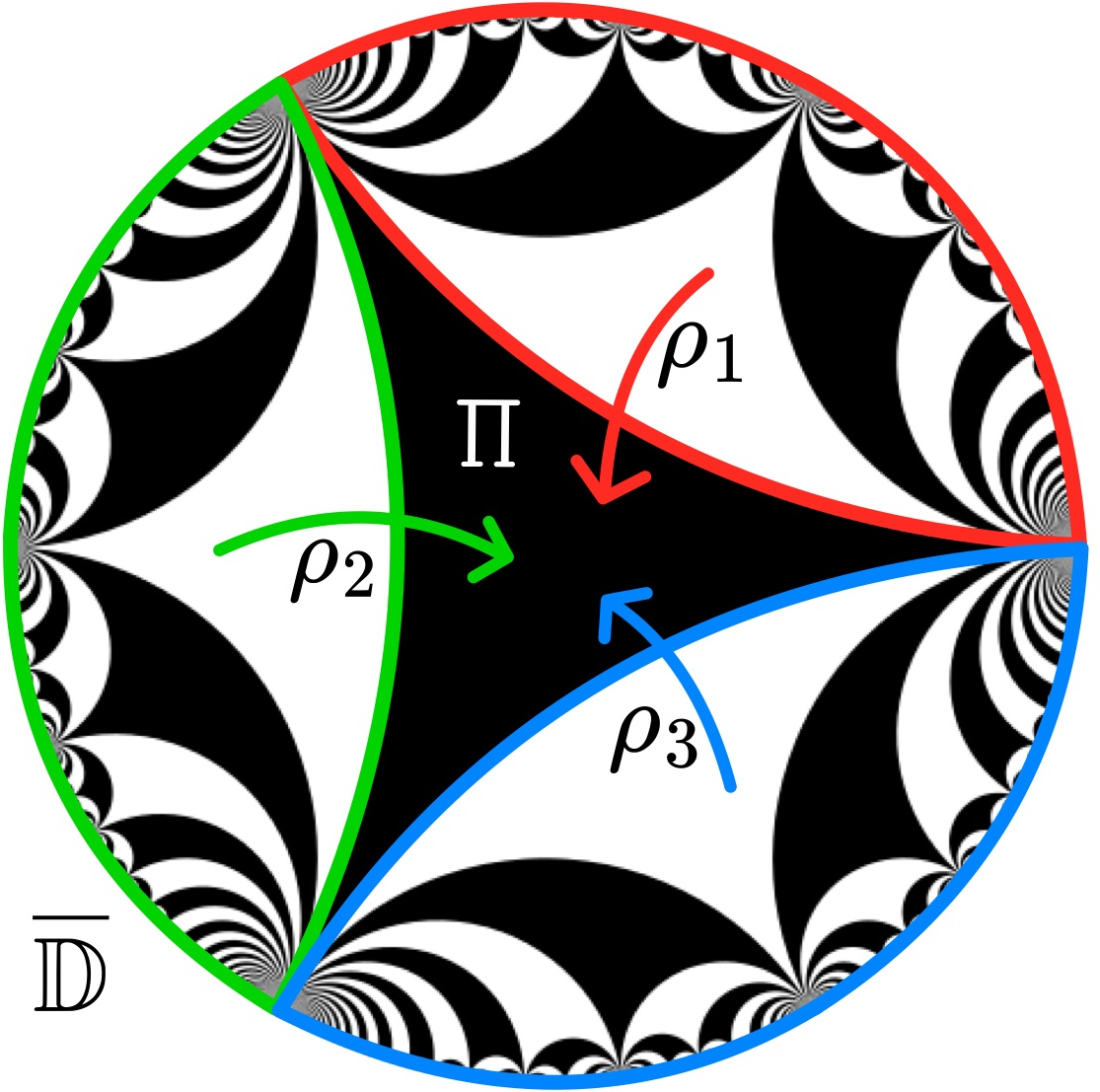}};
\node[anchor=south west,inner sep=0] at (6.4,0) {\includegraphics[width=0.4\textwidth]{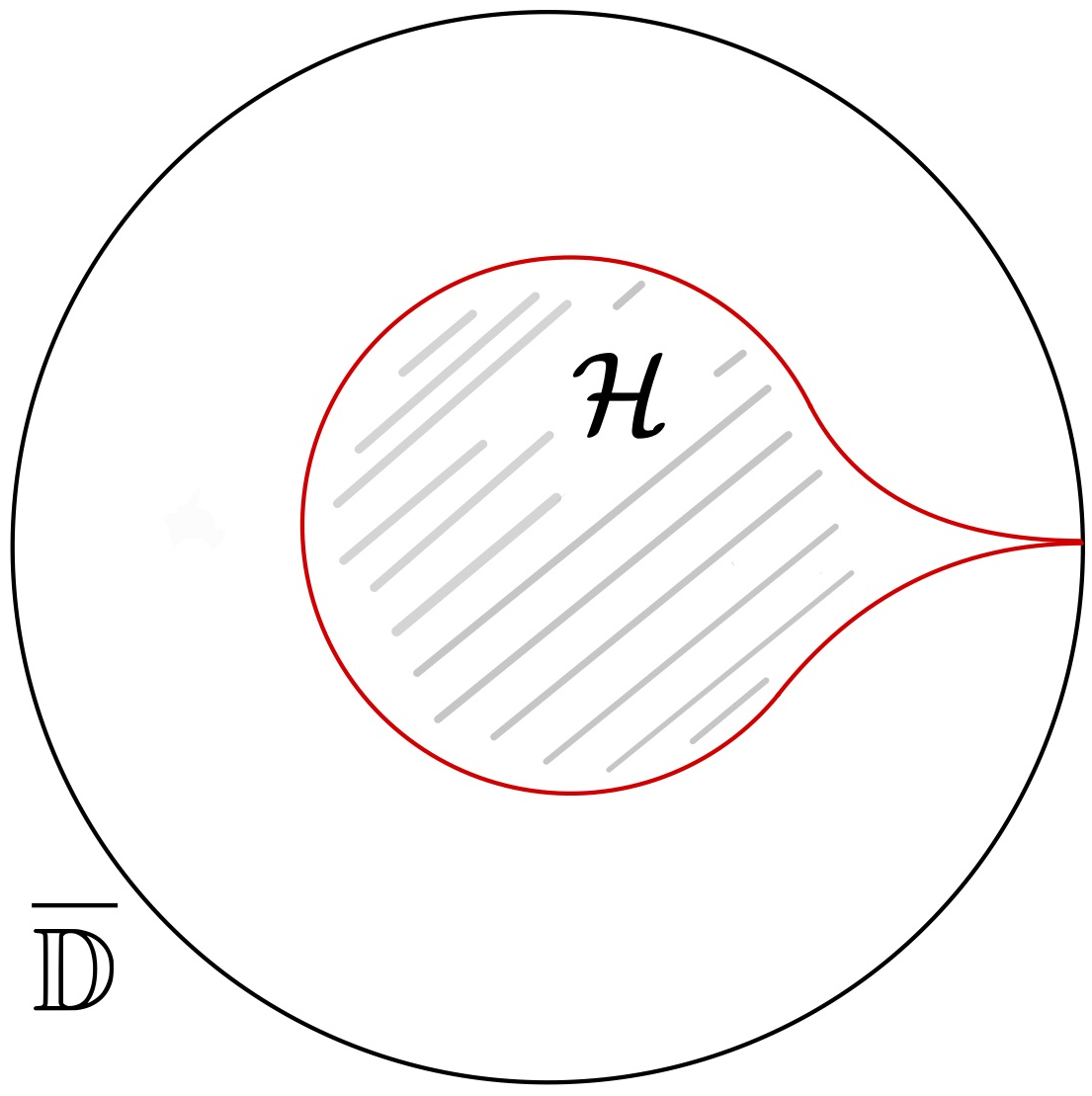}};
\end{tikzpicture}
\caption{Left: The Nielsen map $\cN$ is defined outside the interior of the central black triangle $\Pi$ as reflections in the sides of $\Pi$. This map commutes with rotation by $2\pi/3$. Right: The anti-Farey map, which is a factor of $\cN$ under the $2\pi/3-$rotational symmetry, is defined outside the interior of the `monogon' $\cH$.}
\label{nielsen_farey_fig}
\end{figure}

\noindent\textbf{The anti-Farey map.}
The map $\cN$ commutes with the rotation $M_\omega(z)=\omega z$, where $\omega=\exp{(2\pi i/3)}$. Hence, $\cN:\overline{\D}\setminus\Int{\Pi}\to \overline{\D}$ descends to a factor map
$$
\widehat{\cN}:\left(\overline{\D}\setminus\Int{\Pi}\right)/\langle M_\omega\rangle \longrightarrow \overline{\D}/\langle M_\omega\rangle
$$
via the natural projection from $\overline{\D}$ to the bordered quotient orbifold $\overline{\D}/\langle M_\omega\rangle$. As $z\mapsto z^3$ induces a conformal isomorphism $\zeta: \overline{\D}/\langle M_\omega\rangle\to \overline{\D}$, one obtains a map
$$
\cF:=\zeta\circ\widehat{\cN}\circ\zeta^{-1}:\overline{\D}\setminus\Int{\cH} \to \overline{\D},
$$
where $\cH:=\zeta\left(\Pi/\langle M_\omega\rangle\right)$. The map $\cF$ is called the \emph{anti-Farey} map. We record some important properties of $\cF$ (See Figure~\ref{nielsen_farey_fig}~(right)).
\begin{itemize}[leftmargin=4mm]
\item The map $\cF$ fixes $\partial\cH$ pointwise.
\item The restriction $\cF\vert_{\mathbb{S}^1}$ is an orientation-reversing, piecewise analytic, $C^1$, degree two expansive covering map, which is topologically conjugate to~$\overline{z}^2\vert_{\mathbb{S}^1}$.
\item $\cF:\cF^{-1}(\Int{\cH})\to\Int{\cH}$ is a $3:1$ branched covering map, fully branched over the origin.
\end{itemize}
(See \cite[\S 3.1]{LMM24}, \cite[\S 4.3.1]{LM23} for a detailed account of the anti-Farey map.)

\subsection{Effectiveness of Theorem~\ref{special_case_thm}: the $d_f=3$ case}\label{sharpness_1_subsec}

For $d_f=3$, Theorem~\ref{special_case_thm} reduces to $\mathrm{conn}(\Omega)+\# D\leq 2$. We will now illustrate how to construct quadrature domains with $d_f=3$ having
\begin{itemize}
\item connectivity $2$ and no double point, and
\item connectivity $1$ and one double point.
\end{itemize}

\subsubsection{Connectivity $2$, no double points}

Recall the continuous map $\mathcal{E}:\overline{\cV}\to\widehat{\C}$ (which is antiholomorphic on $\cV$) from Section~\ref{pre_deltoid_external_map_subsec}. By Property~\eqref{identity_prop}, $\cV$ is a quadrature domain of connectivity $2$. By Property~\eqref{degree_prop}, the map $\mathcal{E}:\mathcal{E}^{-1}(\Int{\cV^\complement})\to\Int{\cV^\complement}$ is a branched covering of degree $3$. Hence, the uniformizing meromorphic map $f$ of $\cV$ has degree $3$; i.e., $d_f=3$ (see Proposition~\ref{schwarz_deg_prop}). Finally, Property~\eqref{non_sing_prop} implies that $\partial\cV$ has no double points.

\subsubsection{Connectivity $1$, unique double point}

Let us consider the orientation-reversing, piecewise analytic, $C^1$, degree
two expansive covering maps $\cN$ and $\cF$. According to \cite[Theorem~5.2]{LMMN20}, these two maps can be conformally mated. Indeed, it follows from \cite[Example~4.3]{LMMN20} and \cite[Lemma~3.2]{LMM24} that the maps $\cN$ and $\cF$ satisfy the conditions of conformal mateability (respectively) as required in \cite[Theorem~5.2]{LMMN20}. We can mate these maps using a circle homeomorphism that conjugates $\cN$ to $\cF$ and maps the fixed point $1$ of $\cN$ to the fixed point $1$ of $\cF$.

We denote the conformal mating of $\cN:\overline{\D}\setminus\Int{\Pi}\to \overline{\D}$ and $\cF:\overline{\D}\setminus\Int{\cH}\to \overline{\D}$ by $\sigma:\overline{\Omega}\to\widehat{\C}$. By definition of conformal matings, there exist a $\sigma-$invariant Jordan curve $\mathfrak{J}$, and two conformal maps $\Psi_+:\overline{\D}\to\overline{\mathscr{D}_+}$ and $\Psi_-:\overline{\D^*}\to\overline{\mathscr{D}_-}$ (where $\mathscr{D}_\pm$ are the components of $\widehat{\C}\setminus\mathfrak{J}$) such that $\Psi_+$ conjugates $\cN:\overline{\D}\setminus\Int{\Pi}\to \overline{\D}$ to $\sigma:\overline{\Omega}\cap\overline{\mathscr{D}_+}\to\overline{\mathscr{D}_+}$, and $\Psi_-\circ\kappa$ conjugates $\cF:\overline{\D}\setminus\Int{\cH}\to \overline{\D}$ to $\sigma:\overline{\Omega}\cap\overline{\mathscr{D}_-}\to\overline{\mathscr{D}_-}$ (where $\kappa(z)= 1/\overline{z}$). The domain of definition $\Omega$ of the conformal mating $\sigma$, the Jordan curve $\mathfrak{J}$, and the Jordan domains $\mathscr{D}_\pm$ are depicted in Figure~\ref{suffridge_fig}. 
\begin{figure}[ht!]
\captionsetup{width=0.98\linewidth}
\begin{tikzpicture}
\node[anchor=south west,inner sep=0] at (0,0) {\includegraphics[width=0.45\textwidth]{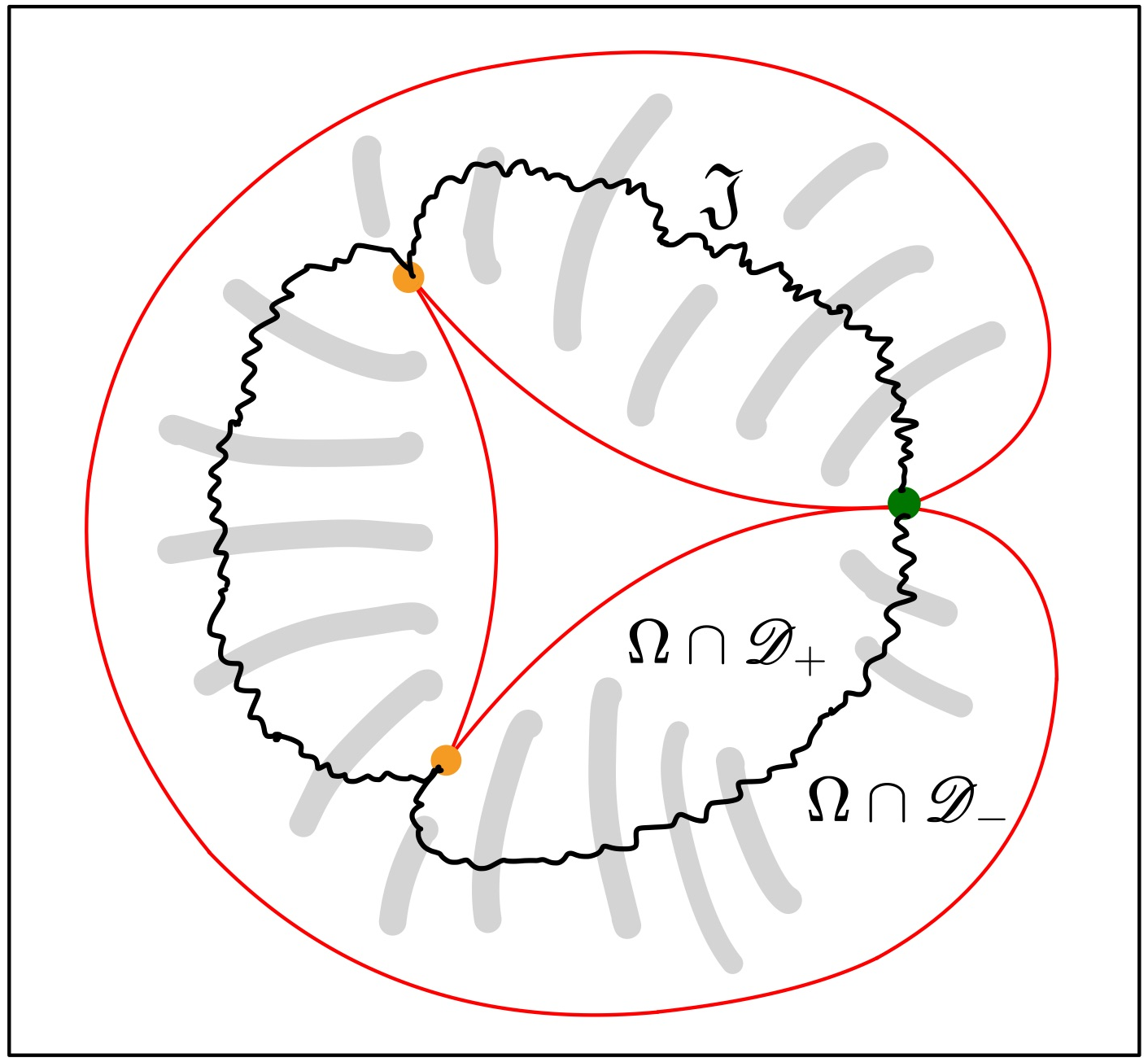}};
\end{tikzpicture}
\caption{The region shaded in gray, enclosed by the red curve, is the domain of definition of the conformal mating $\sigma$ of $\cN$ and $\cF$.}
\label{suffridge_fig}
\end{figure}

We note the following properties of the conformal mating.
\begin{enumerate}[leftmargin=6mm]
    \item[(a)] The set $\Omega$ is a simply connected domain and $\partial\Omega$ has a unique double point. This follows from the facts that $\partial\Pi$ is a triangle in $\overline{\D}$ intersecting $\mathbb{S}^1$ only in its three vertices (the third roots of unity), $\partial\cH$ is a Jordan curve in $\overline{\D}$ intersecting $\mathbb{S}^1$ only at $1$, and the point $1\in\partial\Pi$ is identified with the point $1\in\partial\cH$ in the  mating process (see Figure~\ref{suffridge_fig}).
    \item[(b)] The map $\sigma:\overline{\Omega}\to\widehat{\C}$ is a continuous map that is antiholomorphic on $\Omega$.
    \item[(c)] $\sigma$ acts as the identity map on $\partial\Omega$, and hence $\Omega$ is a quadrature domain of connectivity one with a unique double point 
    on its boundary.
    \item[(d)] As $\cN:\cN^{-1}(\Int{\Pi})\to\Int{\Pi}$ and $\cF:\cF^{-1}(\Int{\cH})\to\Int{\cH}$ are degree $3$
 branched coverings, it follows that $\sigma: \sigma^{-1}(\Int{\Omega^\complement})\to\Int{\Omega^\complement}$ is also a degree $3$
 branched covering. By Proposition~\ref{schwarz_deg_prop}, the uniformizing meromorphic map associated with $\Omega$ has degree $3$.
 \end{enumerate}
 
Thus, $\Omega$ is an example of a quadrature domain with the desired properties.
An explicit formula for the uniformizing meromorphic map of this quadrature domain can be found in \cite[\S 2.1]{LM14}. We also refer the reader to \cite[\S 7]{LMM21} for a different (limiting) construction of this quadrature domain.

\subsection{Effectiveness of Theorem~\ref{special_case_thm}: the $d_f=4$ case}\label{sharpness_2_subsec}
We will now furnish quadrature domains $\Omega$ with $d_f=4$ such that $\Omega^\complement$ has $m$ connected components and $\partial\Omega$ has $n$ double points, where $m+n=4$, $m\geq 1$. The key idea is to start with a specific antiholomorphic rational map, and suitably replace its action on the invariant Fatou components with the actions of the partially defined maps $\mathcal{E}, \cN, \cF$ of the disk.

\subsubsection{The Apollonian anti-rational map}\label{apollo_anti_rat_subsec}
Consider the antiholomorphic rational map (anti-rational map for short)
$$
\pmb{R}(z)=\frac{3\overline{z}^2}{2\overline{z}^3+1}.
$$
The map $\pmb{R}$ is critically fixed; i.e., it fixes all of its four simple critical points. In other words, each critical point of $\pmb{R}$ is a superattracting fixed point, and hence $\pmb{R}$ is a hyperbolic anti-rational map (cf. \cite[\S 19]{Mil06}). The dynamics of the map $\pmb{R}$ was considered in \cite[\S 8]{LLMM4}, where it was shown that the Julia set of $\pmb{R}$ is homeomorphic to the classical Apollonian gasket in a dynamically natural way (cf. \cite{LLM22}, \cite[\S 5.2]{LM23}). 
\begin{figure}[ht!]
\captionsetup{width=0.98\linewidth}
\begin{tikzpicture}
\node[anchor=south west,inner sep=0] at (0,0) {\includegraphics[width=0.45\textwidth]{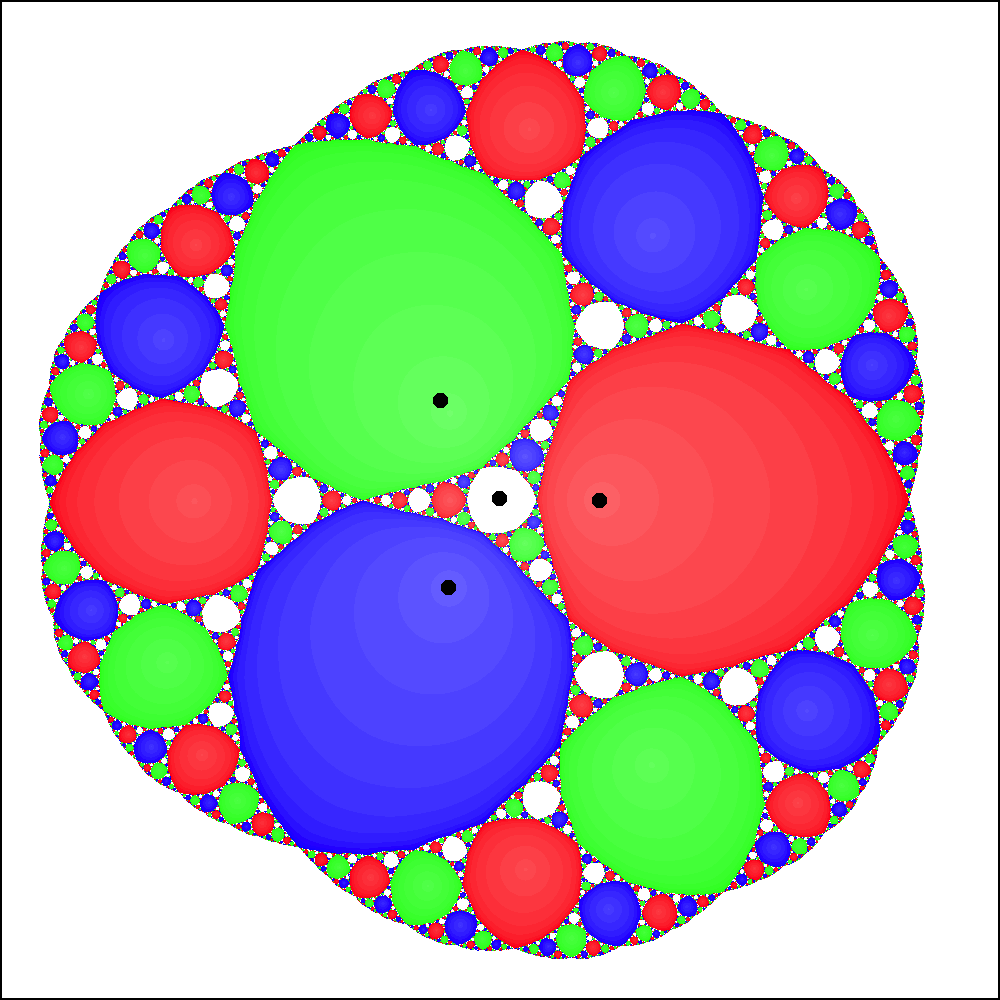}};
\end{tikzpicture}
\caption{The Julia set of the critically fixed cubic anti-rational map $\pmb{R}$ is shown. The black dots represent the critical points. The Fatou components containing the critical points touch each other pairwise.}
\label{apollonian_julia_fig}
\end{figure}

According to \cite[\S 8]{LLMM4}, the four super-attracting immediate basins of $\pmb{R}$ (i.e., the invariant Fatou components of $\pmb{R}$ containing the four critical points) are Jordan domains that touch each other pairwise at the six repelling fixed points of~$\pmb{R}$ (see Figure~\ref{apollonian_julia_fig}). We enumerate these invariant Fatou components as $\cW_1,\cdots,\cW_4$. For definiteness, let us refer to the blue, green, red, and white components as $\cW_1,\cW_2,\cW_3,\cW_4$, respectively. We denote (the homeomorphic extensions of) the Riemann uniformizations of the components $\cW_j$ by $\Psi_j:\overline{\D}\to\overline{\cW_j}$, such that $\Psi_j$ sends the origin to the (fixed) critical point in $\cW_j$, $j\in\{1,\cdots,4\}$. After possibly precomposing $\Psi_j$ with a rotation, we may assume that it conjugates $\overline{z}^2\vert_{\overline{\D}}$ to $\pmb{R}\vert_{\overline{\cW_j}}$, for $j\in\{1,\cdots,4\}$.  

\subsubsection{Constructing topological Schwarz reflection maps via surgery}\label{topo_schwarz_subsec}
We will now manufacture partially defined continuous maps on the sphere by replacing the action of $\pmb{R}$ on $\cup_{j=1}^4\cW_j$ with those of $\mathcal{E}, \cN$, or $\cF$ appropriately. The resulting maps will serve as topological models of the desired Schwarz reflection maps.
\begin{figure}[ht!]
\captionsetup{width=0.98\linewidth}
\begin{tikzpicture}
\node[anchor=south west,inner sep=0] at (0,0) {\includegraphics[width=0.9\textwidth]{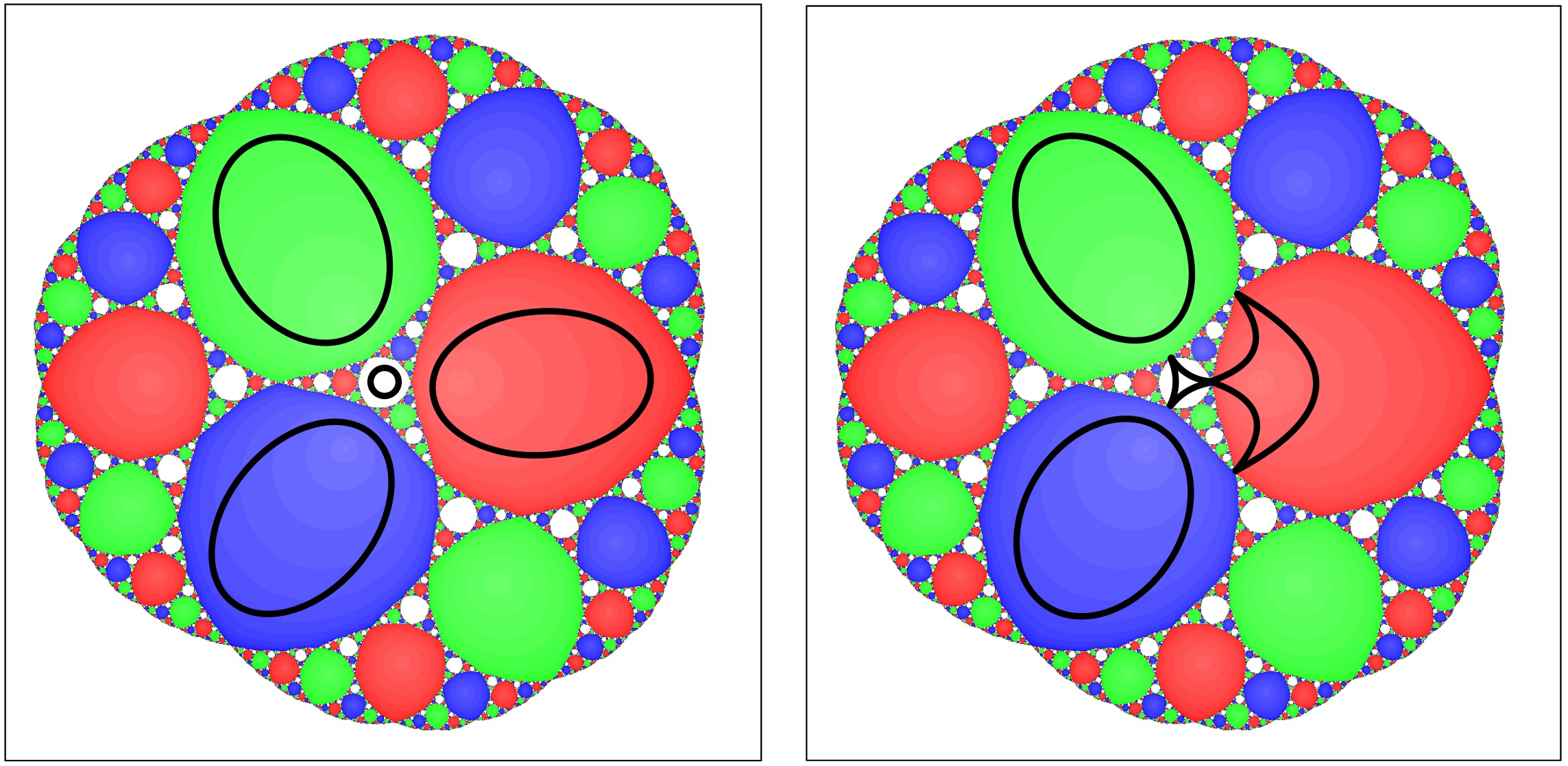}};
\node[anchor=south west,inner sep=0] at (0,-6) {\includegraphics[width=0.9\textwidth]{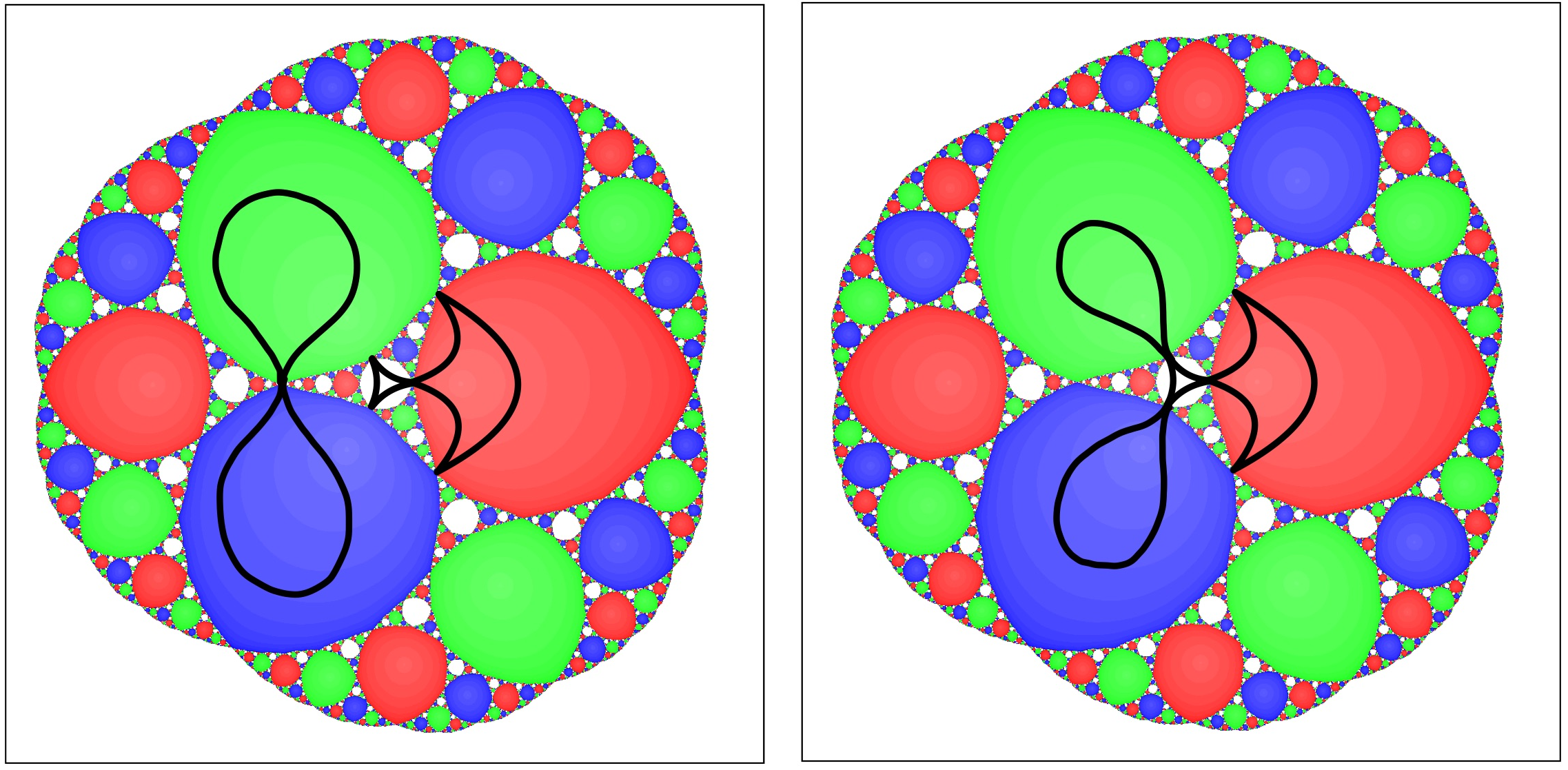}};
\end{tikzpicture}
\caption{The domains outside the black curves are the domains of definition of the `topological Schwarz reflection maps' constructed in Section~\ref{topo_schwarz_subsec}.}
\label{apollo_schwarz_fig}
\end{figure}

To this end, let $\cM_j\in\{\mathcal{E},\cN,\cF\}$, for $j\in\{1,\cdots,4\}$. There exist homeomorphisms $\theta_j:\mathbb{S}^1\to\mathbb{S}^1$ conjugating $\overline{z}^2\vert_{\mathbb{S}^1}$ to $\cM_j\vert_{\mathbb{S}^1}$. 
\begin{itemize}[leftmargin=4mm]
    \item If $\cM_j=\mathcal{E}$, then the conjugacy $\theta_j$ is quasisymmetric; indeed, two real-analytic, expanding endomorphisms of the circle of the same degree and orientation are quasisymmetrically conjugate (cf. \cite[Proposition~6.3]{McM88}, \cite[Exercise~2.3]{MS93}). In this case, we extend $\theta_j$ continuously to a  homeomorphism of $\overline{\D}$ that is quasiconformal on $\D$. 
    \item If $\cM_j=\cN$, then by \cite[Theorem~4.14]{LMMN20} the conjugacy $\theta_j$ between $\overline{z}^2$ and $\cM_j$ can be extended continuously to a homeomorphism of $\overline{\D}$ that is a David homeomorphism on $\D$. 
    \item Finally, if $\cM_j=\cF$, then by \cite[Lemma~3.2]{LMM24} the conjugacy $\theta_j$ between $\overline{z}^2$ and $\cM_j$ can be extended continuously to a homeomorphism of $\overline{\D}$ that is a David homeomorphism on $\D$.
\end{itemize}

We define a partially defined continuous map $\widetilde{\sigma}:\widehat{\C}\dashrightarrow\widehat{\C}$ as
$$
\widetilde{\sigma}\equiv \left\{\begin{array}{ll}
                    \pmb{R} & \mbox{on}\quad \widehat{\C}\setminus \displaystyle\bigcup_{j=1}^4 \mathcal{W}_j, \\
                   \Psi_j\circ\theta_j^{-1}\circ\cM_j\circ\theta_j\circ\Psi_j^{-1} & \mbox{on}\quad \Psi_j\left(\theta_j^{-1}\left( \mathrm{Dom}(\cM_j)\cap\D\right)\right),\ j\in\{1,\cdots,4\}.
                            \end{array}\right.
$$
(Here $\mathrm{Dom}(\cM_j)$ stands for the domain of definition of the map $\cM_j$.)
Displayed in Figure~\ref{apollo_schwarz_fig} are four instances of this construction, which we now explicate. Recall that the blue, green, red, and white immediate basins in Figure~\ref{apollo_schwarz_fig} are enumerated as $\cW_1,\cW_2,\cW_3$, and $\cW_4$, respectively.
\begin{enumerate}[leftmargin=6mm]
    \item (Top left of Figure~\ref{apollo_schwarz_fig}) Here, $\cM_j=\mathcal{E}$, for $j\in\{1,\cdots,4\}$. The domain of definition $\mathrm{Dom}(\widetilde{\sigma})$ of $\widetilde{\sigma}$ is the complement of four Jordan domains with pairwise disjoint closures. In particular, $\widetilde{\Omega}:=\Int{\mathrm{Dom}(\widetilde{\sigma})}$ is connected, it has $4$ complementary components, and its boundary has no cut-point.
    
    \item (Top right of Figure~\ref{apollo_schwarz_fig}) Here, $\cM_j=\mathcal{E}$, for $j\in\{1,2\}$, and $\cM_j=\cN$, for $j\in\{3,4\}$. The domain of definition $\mathrm{Dom}(\widetilde{\sigma})$ is the complement of four disjoint Jordan domains, two of which touch at a unique point, while the other two have disjoint closures. In particular, $\widetilde{\Omega}$ is connected, it has $3$ complementary components, and its boundary has a unique cut-point.
    
    \item (Bottom left of Figure~\ref{apollo_schwarz_fig}) In this case, $\cM_j=\cF$, for $j\in\{1,2\}$, and $\cM_j=\cN$, for $j\in\{3,4\}$. Here, $\mathrm{Dom}(\widetilde{\sigma})$ is the complement of four disjoint Jordan domains, such that the first pair (respectively, the second pair) touches at a unique point. Further. $\widetilde{\Omega}$ is connected, it has $2$ complementary components, and its boundary has exactly $2$ cut-points.
    
    \item (Bottom right of Figure~\ref{apollo_schwarz_fig}) As in the previous case, we have $\cM_j=\cF$, for $j\in\{1,2\}$, and $\cM_j=\cN$, for $j\in\{3,4\}$. However, the conformal maps $\Psi_j$ have been normalized in such a way that the point $1\in\partial\cH$ is welded at the unique point of intersection of $\partial \cW_j\cap\partial\cW_4$, for $j\in\{1,2\}$. Consequently, $\widetilde{\Omega}$ is connected, simply connected (i.e., it has a unique complementary component), and its boundary has $3$ cut-points.
\end{enumerate}

We note that in all the above cases, $\widetilde{\sigma}$ fixes the boundary $\partial\widetilde{\Omega}$ of its domain of definition pointwise.

\subsubsection{Uniformizing topological Schwarz reflections}\label{top_to_anal_subsec}

We pull back the standard complex structure on $\D$ under the map $\theta_j\circ\Psi_j^{-1}$ to define an almost complex structure on $\cW_j$, $j\in\{1,\cdots,4\}$. Since $\cM_j$ preserves the standard complex structure on $\D$ (while reversing the orientation), the map $\widetilde{\sigma}$ preserves the almost complex structure on $\cup_{j=1}^4\cW_j$ (while reversing the orientation). We extend this almost complex structure to the entire Fatou set of $\pmb{R}$ by pulling it back under $\widetilde{\sigma}$ (recall that $\widetilde{\sigma}$ is antiholomorphic outside $\cup_{j=1}^4\cW_j$). Since the Julia set of a hyperbolic rational map has zero area, this defines a $\widetilde{\sigma}-$invariant almost complex structure on the sphere (more precisely, $\widetilde{\sigma}$ preserves the measurable ellipse field, but reverses the orientation).

As $\pmb{R}$ is hyperbolic, the proof of \cite[Lemma~7.1]{LMMN20} applies verbatim to the current situation to show that the Beltrami coefficient on $\widehat{\C}$ defined by the above almost complex structure is a \emph{David coefficient} (see \cite[Chapter~20]{AIM09}, \cite[\S 2]{LMMN20} for background on David Beltrami coefficients and David homeomorphisms). We denote by $\mathfrak{G}$ the David homeomorphism of $\widehat{\C}$ that pulls back the standard complex structure to the above almost complex structure. Once again, the proof of \cite[Lemma~7.1]{LMMN20} implies that 
$$
\sigma:=\mathfrak{G}\circ\widetilde{\sigma}\circ\mathfrak{G}^{-1}:\overline{\Omega}\to\widehat{\C}
$$ is a continuous map that is antiholomorphic on $\Omega$, where $\Omega:=\mathfrak{G}(\widetilde{\Omega})$. By construction, $\sigma$ fixes $\partial\Omega$ pointwise, and hence it is the Schwarz reflection map associated with the quadrature domain $\Omega$. It now follows that the four different constructions described at the end of Section~\ref{topo_schwarz_subsec} give rise to quadrature domains $\Omega$ of connectivity $m\geq 1$ with $n$ double points on $\partial\Omega$ such that $m+n=4$. 

It remains to argue that the uniformizing meromorphic map $f$ of $\Omega$ has degree $d_f=4$. By Proposition~\ref{schwarz_deg_prop}, the integer $d_f$ is equal to the number of $\sigma-$preimages of a point $q\in\mathfrak{G}(\cW_1)\cap\Int{\Omega^\complement}$, counted with multiplicity. Note that $\pmb{R}^{-1}(\cW_1)=\cW_1\sqcup\cW_1'$, where $\cW_1$ maps to itself with degree $2$, and $\cW_1'$ is a prefixed Fatou component of $\pmb{R}$ that maps to $\cW_1$ univalently. Since the map $\widetilde{\sigma}$ agrees with $\pmb{R}$ outside $\cup_{j=1}^4\cW_j$, it follows that $\sigma^{-1}(q)\setminus \mathfrak{G}(\cW_1)$ consists of a unique point $q'\in\mathfrak{G}(\cW_1')$, and $q'$ maps to $q$ with local degree $1$. On the other hand, by the mapping properties of $\mathcal{E}$ and $\cF$, the point $q$ has three preimages (counted with multiplicity) under $\sigma$ in $\mathfrak{G}(\cW_1)$. 
Thus, the map $\sigma:\sigma^{-1}(\Int{\Omega^\complement})\to\Int{\Omega^\complement}$ has degree $4$, and hence $d_f=4$.

\subsection{Non-singular quadrature domains of maximal connectivity $2d_f-4$}\label{non_sing_qd_max_conn_subsec}

Consider a triangulation of the sphere having $s\geq 3$ vertices and $2s-4$ faces. Such a triangulation can be constructed, for instance, by starting with a single triangle and subdividing the faces into further triangles until one ends up with $s$ vertices. The $1$-skeleton of this triangulation is a connected, simple, $2$-connected plane graph. 

By \cite[Theorem~1.1]{LLM22}, there exists a critically fixed anti-rational map $\cR$ of degree $s-1$ such that the planar dual of the \emph{Tischler graph} $\mathcal{T}(\cR)$ of $\cR$ is isomorphic to the $1$-skeleton of the above triangulation (as plane graphs). The Tischler graph of $\cR$ is the union of (the closures of) all fixed internal rays in the invariant Fatou components; in particular, the number of vertices of $\cT(\cR)$ is equal to the number of distinct critical points of $\cR$ (cf. \cite[\S 4.1]{LLM22}). The fact that the chosen triangulation has $2s-4$ faces implies that $\cT(\cR)$ has $2s-4$ vertices; i.e., the anti-rational map $\cR$ has $2s-4$ distinct critical points each of which is fixed by $\cR$. By \cite[Theorem~9.3]{Mil06}, the Fatou components of $\cR$ containing the critical points are simply connected, and the restriction of $\cR$ to each of these invariant Fatou components is conformally conjugate to $\overline{z}^2\vert_{\D}$.

Recall the orientation-reversing double covering $\mathcal{E}$ of the circle. Using the arguments of Sections~\ref{topo_schwarz_subsec} and~\ref{top_to_anal_subsec}, we can now replace the action of $\cR$ on each of its invariant Fatou components with that of the map $\mathcal{E}$. This produces a Schwarz reflection map $\sigma$ defined on a non-singular quadrature domain $\Omega$ of connectivity $2s-4$; indeed, each component of $\Omega^\complement$ corresponds to an invariant Fatou component of $\cR$. Finally, the arguments of the last paragraph of Section~\ref{top_to_anal_subsec} show that in this case, each point in $\Int{\Omega^\complement}$ has $s$ preimages under $\sigma$, counted with multiplicity. It follows that the degree of the uniformizing meromorphic map $f$ of the quadrature domain $\Omega$ is equal to $s$; i.e., $s=d_f$. Thus, $\Omega$ is the desired non-singular quadrature domain of connectivity $2d_f-4$.

\end{document}